\documentclass[12pt]{amsart}
\usepackage{amsmath}

\usepackage{stmaryrd}
\SetSymbolFont{stmry}{bold}{U}{stmry}{m}{n}
\usepackage{amssymb}
\usepackage{longtable}
\usepackage[titletoc]{appendix}
\usepackage{color}
\usepackage{multirow,bigdelim}
\usepackage{stmaryrd}
\usepackage{verbatim}
\usepackage{mathrsfs}
\usepackage{tikz}
\usepackage{tikz-cd}
\usepackage{bm}
\usepackage{mathdots}
\usepackage[pdfencoding=auto, psdextra]{hyperref}

\usepackage[T1]{fontenc}

\allowdisplaybreaks 

\newtheorem{theorem}[equation]{Theorem}
\newtheorem{lemma}[equation]{Lemma}
\newtheorem{proposition}[equation]{Proposition}

\newtheorem{definition}[equation]{Definition}

\newtheorem{example}[equation]{Example}

\theoremstyle{remark}
\newtheorem{remark}[equation]{Remark}

\numberwithin{equation}{subsection}

\allowdisplaybreaks[1]

\newcommand{\inorm}[1]{{\lvert #1 \rvert}_{\infty}}

\newcommand{\bsalpha}{\bm{\alpha}}

\newcommand{\bsg}{\bm{g}}
\newcommand{\bsh}{\bm{h}}

\newcommand{\bsy}{\bm{y}}

\newcommand{\FF}{\mathbb{F}}

\newcommand{\TT}{\mathbb{T}}
\newcommand{\GG}{\mathbb{G}}

\newcommand{\CC}{\mathbb{C}}

\newcommand{\KK}{\mathbb{K}}

\newcommand{\bA}{\mathbf{A}}

\newcommand{\bz}{\mathbf{z}}

\newcommand{\cG}{\mathcal{G}}
\newcommand{\cM}{\mathcal{M}}
\newcommand{\cN}{\mathcal{N}}

\newcommand{\cF}{\mathcal{F}}

\DeclareMathAlphabet{\matheur}{U}{eur}{m}{n}

\newcommand{\rd}{\mathrm{d}}

\newcommand{\Cof}{\mathrm{Cof}}

\newenvironment{psmallmatrix}
{\left(\begin{smallmatrix}}
	{\end{smallmatrix}\right)}

\DeclareMathOperator{\Exp}{Exp}
\DeclareMathOperator{\Lie}{Lie}
\DeclareMathOperator{\Ker}{Ker} \DeclareMathOperator{\GL}{GL}
\DeclareMathOperator{\Mat}{Mat} 
\DeclareMathOperator{\End}{End} 
 
 \DeclareMathOperator{\Res}{Res}
 
\DeclareMathOperator{\Id}{Id} 
\DeclareMathOperator{\trdeg}{tr.deg} 
\DeclareMathOperator{\Ext}{Ext}  
\DeclareMathOperator{\rank}{rank} 
\DeclareMathOperator{\ad}{ad}

\newcommand{\oK}{\overline{K}}

\newcommand{\tr}{\mathrm{tr}}
\newcommand{\wPhi}{\widetilde{\Phi}}
\newcommand{\wPsi}{\widetilde{\Psi}}

\newcommand{\laurent}[2]{{#1 (\!( #2 )\!)}}

\newcommand{\twistop}[2]{\langle #1 \mathbin{|} #2 \rangle}

\begin{document}
\title[]{{\large{O\MakeLowercase{n the third kind periods for abelian $\MakeLowercase{t}$-modules}}}}

\author{Yen-Tsung Chen}
\address{Department of Mathematics, Pennsylvania State University, University Park, PA 16802, U.S.A.}

\email{ytchen.math@gmail.com}

\author{Changningphaabi Namoijam}
\address{Department of Mathematics, Colby College, Waterville, ME 04901, U.S.A.}
\email{cnamoijam@gmail.com}

\thanks{}


\date{\today}

\begin{abstract} 
    Inspired by the relations between periods of elliptic integrals of the third kind and the periods of the extensions of the corresponding elliptic curves by the multiplicative group, we introduce the notion of the third kind periods for abelian $t$-modules and establish an evaluation for these periods that is parallel to the classical setting. When we specialize our result to the case of Drinfeld modules, an explicit formula for these third kind periods is established. We also prove the algebraic independence of periods of the first, the second, and the third kind for Drinfeld modules of arbitrary rank. This generalizes prior results of Chang for rank $2$ Drinfeld modules. 
\end{abstract}


\maketitle


\section{Introduction} \label{S:Intro}
\subsection{Motivation}
    Let $E$ be an elliptic curve defined over a number field $L$ with the defining equation in Weierstrass form given by
    \[
        Y^2=4X^3-g_2X-g_3,
    \]
    where $g_2,g_3\in L$. By the analytic uniformization theorem, let $\Lambda\subset\mathbb{C}$ denote the period lattice of $E$. Let $\wp_{\Lambda}(z)$ and $\zeta_{\Lambda}(z)$ be the associated Weierstrass $\wp$-function and the Weierstrass zeta function of $\Lambda$ respectively. Then, the quasi-period map $\eta:\Lambda\to\mathbb{C}$ is uniquely determined by $\eta(w):=\zeta_\Lambda(z+w)-\zeta_\Lambda(z)$ for $w\in\Lambda$. Given any meromorphic differential $\xi$ on $E$ defined over $L$, it is known that we have the following decomposition
    \[
        \xi=\frac{1}{2}\sum_{j=1}^nc_j\frac{y+\wp'_\Lambda(u_j)}{x-\wp_{\Lambda}(u_j)}\frac{dx}{y}+a\frac{dx}{y}+bx\frac{dx}{y}+d\chi
    \]
    where $\chi$ is a rational function on $E$, $a,b,c_j,d$ are algebraic numbers, and $\wp_\Lambda(u_j)\in\overline{\mathbb{Q}}$ for some $u_j\in\mathbb{C}$. If $\gamma$ is a closed path on $E(\mathbb{C})$ so that $\xi$ is holomorphic along $\gamma$, then to calculate the abelian integral $\int_{\gamma}\xi$, it suffices to deal with the following components. The first one is \emph{the elliptic integral of the first kind} $w=\int_{\gamma}dx/y\in\Lambda$ which gives a period of the elliptic curve $E$. The second one is \emph{the elliptic integral of the second kind} $\eta(w)=\int_{\gamma}xdx/y$ that is related to the period of extensions of $E$ by the additive group $\mathbb{G}_a$. Finally, the last part is \emph{the elliptic integral of the third kind} $\int_{\gamma}\frac{1}{2}\frac{y+\wp'_{\Lambda}(u_j)}{x-\wp_{\Lambda}(u_j)}\frac{dx}{y}$. It was shown by Serre \cite[Appendix~II]{Wal79} (see also \cite[\S2]{Ber08}) that elliptic integrals of the third kind occur in the period of extensions of the elliptic curve $E$ by the multiplicative group $\mathbb{G}_m$. In fact, if we set $\lambda(w,u_j):=w\zeta_{\Lambda}(u_j)-\eta(w)u_j$, then we have
    \begin{equation}\label{Eq:Third_Kind_Periods}
        \frac{1}{2}\int_{\gamma}\frac{y+\wp'_{\Lambda}(u_j)}{x-\wp_{\Lambda}(u_j)}\frac{dx}{y}=\lambda(w,u_j)+2m\pi\sqrt{-1}
    \end{equation}
    for some integer $m$. Wüstholz \cite{Wüs84} proved that the abelian integral $\int_\gamma\xi$ is either zero or transcendental. This extended prior results of Laurent who established the same result under certain restrictions \cite{Lau80,Lau82}.  
    It was established by Wüstholz \cite{Wüs84} that if $w,u_1,\dots,u_n$ are linearly independent over $\mathbb{Q}$, then
    \[
        \dim_{\overline{\mathbb{Q}}}\mathrm{Span}_{\overline{\mathbb{Q}}}\{1,w,\eta(w),\lambda(w,u_1),\dots,\lambda(w,u_n)\}=3+n.
    \]
    However, the algebraic independence of the elliptic integrals of the first, the second, and the third kind still remains mysterious. For the relevant reference about the transcendence results, we refer readers to \cite[\S6.2]{BW07} for more details.

    The main purpose of the present article is to seek an analogue of \eqref{Eq:Third_Kind_Periods} for the third kind periods of uniformizable abelian $t$-modules. In particular, when we specialize our results to Drinfeld modules, we establish an explicit formula generalizing the rank $2$ result of Chang \cite[Thm.~2.4.4]{Chang13} to arbitrary rank. Furthermore, we provide an affirmative answer to the algebraic independence of the first, the second, as well as the third kind periods of Drinfeld modules.

\subsection{Main results}
    Let $\mathbb{F}_q$ be a finite field of $q$ elements. We set $A=\mathbb{F}_q[\theta]$ to be the polynomial ring in variable $\theta$ over $\mathbb{F}_q$ and $K$ to be the field of fractions of $A$. Let $|\cdot|_\infty$ be the normalized non-Archimedean norm on $K$ so that $|f/g|_\infty:=q^{\deg(f)-\deg(g)}$, where $f,g\in A$ and are coprime to each other. Let $K_\infty$ be the completion of $K$ with respect to $|\cdot|_\infty$. We identify $K_\infty$ with the Laurent series field $\laurent{\FF_q}{1/\theta}$. We set $\mathbb{C}_\infty$ to be the completion of a fixed algebraic closure $\overline{K}_\infty$ and $\overline{K}$ to be the algebraic closure of $K$ inside $\mathbb{C}_\infty$. Finally, we set $\bA=\mathbb{F}_q[t]$ and $\mathbf{K}=\mathbb{F}_q(t)$ which are the isomorphic copies of $A$ and $K$ that serve as the parameters of the operators. For a field $\KK\subset\mathbb{C}_\infty$, we consider the twisted polynomial ring $\KK[\tau]$ subject to the following relation $\alpha^q\tau=\tau\alpha$.

    Let $\Lambda\subset\mathbb{C}_\infty$ be a discrete free $\bA$-module of rank $r>0$ whose $\bA$-module structure is given by $a(t)\cdot\lambda:=a(\theta)\lambda$ for $a\in\bA$ and $\lambda\in\Lambda$. Consider the lattice function
    \[
        \Exp_\Lambda(z):=z\prod_{0\neq\lambda\in\Lambda}\left(1-\frac{z}{\lambda}\right)\in\mathbb{C}_\infty\llbracket z\rrbracket.
    \]
    It induces an entire, surjective, and $\mathbb{F}_q$-linar analytic function $\Exp_\Lambda(\cdot):\mathbb{C}_\infty\to\mathbb{C}_\infty$. Then we have the following analytic uniformization
    \[
        0\to\Lambda\hookrightarrow\mathbb{C}_\infty\overset{\Exp_\Lambda(\cdot)}{\twoheadrightarrow}\mathbb{C}_\infty\to 0.
    \]
    It induces a non-scalar $\bA$-module structure, namely an $\mathbb{F}_q$-algebra homomorphism $\rho^\Lambda:\bA\to\mathbb{C}_\infty[\tau]$ on the additive group $\mathbb{C}_\infty$ such that $\rho^\Lambda_a(\Exp_\Lambda(z))=\Exp_\Lambda(a(\theta)z)$ for any $a\in\bA$ and $z\in\mathbb{C}_\infty$, where $\rho^\Lambda_a$ is the image of $a$ under $\rho^\Lambda$. We call the pair $E=(\mathbb{G}_a,\rho^\Lambda)$ a Drinfeld module of rank $r$. The theory of Drinfeld modules was introduced by Drinfeld in \cite{Dri76} under the name \emph{elliptic modules}, while the case of $r=1$ was initiated by Carlitz in \cite{Car35}. By the similarity of the analytic uniformization, one can regard Drinfeld modules of rank $r\geq 2$ as a function field analogue of complex elliptic curves. The case of $r=1$ with a suitable normalization, called Carlitz module $\mathbf{C}=(\mathbb{G}_a,[\cdot])$, can be viewed as the counterpart of the multiplicative group $\mathbb{G}_m$. The correspondence between the Carlitz module $\mathbf{C}$ and the multiplicative group $\mathbb{G}_m$ plays a crucial role in this paper.
    
    Due to the lack of a direct analogue of integration over function fields in positive characteristics, we do not have an immediate correspondence of the elliptic integrals. Nevertheless, by the theory of biderivations and quasi-periodic extensions for Drinfeld modules developed by Anderson, Deligne, Gekeler, and Yu (see \cite{Gek89,Yu90}), we have a parallel theory of periods and quasi-periods for Drinfeld modules that can be regarded as counterparts of the elliptic integrals of the first and the second kind respectively. More precisely, fix a Drinfeld module $E=(\mathbb{G}_a,\rho^\Lambda)$ of rank $r$. We call an $\mathbb{F}_q$-linear map $\delta:\bA\to\mathbb{C}_\infty[\tau]$ a $\rho^\Lambda$-biderivation (will be called a $(\rho^\Lambda,[\cdot]_0)$-biderivation later for consistency of notation) if for any $a,b\in\bA$ we have
    \begin{equation}\label{Eq:Intro_Derivation}
        \delta(ab)=a(\theta)\delta(b)+\delta(a)\rho^\Lambda_b.
    \end{equation}
    Note that each $\rho^\Lambda$-biderivation $\delta$ is uniquely determined by the image of $t$. If $\delta(t)\in\mathbb{C}_\infty[\tau]\tau$, then there is a unique $\mathbb{F}_q$-linear power series $\mathrm{F}_\delta(z)\in z^q\mathbb{C}_\infty\llbracket z\rrbracket$ that induces an entire function on $\mathbb{C}_\infty$ so that $\mathrm{F}_\delta(\theta z)=\theta\mathrm{F}_\delta(z)+\delta(t)\big(\Exp_\Lambda(z)\big)$. The function $\mathrm{F}_\delta$ is called the \emph{quasi-periodic function} associated to $\delta$. If we set $\delta_0$ to be the $\rho$-biderivation with $\delta_0(t)=\rho^\Lambda_t-t\in\mathbb{C}_\infty[\tau]\tau$, then for any $\lambda\in\Lambda$ we have that $\mathrm{F}_{\delta_0}(\lambda)=-\lambda$ recovers the period of $E$ (cf. elliptic integrals of the first kind). By an abuse of notation, for each $1\leq i\leq r-1$ we denote by $\tau^i$ the $\rho^\Lambda$-biderivation with $\tau^i(t)=\tau^i\in\mathbb{C}_\infty[\tau]\tau$. Then, for any $\lambda\in\Lambda$, $\mathrm{F}_{\tau^i}(\lambda)$ comprise the quasi-periods of $E$ (cf. elliptic integrals of the second kind). Note that the transcendence results on periods and quasi-periods for Drinfeld modules were initiated by Yu \cite{Yu86,Yu90} and their algebraic relations were determined completely by Chang and Papanikolas \cite{CP12}.

    After the work of Anderson \cite{And86} which introduced the higher dimensional generalization of Drinfeld modules, now called $t$-modules, the theory of biderivations and quasi-periods was developed further for general $t$-modules by Brownawell and Papanikolas \cite{BrPa02} while investigating geometric $\Gamma$-values (see also \cite[\S4]{NPapanikolas21} and \cite[\S2.5.7]{HJ20}). 
    Roughly speaking, a $d$-dimensional $t$-module is a pair $G=(\mathbb{G}_a^d,\varphi)$ where $\varphi:\bA\to\Mat_d(\mathbb{C}_\infty[\tau])$ is an $\mathbb{F}_q$-algebra homomorphism satisfying some extra properties. In particular, Drinfeld modules are $1$-dimensional $t$-modules. Due to Anderson \cite{And86}, each $t$-module $G$ has the exponential map $\Exp_\varphi:\Lie(G)(\mathbb{C}_\infty)\to G(\mathbb{C}_\infty)$ that serves as a replacement of the lattice function in the higher dimensional setting (see \S2.2 for precise details). Unlike in the case of Drinfeld modules, $\Exp_\varphi$ is not always surjective. If it is surjective, then we call the $t$-module $G$ \emph{uniformizable}. We call  $\Lambda_G:=\Ker(\Exp_\varphi)$ the period lattice of $G$ and each $\lambda\in\Lambda_G$ a period of $G$. An $\mathbb{F}_q$-linear map $\delta:\bA\to\Mat_{1\times d}(\mathbb{C}_\infty[\tau])$ is called a $\varphi$-biderivation (will be called a $(\varphi,[\cdot]_0)$-biderivation later for consistency of notation) if for any $a,b\in\bA$ we have
    \begin{equation}\label{Eq:Intro_Derivation_t}
        \delta(ab)=a(\theta)\delta(b)+\delta(a)\varphi_b
    \end{equation}
    Similar to the case of Drinfeld modules, if $\delta(t)\in\Mat_{1\times d}(\mathbb{C}_\infty[\tau]\tau)$, then there is a unique $\mathbb{F}_q$-linear, entire, quasi-periodic function $\mathrm{F}_\delta:\Mat_{d\times 1}(\mathbb{C}_\infty)\to\mathbb{C}_\infty$ with various appropriate properties (see \S2.3 for more details). We denote by $\mathrm{Der}(\varphi,[\cdot]_0)$ the collection of all $\varphi$-biderivations. For any $\lambda\in\Lambda_G$ and $\delta\in\mathrm{Der}(\varphi,[\cdot]_0)$, we call $\mathrm{F}_\delta(\lambda)$ a \emph{quasi-period} of $G$. More generally, for each $\bm{y}\in G(\mathbb{C}_\infty)$ we call $\mathrm{F}_\delta(\bm{y})$ a \emph{quasi-logarithm} at $\bm{y}$ for $G$.
    
    To seek an analogue of the elliptic integrals of the third kind for Drinfeld modules or more generally, the abelian integrals of the third kind for uniformizable abelian $t$-modules, we follow the idea of Chang \cite{Chang13} that was inspired by the classical geometric viewpoint. In the classical setting, it was explained in \cite[\S 9]{Ber83} and \cite[\S 2.3]{BPSS22} that the abelian integrals of the third kind on an abelian variety $\mathcal{A}$ is closely related to extensions of $\mathcal{A}$ by the torus $\mathbb{G}_m^s$ with $s\geq 1$. With the correspondence between $\mathbb{G}_m$ and the Carlitz module $\mathbf{C}$ in mind, it is natural to investigate period vectors of $t$-module extensions of a uniformizable abelian $t$-module $G=(\mathbb{G}_a^d,\varphi)$ by the Carlitz module $\mathbf{C}$ and its higher dimensional variants. To be more precise, assume that $G=(\mathbb{G}_a^d,\varphi)$ is defined over $\overline{K}$ in the sense that $\varphi_t\in\Mat_d(\oK[\tau])$, and $\mathbf{C}^{\otimes n}=(\mathbb{G}_a^n,[\cdot]_n)$ is the $n$-th tensor power of the Carlitz module (see \S2.2 for more details). It was shown in \cite[Lem.~2.1]{PR03} that the $t$-module extensions of $G$ by $\mathbf{C}^{\otimes n}$ are parameterized by $(\varphi,[\cdot]_n)$-biderivations, that are $\mathbb{F}_q$-linear maps $\delta:\bA\to\Mat_{n\times d}(\mathbb{C}_\infty[\tau])$ satisfying
    \begin{equation}\label{Eq:Intro_Derivation_G_n}
        \delta(ab)=[a]_n\delta(b)+\delta(a)\varphi_b.
    \end{equation}
    Fix a $(\varphi,[\cdot]_n)$-biderivation $\delta$ with $\delta(t)\in\Mat_{n\times d}(\oK[\tau]\tau)$. Consider the corresponding $t$-module extension $G_\delta=(\mathbb{G}_a^{n+d},\phi)$ defined precisely in \eqref{Eq:t-module_from_delta}. The period lattice of $G_\delta$ is of the form (see \eqref{Eq:Exp_of_Extension} or \cite[p.422]{PR03})
    \[
        \Lambda_{G_\delta}=\bA\begin{pmatrix}
            \bm{\omega}_1\\
            \bm{\lambda}_1
        \end{pmatrix}+\cdots+\bA\begin{pmatrix}
            \bm{\omega}_r\\
            \bm{\lambda}_r
        \end{pmatrix}+\bA\begin{pmatrix}
            \bm{0}\\
            \bm{\gamma}_n
        \end{pmatrix}\subset\Mat_{(d+n)\times 1}(\mathbb{C}_\infty),
    \]
    where $\bm{\omega}_1,\dots,\bm{\omega}_r\in\Mat_{d\times 1}(\mathbb{C}_\infty)$ generates the period lattice of $G$ over $\bA$ and $\bm{\gamma}_n\in\Mat_{n\times 1}(\mathbb{C}_\infty)$ is the generator of the period lattice of $\mathbf{C}^{\otimes n}$ whose last entry is given by $\Tilde{\pi}^n$ where 
    \[
        \Tilde{\pi}:=-(-\theta)^{q/(q-1)}\prod_{i=1}^\infty\left(1-\theta^{1-q^i}\right)^{-1}\in\mathbb{C}_\infty^\times.
    \]
    Here, $(-\theta)^{1/(q-1)}$ is a fixed choice of the $(q-1)$-st root of $-\theta$.
    
    We are interested in the last $n$ entries of any period vector in $\Lambda_{G_\delta}$. Especially in the case of $n=1$, we call the last entry of any period vector in $\Lambda_{G_\delta}$ \emph{the third kind period} of $G$. Our first main result concerns the formula of $\bm{\lambda}_j$ in the same spirit of \eqref{Eq:Third_Kind_Periods}. More precisely, let $H=(\mathbb{G}_a^s,\varrho)$ be the explicitly constructed $t$-module coming from the tensor product of the dual $t$-motive of $\mathbf{C}^{\otimes n}$ with the exterior powers of the dual $t$-motive of $G$ (see Remark~\ref{Rem:Subtle_Points}). 
    Then, our first result, restated as Theorem~\ref{Thm:Generates_Trackable_Coordinate} later, asserts that under some mild technical conditions, the last entry of $\bm{\lambda}_j$ can be generated by periods and quasi-periods of $G$ together with logarithm and quasi-logarithm of the $t$-module $H$ at a specific algebraic point. 
    
    \begin{theorem}\label{Thm:Intro_Thm1}
        Assume that $H$ is almost strictly pure and $\delta\in\mathrm{Der}_{\mathrm{sp}}(\varphi,[\cdot]_n)$ is a special $(\varphi,[\cdot]_n)$-biderivation (see Definition~\ref{Def:Special_Extensions}) with $\delta(t)\in\Mat_{n\times d}(\oK[\tau]\tau)$. If we express $\bm{\lambda}_j=(\lambda_{j1},\dots,\lambda_{jn})^\tr\in\Mat_{n\times 1}(\mathbb{C}_\infty)$, then there exists $\bm{y}\in(\Lie H)(\mathbb{C}_\infty)$ with $\Exp_{\varrho}(\bm{y})\in H(\overline{K})$ so that if we set
        \[
            S:=\{1,\mathrm{F}_{\bm{\epsilon}}(\bm{y})\mid\bm{\epsilon}\in\mathrm{Der}_\partial(\varrho,[\cdot]_0),~\bm{\epsilon}(t)\in \Mat_{1\times s}(\overline{K}[\tau])\tau\}
        \]
        to be the collection of $1$ and all quasi-logarithms of $\bm{y}$ for $H$ as well as
        \[
            T:=\{\mathrm{F}_{\bm{\varpi}}(\bm{\omega})\mid\bm{\varpi}\in\mathrm{Der}_\partial(\varphi,[\cdot]_0),~\bm{\varpi}(t)\in \Mat_{1\times d}(\oK[\tau])\tau,~\bm{\omega}\in\Lambda_G\}
        \]
        to be the collection of all periods and quasi-periods of $G$, then we have
        \[
            \lambda_{jn}\in\mathrm{Span}_{\overline{K}}\bigg(\{\Tilde{\pi}^n\}\cup\{xy\mid x\in S,~y\in T\}\bigg).
        \]
    \end{theorem}

    When we specialize to the case of Drinfeld modules, Theorem~\ref{Thm:Intro_Thm1} can be simplified and it becomes more concrete. For instance, in the case of $G=E=(\mathbb{G}_a,\rho)$ a Drinfeld module, the $t$-module $H$ in Theorem~\ref{Thm:Intro_Thm1} is always almost strictly pure and it is isomorphic to the $t$-module $E_{n-1}:=\mathbf{C}^{\otimes n-1}\otimes\wedge^{r-1}E:=(\mathbb{G}_a^{rn-1},\psi_{n-1})$ constructed from their $t$-motives (see Example~\ref{Ex:Exterior_Powers}). Furthermore, any $(\rho,[\cdot]_n)$-biderivation $\delta$ is automatically special as long as $\delta(t)\in\oK[\tau]\tau$ (see Proposition~\ref{Lem:Special_Extensions_Drinfeld_Modules}).
    Our second main result, which is the assembly of Theorem~\ref{Thm:Trackable_Coordinate_Drinfeld_Modules},~Theorem~\ref{Thm:Third_Kind_Periods_Formula}, and Theorem~\ref{Thm:Algebraic_Independence}, asserts an explicit analogue of \eqref{Eq:Third_Kind_Periods} for Drinfeld modules and gives the algebraic independence result for periods of the first, the second, and the third kind of Drinfeld modules.
    \begin{theorem}\label{Thm:Intro_Thm2}
        Let $E=(\mathbb{G}_a,\rho)$ be a Drinfeld module of rank $r\geq 2$ defined over $\overline{K}$ with the period lattice $\Lambda_E=\bA w_1+\cdots+\bA w_r$ and the endomorphism ring
        \[
            \End(E):=\{u\in\mathbb{C}_\infty[\tau]\mid u\rho_t=\rho_t u\}
        \]
        is free of finite rank $\bm{s}\geq 1$ over $\bA$. Let $n\geq 1$ be a positive integer. Assume that $\delta$ is a $(\rho,[\cdot]_n)$-biderivation with $\delta(t)\in\Mat_{n\times 1}(\oK[\tau]\tau)$. Then, the following assertions hold.
        \begin{enumerate}
            \item If we express $\bm{\lambda}_j=(\lambda_{j1},\dots,\lambda_{jn})^\tr\in\Mat_{n\times 1}(\mathbb{C}_\infty)$, then there exists $\bm{y}\in(\Lie E_{n-1})(\mathbb{C}_\infty)$ with $\Exp_{\psi_{n-1}}(\bm{y})\in E_{n-1}(\overline{K})$ so that if we set $S$ to be the union of $1$ and all quasi-logarithms of $\bm{y}$ for $E_{n-1}$, and $T$ to be the collection of all periods and quasi-periods of $E$, then we have
            \[
                \lambda_{jn}\in\mathrm{Span}_{\overline{K}}\bigg(\{\Tilde{\pi}^n\}\cup\{xy\mid x\in S,~y\in T\}\bigg).
            \]
            In particular, $\lambda_{jn}$ is either zero or transcendental over $\oK$.
            \item If $n=1$ and $\delta(t)\in\oK[\tau]\tau$ with $\deg_\tau\delta(t)\leq r-1$, then there exists an explicitly constructed algebraic point $\bm{\alpha}_\delta\in E_0(\overline{K})$ (see Lemma~\ref{Lem:Special_Algebraic_Point}), an explicitly determined constant $\mathfrak{c}_E\in\overline{K}^\times$ (see \eqref{Eq:Constant_in_Formula}), and some $a_j\in A$ such that
            \[
                \lambda_{j1}=-\mathfrak{c}_E^{-1}\left[\sum_{\ell=2}^r y_{r-\ell+1} \mathrm{F}_{\tau^{\ell-1}}(\omega_j) + (-1)^{r-1}\mathrm{F}_{\bm{\epsilon}}(\bsy_\delta)\omega_j\right]+a_j\widetilde{\pi},
            \]
            where $\bm{y}_\delta=(y_1,\dots,y_{r-1})^\tr\in(\Lie E_0)(\mathbb{C}_\infty)$ is chosen so that $\Exp_{\psi_0}(\bm{y}_\delta)=\bm{\alpha}_\delta$, $\mathrm{F}_{\tau^{\ell-1}}(\omega_j)$ are quasi-periods of $E$ associated to $w_j$ for all $2\leq\ell\leq r$, and $\mathrm{F}_{\bm{\epsilon}}(\bsy_\delta)$ is the quasi-logarithm of $\bm{y}_\delta$ with $\bm{\epsilon}$ a $\psi_0$-biderivation uniquely determined by $\bm{\epsilon}(t)=(\tau,0,\dots,0)\in\Mat_{1\times (r-1)}(\oK[\tau])\tau$.
            \item Assume that $n=1$ and $\delta_1,\dots,\delta_m$ are $(\rho,[\cdot]_1)$-biderivations with $\delta_i(t)\in\oK[\tau]\tau$ and $\deg_\tau\delta_i(t)\leq r-1$ for each $1\leq i\leq m$. We denote by $\lambda_{j1}^{[i]}$ the third kind period associated to the $t$-module $E_{\delta_i}$. Let $\bm{\alpha}_{\delta_i}\in E_0(\oK)$ be chosen as in $\mathrm{(2)}$. If we denote by $\End(E_0)$ the endomorphism ring of $E_0$ (see \eqref{Eq:Endomorphism_Ring}) and
            \[
                \rank_{\End(E_0)}\mathrm{Span}_{\End(E_0)}\big(\bsalpha_{\delta_1},\dots,\bsalpha_{\delta_m}\big)=m,
            \]
            then we have
            \[
                \trdeg_{\oK}\overline{K}\bigg(w_j,\mathrm{F}_{\tau^{\ell-1}}(\omega_j),\lambda_{j1}^{[i]}\mid 2\leq\ell\leq r,~1\leq j\leq r,~1\leq i\leq m\bigg)=\frac{r^2}{\bm{s}}+rm.
            \]
        \end{enumerate}
    \end{theorem}

    We mention that in the case of rank $2$ Drinfeld modules, our formula \ref{Thm:Intro_Thm2}(2) for the third kind periods matches with the result in \cite[Thm.~2.4.4]{Chang13}. Moreover, the algebraic independence result Theorem~\ref{Thm:Intro_Thm2}(3) gives a positive answer to an analogue of a special case ($1$ elliptic curve defined over $\overline{\mathbb{Q}}$ with $m$ extensions of the given elliptic curve by $\mathbb{G}_m$) of the $1$-motivic elliptic conjecture stated in \cite[\S4]{Ber20}. Note that the case of $m=0$ in Theorem~\ref{Thm:Intro_Thm2}(3) was established in \cite{CP12}. Our algebraic independence result is built on their theorem, our formula \ref{Thm:Intro_Thm2}(2), and the recent advances given in \cite{GN24}.

\subsection{Strategy and organization}
    Due to the work of Anderson \cite{And86}, we know that the information of the period lattice of a uniformizable abelian $t$-module can be encoded in the rigid analytic trivialization of the $t$-motive attached to the $t$-module in question. Thus, our strategy for proving Theorem~\ref{Thm:Intro_Thm1} is producing relations involving entries of the rigid analytic trivialization.  Let $\Delta\in\Mat_{r}(\overline{K}[t])\cap\GL_r(\overline{K}(t))$. Then, the Frobenius difference equation defined by $\Delta$ is given by
    \begin{equation}\label{Eq:Difference_Equation}
        \mathbf{X}^{(-1)}=\Delta\mathbf{X},
    \end{equation}
    where $\mathbf{X}=(X_{ij})$ is the square matrix of size $r$ with indeterminates $X_{ij}$ in the Tate algebra $\TT\subset\mathbb{C}_\infty\llbracket t\rrbracket$ on the closed unit disc, and the twisting operation $(\cdot)^{(-1)}$ is defined in \eqref{Eq:Twisting}. We set
    \[
        \mathrm{Sol}(\Delta):=\{\mathbf{X}\in\GL_{r}(\TT)\mid\mathbf{X}=(X_{ij})~\mathrm{satisfies}~\eqref{Eq:Difference_Equation}\}.
    \]
    
    The key ingredient of our strategy is the comparison of different constructions of the solutions of the same Frobenius difference equation. In the present paper, we adopt some constructions in \cite{NPapanikolas21} and develop new techniques so that we can extend the original strategy of Chang \cite{Chang13} on rank 2 Drinfeld modules to a wide class of $t$-modules of higher ranks and dimensions. More precisely, let $G=(\mathbb{G}_a^d,\varphi)$ be a $d$-dimensional uniformizable abelian $t$-module and $\mathbf{C}^{\otimes n}=(\mathbb{G}_a^n,[\cdot]_n)$ be the $n$-th tensor power of the Carlitz module with $n\geq d$. Given a $(\varphi,[\cdot]_n)$-biderivation $\delta$, its corresponding $t$-module extension $G_\delta=(\mathbb{G}_a^{n+d},\phi)$ has a rigid analytic trivialization $\Upsilon_\phi$ given in \eqref{Eq:Rigid_Analytic_Trivialization_Ext}. It follows from construction that
    \[
        \mathfrak{u}_{\widetilde{c}}\Omega^d\Upsilon_\phi\in\mathrm{Sol}(\wPhi_\phi^{\ad})
    \]
    for some $\mathfrak{u}_{\widetilde{c}}\in\mathbb{C}_\infty^\times$, and $\Omega$ as well as $\wPhi_\phi^{\ad}$ are defined in \S\ref{SS:Frobsdiffs}. One the one hand, we use the relations among $t$-motives, dual of the $t$-motives, and dual $t$-motives to construct the injection (see Lemma~\ref{Lem:Reduction})
    \[
        \Theta_1:\mathrm{Sol}(\wPhi_\phi^{\ad})\hookrightarrow\mathrm{Sol}\big(\mathrm{Cof}(\Phi_\phi)\big).
    \]
    On the other hand, inspired by Anderson's exponentiation theorem (see \cite[Thm.~2.5.21]{HJ20} or \cite[Thm.~3.4.2]{NPapanikolas21}), there exists an $\bm{h}\in\Mat_{1\times r}(\oK[t])$, depending on $\delta$, so that we can produce solutions of the Frobenius difference equation defined by $\mathrm{Cof}(\Phi_\phi)$ from the dual $t$-motive side (see Lemma~\ref{Lem:Solution_Difference_Equation})
    \[
        \Theta_2:\{\bm{g}\in\Mat_{1\times r}(\TT)\mid\bm{g}^{(-1)}(t-\theta)^{n-d}\mathrm{Cof}(\Phi_G)-\bm{g}=\bm{h}\}\hookrightarrow\mathrm{Sol}\big(\mathrm{Cof}(\Phi_\phi)\big).
    \]
    Under the assumption that $\delta$ is a special $(\varphi,[\cdot]_n)$-biderivation with $\delta(t)\in\Mat_{n\times d}(\oK[\tau]\tau)$, we deduce that
    \[
        \Theta_2\bigg(\{\bm{g}\in\Mat_{1\times r}(\TT)\mid\bm{g}^{(-1)}(t-\theta)^{n-d}\mathrm{Cof}(\Phi_G)-\bm{g}=\bm{h}\}\bigg)\subset\Theta_1\big(\mathrm{Sol}(\wPhi_\phi^{\ad})\big).
    \]
     Given $\mathbf{X}_1,\mathbf{X}_2\in\mathrm{Sol}(\mathrm{Cof}(\Phi_\phi))$, there is $C\in\GL_r(\bA)$ so that $\mathbf{X}_2=\mathbf{X}_1C$.
    This is the main approach we use for producing relations involving entries of $\Upsilon_\phi$. The details of the above strategy will be presented in the whole of Section~\ref{S:DUalityNAGF}.

    In Section~\ref{S:PeriodsofExts}, we perform a detailed analysis to the case of $G=E=(\mathbb{G}_a,\rho)$ a Drinfeld module of rank $r\geq 2$. More precisely, we prove in Section~\ref{SS:PeriodsofExts1} that for any $n\geq 1$ and $(\rho,[\cdot]_n)$-biderivation $\delta$, it is automatically special as long as $\delta(t)\in\oK[\tau]\tau$. Then, Theorem~\ref{Thm:Intro_Thm1} specializes to Theorem~\ref{Thm:Intro_Thm2}(1) (see Theorem~\ref{Thm:Trackable_Coordinate_Drinfeld_Modules}). In Section~\ref{SS:PeriodsofExts2}, we consider the case $n=1$ and $(\rho,[\cdot]_1)$-biderivations $\delta$ with $\delta(t)\in\oK[\tau]\tau$ with $\deg_\tau\delta(t)\leq r-1$. 
    With the constructions developed in \cite{NPapanikolas21} in hand, we can use Anderson generating functions to find $\Theta_2\big(\{\bm{g}\in\Mat_{1\times r}(\TT)\mid\bm{g}^{(-1)}(t-\theta)^{n-1}\mathrm{Cof}(\Phi_E)-\bm{g}=\bm{h}\}\big)$ explicitly. This allows us to establish Theorem~\ref{Thm:Intro_Thm2}(2) (see Theorem~\ref{Thm:Third_Kind_Periods_Formula}). Finally, for the transcendence result Theorem~\ref{Thm:Intro_Thm2}(3) (see Theorem~\ref{Thm:Algebraic_Independence}), we mention that the algebraic independence among periods and quasi-periods of a Drinfeld module defined over $\oK$ is completely determined by Chang and Papanikolas in \cite{CP12} using the theory of $t$-motivic Galois group developed in \cite{Pap08}. Thus, the new innovation of our result is to show the algebraic independence of the third kind periods for the Drinfeld module in question over the field generated by its periods and quasi-periods over $\oK$. This is achieved by adopting our formula for the third kind periods and the recent advances from \cite{GN24}.

\subsection*{Acknowledgments}
The authors thank Chieh-Yu Chang for helpful comments and suggestions. The first author was partially supported by the AMS-Simons Travel Grants and the Department of Mathematics at Penn State University. The second author was partially supported by a Colby College Research Grant. 

\section{Preliminaries} \label{S:Prelim}
\subsection{Notation}
\begin{longtable}{p{2.25truecm}@{\hspace{5pt}$=$\hspace{5pt}}p{11truecm}}
$\FF_q$ & finite field with $q$ elements, where $q$ is a positive power of a prime $p$. \\
$A$ & $\FF_q[\theta]$, the polynomial ring in $\theta$ over $\FF_q$. \\
$K$ & $\FF_q(\theta)$, the fraction field of $A$. \\
$K_\infty$ & $\laurent{\FF_q}{1/\theta}$, the completion of $K$ with respect to $\inorm{\,\cdot\,}$. \\
$\CC_{\infty}$ & the completion of an algebraic closure of $K_\infty$. \\
$\oK$ & the algebraic closure of $K$ inside $\CC_{\infty}$.  \\
$\bA$ & $\FF_q[t]$, the polynomial ring in $t$ over $\FF_q$, $t$ independent from $\theta$. \\
$\mathbf{K}$ & $\FF_q(t)$, the fraction field of $\bA$. \\
$\TT$ &  the Tate algebra of the closed unit disk of $\CC_{\infty}$.\\
$\|\cdot \|$ & the Gauss norm on $\TT$ defined by $\|\sum_{i\geq 0}a_it^i\|:=\sup_{i\geq 0}\{|a_i|_\infty\}$ with $\sum_{i\geq 0}a_it^i\in\TT$.
\end{longtable}

For $n\in\mathbb{Z}$, we define the $n$-fold Frobenius twisting on the Laurent series field $\mathbb{C}_\infty(\!(t)\!)$ by setting
\begin{align}\label{Eq:Twisting}
    \begin{split}
        \mathbb{C}_\infty(\!(t)\!)&\to\mathbb{C}_\infty(\!(t)\!)\\
    f=\sum c_it^i&\mapsto f^{(n)}:=\sum c_i^{q^n}t^i.
    \end{split}
\end{align}
For $b = \sum b_i\tau^i \in \CC_{\infty}[\tau]$, we define $b^* := \sum b_i^{(-i)}\sigma^i \in \CC_{\infty}[\sigma]$.  If $B = (b_{ij}) \in \Mat_{e_1\times e_2}(\CC_{\infty}[\tau])$, then we set $B^* := (b_{ji}^*) \in \Mat_{e_2 \times e_1}(\CC_{\infty}[\sigma])$. For a square matrix $C\in\Mat_{e}(\mathbb{C}_\infty)$, we denote by $\mathrm{Cof}(C)$ the cofactor matrix of $C$. We further set $C^{\mathrm{ad}}=\mathrm{Cof}(C)^\tr$ to be the adjugate matrix of $C$. Note that if $C$ is invertible, then we have the relation
\[
    C^{-1}=\frac{1}{\det(C)}C^{\mathrm{ad}}=\frac{1}{\det(C)}\mathrm{Cof}(C)^\tr.
\]
In particular, $\big(C^{-1}\big)^{\tr}=\det(C)^{-1}\mathrm{Cof}(C)$ defines an operator preserving the ordering of matrix multiplication in the sense that for invertible matrices $C_1,C_2$ of the same size, we have
\begin{equation}\label{Eq:Dagger_Operator}
    \bigg((C_1C_2)^{-1}\bigg)^{\tr}=\big(C_1^{-1}\big)^\tr\big(C_2^{-1}\big)^\tr=\frac{1}{\det(C_1)\det(C_2)}\mathrm{Cof}(C_1)\mathrm{Cof}(C_2).
\end{equation}

Finally, by following the notation in \cite{NPapanikolas21}, for $B=\sum_{i=0}^N B_i\tau^i\in\Mat_{e_1\times e_2}\mathbb{C}_\infty[\tau]$ and $F\in\Mat_{e_2\times e_3}\big(\mathbb{C}_\infty(\!(t)\!)\big)$, we set
\begin{equation}\label{Eq:Inner_Product}
    \langle B\mid F\rangle:=\sum_{i=0}^NB_iF^{(i)}\in\Mat_{e_1\times e_3}\big(\mathbb{C}_\infty(\!(t)\!)\big).
\end{equation}
We also set $\rd B :=B_0$.

\subsection{\texorpdfstring{$t$}{t}-motives and dual \texorpdfstring{$t$}{t}-motives}\label{SS:tmanddtm}
    In this subsection, we follow 
    \cite{NPapanikolas21} closely to review the notions of $t$-motives and dual $t$-motives. For a field $\mathbb{K} \subset \CC_\infty$ an \emph{Anderson $t$-module} of dimension~$d$ defined over $\mathbb{K}$ is a pair $G=(\mathbb{G}_a^d,\varphi)$, where $\mathbb{G}_a^d$ is the underlying space and $\varphi$ is an $\FF_q$-algebra homomorphism 
    \begin{align*}
        \varphi: \bA &\to \Mat_d(\mathbb{K}[\tau])\\
        a&\mapsto \varphi_a
    \end{align*}
    such that $\varphi$ is uniquely determined by $\varphi_t = A_0 + A_1 \tau + \dots+ A_\ell \tau^\ell$ with $A_i \in \Mat_d(\mathbb{K})$ and $\rd \varphi_t:= A_0 = \theta \Id_d + N$ for some nilpotent matrix $N$. We set
    \begin{equation}\label{Eq:Endomorphism_Ring}
        \End(G):=\{u\in\Mat_d(\KK[\tau])\mid u\varphi_a=\varphi_au~\mathrm{for~all}~a\in\bA\}
    \end{equation}
    to be the endomorphism ring of $G$. It is equipped with an $\bA$-module structure via $\varphi$. The Lie algebra of $G$ is denoted by $\Lie G$, where its $\mathbb{C}_\infty$-valued points $(\Lie G)(\mathbb{C}_\infty)=\Mat_{d\times 1}(\mathbb{C}_\infty)$ has the induced $\bA$-module structure that is uniquely determined by the left multiplication of $\rd\varphi_a$ for $a\in\bA$. It was shown by Anderson \cite[Thm.~3]{And86} that there is a unique $d$-variable $\mathbb{F}_q$-linear everywhere convergent series of the form
    \begin{equation}\label{E:expgen}
        \Exp_\varphi\begin{pmatrix}
            z_1\\
            \vdots\\
            z_d
        \end{pmatrix}=\begin{pmatrix}
            z_1\\
            \vdots\\
            z_d
        \end{pmatrix}+\sum_{n\geq 1}B_n\begin{pmatrix}
            z_1^{q^n}\\
            \vdots\\
            z_d^{q^n}
        \end{pmatrix},~B_n\in\Mat_d(\mathbb{K})
    \end{equation}
    so that $\varphi_a\circ\Exp_\varphi=\Exp_\varphi\circ\rd\varphi_a$ for all $a\in\bA$. It induces an analytic $\bA$-linear homomorphism $\Exp_\varphi:(\Lie G)(\mathbb{C}_\infty)=\Mat_{d\times 1}(\mathbb{C}_\infty)\to G(\mathbb{C}_\infty)=\Mat_{d\times 1}(\mathbb{C}_\infty)$. We call $\Exp_\varphi$ the \emph{exponential map} of $G$ and denote by $\Lambda_G:=\Ker\big(\Exp_\varphi(\cdot)\big)\subset(\Lie G)(\mathbb{C}_\infty)$ the \emph{period lattice} of $G$, which is a free discrete $\bA$-submodule of $(\Lie G)(\mathbb{C}_\infty)$. The $t$-module $G$ is called \emph{uniformizable} if the exponential map $\Exp_\varphi$ is surjective. 

    Assume that $\mathbb{K}$ is a perfect field and $G=(\mathbb{G}_a^d,\varphi)$ is a fixed $d$-dimensional Anderson $t$-module defined over $\mathbb{K}$. We can associate a $\mathbb{K}[t,\tau]$-module to $G$ as follows. Set $\mathcal{M}_G:=\Mat_{1\times d}(\mathbb{K}[\tau])$ with the natural left $\mathbb{K}[\tau]$-module structure. We endow a $\mathbb{K}[t]$-module structure on $\mathcal{M}_G$ which is characterized by
    \[
        t\cdot m:=m\varphi_t,~m\in\mathcal{M}_G.
    \]
    If $\mathcal{M}_G$ is free of finite rank over $\mathbb{K}[t]$, then $\mathcal{M}_G$ is called the $t$-motive associated to $G$, and we call $G$ an \emph{abelian $t$-module}. On the other hand, we can also associate a $\mathbb{K}[t,\sigma]$-module to $G$. Let $\mathcal{N}_G:=\Mat_{1\times d}(\mathbb{K}[\sigma])$. Then, $\mathcal{N}_G$ can be regarded as a left $\mathbb{K}[\sigma]$-module in a canonical manner. One can define a $\mathbb{K}[t]$-module structure on $\mathcal{N}_G$ by setting
    \[
        t\cdot n:=n\varphi_t^*,~n\in\mathcal{N}_G.
    \]
    If $\mathcal{N}_G$ is free of finite rank over $\mathbb{K}[t]$, then $\mathcal{N}_G$ is called the dual $t$-motive associated to $G$, and we call $G$ a \emph{$t$-finite $t$-module}. It has been shown by Maurischat \cite{Mau21} that an Anderson $t$-module is abelian if and only if it is $t$-finite.

    Let $G=(\mathbb{G}_a^d,\varphi)$ be an abelian $t$-module (hence $t$-finite). We define $r:=\rank_{\KK[t]}\mathcal{M}_G$ to be the \emph{rank} of $G$. Let $\mathbf{m}=(\mathbf{m}_1,\dots,\mathbf{m}_r)^\tr\in\Mat_{r\times 1}(\mathcal{M}_G)$ be a fixed ordered $\KK[t]$-basis for $\mathcal{M}_G$ (resp. $\mathbf{n}=(\mathbf{n}_1,\dots,\mathbf{n}_r)\in\Mat_{r\times 1}(\mathcal{N}_G)$ be a fixed ordered $\KK[t]$-basis for $\mathcal{N}_G$). Suppose that the matrix $\wPhi_G\in\Mat_{r}(\KK[t])$ (resp. $\Phi_G\in\Mat_{r}(\KK[t])$) represents the $\tau$-action on $\mathbf{m}$ (resp. $\sigma$-action on $\mathbf{n}$), namely $\tau\mathbf{m}=\wPhi_G\mathbf{m}$ (resp. $\sigma\mathbf{n}=\Phi_G\mathbf{n}$). The $\KK[t]$-basis $\mathbf{m}$ on $\mathcal{M}_G$ together with the representing matrix $\wPhi_G$ (resp. $\mathbf{n}$ on $\mathcal{N}_G$ together with the representing matrix $\Phi_G$) induce an isomorphism of $\KK[t,\tau]$-modules between $\Mat_{1\times r}(\KK[t])$ and $\mathcal{M}_G$ (resp. $\KK[t,\sigma]$-modules between $\Mat_{1\times r}(\KK[t])$ and $\mathcal{N}_G$) given by $\iota_{\mathbf{m}}(a_1,\dots,a_r):=a_1\mathbf{m}_1+\cdots+a_r\mathbf{m}_r\in\mathcal{M}_G$ (resp. $\iota_{\mathbf{n}}(a_1,\dots,a_r):=a_1\mathbf{n}_1+\cdots+a_r\mathbf{n}_r\in\mathcal{N}_G$). We call the pair $(\iota_\mathbf{m},\wPhi_G)$ a \emph{$t$-frame} for the $t$-motive $\mathcal{M}_G$ (resp. $(\iota_\mathbf{n},\Phi_G)$ a \emph{$t$-frame} for the dual $t$-motive $\mathcal{N}_G$). Given a $t$-frame $(\iota_\mathbf{m},\wPhi_G)$ of $\mathcal{M}_G$ (resp. $(\iota_\mathbf{n},\Phi_G)$ of $\mathcal{N}_G$), if there exists $\wPsi_G\in\GL_r(\TT)$ so that $\wPsi_G=\wPhi_G\wPsi_G^{(-1)}$ (resp. $\Psi_G\in\GL_r(\TT)$ so that $\Psi_G^{(-1)}=\Phi_G\Psi_G$), then we call the tuple $(\iota_\mathbf{m},\wPhi_G,\wPsi_G)$ (resp. $(\iota_\mathbf{n},\Phi_G,\Psi_G)$) a \emph{rigid analytic trivialization} of $\mathcal{M}_G$ (resp. $\mathcal{N}_G$). It was shown by Anderson \cite[Thm.~4]{And86} that a $t$-module $G$ is uniformizable if and only if $\mathcal{M}_G$ (resp. $\mathcal{N}_G$) has a rigid analytic trivialization.

    We finish this subsection by giving some fundamental examples of $t$-modules which play crucial roles in the present paper.

    \begin{example}\label{Ex:Carlitz_Tensor_Powers}
        For $n\geq 1$, let $\mathbf{C}^{\otimes n}=(\mathbb{G}_a^n,[\cdot]_n)$ be the $n$-th tensor power of the Carlitz module, where $[\cdot]_n:\bA\to\Mat_n(\KK[\tau])$ is the $\mathbb{F}_q$-algebra homomorphism uniquely determined by
        \[
            [t]_n=\begin{pmatrix}
                \theta & 1 & & \\
                 & \ddots & \ddots & \\
                 & & \theta & 1 \\
                 \tau & & & \theta
            \end{pmatrix}
        \]
        if $n \geq 2$, and 
          \[
        [t]_1 = \theta+\tau.
        \] 
        The associated $t$-motive of $\mathbf{C}^{\otimes n}$ is given by $\mathcal{M}_{\mathbf{C}^{\otimes n}}=\Mat_{1\times n}(\KK[\tau])$. Note that $\mathbf{m}_{\mathbf{C}^{\otimes n}}:=(1,0,\dots,0)$ is a $\KK[t]$-basis of $\mathcal{M}_{\mathbf{C}^{\otimes n}}$ so that $\tau\mathbf{m}_{\mathbf{C}^{\otimes n}}=(t-\theta)^n \mathbf{m}_{\mathbf{C}^{\otimes n}}$. Similarly, consider the associated dual $t$-motive $\mathcal{N}_{\mathbf{C}^{\otimes n}}=\Mat_{1\times n}(\KK[\sigma])$. Then, $\mathbf{n}_{\mathbf{C}^{\otimes n}}:=(0,\dots,0,1)$ is a $\KK[\sigma]$-basis so that $\sigma\mathbf{n}_{\mathbf{C}^{\otimes n}}=(t-\theta)^n\mathbf{n}_{\mathbf{C}^{\otimes n}}$. In this case, we have
        \begin{equation}\label{Eq:Carlitz_wPhi_Phi}
            \wPhi_{\mathbf{C}^{\otimes n}}=\Phi_{\mathbf{C}^{\otimes n}}=(t-\theta)^n.
        \end{equation}
    \end{example}

    \begin{example}\label{Ex:Drinfeld_Module}
        Let $E=(\mathbb{G}_a,\rho)$ be a Drinfeld module of rank $r\geq 2$ with $\rho_t=\theta+\kappa_1\tau+\cdots+\kappa_r\tau^r\in\mathbb{K}[\tau]$ and $\kappa_r\neq 0$. The associated $t$-motive of $E$ is given by $\mathcal{M}_E=\mathbb{K}[\tau]$ whose $\mathbb{K}[t]$-action is uniquely determined by
        \[
            t\cdot m:=m\rho_t=m\left(\theta+\kappa_1\tau+\cdots+\kappa_r\tau^r\right),~m\in\mathcal{M}_E.
        \]
        It is clear that $\{1\}$ is a $\mathbb{K}[\tau]$-basis and $\mathbf{m}_E=(1,\tau,\dots,\tau^{r-1})^\tr\in\Mat_{r\times 1}(\mathcal{M}_E)$ forms an ordered $\mathbb{K}[t]$-basis with $\tau\mathbf{m}_E=\wPhi_E\mathbf{m}_E$, where $\wPhi_E$ is given by
        \begin{equation}\label{Eq:Drinfeld_wPhi}
            \wPhi_E:=\frac{1}{\kappa_r}\begin{pmatrix}
                0 & \kappa_r & & \\
                \vdots & & \ddots \\
                0 & & & \kappa_r\\
                t-\theta & -\kappa_1 & \cdots &
                -\kappa_{r-1} 
            \end{pmatrix}\in\Mat_r(\KK[t])\cap\GL_r(\KK(t)).
        \end{equation}
        Likewise, the associated dual $t$-motive of $E$ is given by $\mathcal{N}_E=\KK[\sigma]$ whose $\KK[t]$-action is uniquely determined by
        \[
            t\cdot n:=n\rho_t^*=n\left(\theta+\kappa_1^{(-1)}\sigma+\cdots+\kappa_r^{(-r)}\sigma^r\right),~n\in\mathcal{N}_E.
        \]
        It is clear that $\{1\}$ is a $\mathbb{K}[\sigma]$-basis and $\mathbf{n}_E=(1,\sigma,\dots,\sigma^{r-1})^\tr\in\Mat_{r\times 1}(\mathcal{N}_E)$ forms an ordered $\mathbb{K}[t]$-basis with $\sigma\mathbf{n}_E=\Phi_E\mathbf{n}_E$, where $\Phi_E$ is given by
        \begin{equation}\label{Eq:Drinfeld_Phi}
            \Phi_E:=\frac{1}{\kappa_r^{(-r)}}\begin{pmatrix}
                0 & \kappa_r^{(-r)} & & \\
                \vdots & & \ddots \\
                0 & & & \kappa_r^{(-r)}\\
                t-\theta & -\kappa_1^{(-1)} & \cdots &
                -\kappa_{r-1}^{(1-r)} 
            \end{pmatrix}\in\Mat_r(\KK[t])\cap\GL_r(\KK(t)).
        \end{equation}
    \end{example}
    \begin{example}\label{Ex:Exterior_Powers}
        Let $E=(\mathbb{G}_a,\rho)$ be the Drinfeld module of rank $r\geq 2$ as in Example~\ref{Ex:Drinfeld_Module}. For $e\geq 0$, consider the $t$-module $E_e:=\mathbf{C}^{\otimes e}\otimes\wedge^{r-1}E:=(\mathbb{G}_a^{re+r-1},\psi_e)$, where for $e=0$, we set $E_0:=\wedge^{r-1}E$ with
        \begin{align*}
            \psi_0:=\wedge^{r-1}\rho:\bA&\to\Mat_{r-1}(\mathbb{K}[\tau])\\
            t\mapsto\theta\Id_{r-1}+(-1)^{r-1}&\begin{pmatrix}
                -\kappa_{r-1} & \kappa_r &  &  \\
                \vdots &  & \ddots &  \\
                -\kappa_2 &  & & \kappa_r \\
                -\kappa_1 & 0 & \cdots & 0
            \end{pmatrix}\tau+\begin{pmatrix}
                0 & \cdots & \cdots & 0\\
                \vdots & \ddots &  & \vdots\\
                0 & & \ddots & \vdots\\
                \kappa_r & 0 & \cdots & 0
            \end{pmatrix}\tau^2,
        \end{align*}
        while for $e\geq 1$, we set 
        \begin{align*}
            \psi_e:&\bA\to\Mat_{re+r-1}(\mathbb{K}[\tau])\\
            & t\mapsto\theta\Id_{re+r-1}+N+ B\tau, 
        \end{align*}
        with 
        \begin{equation}\label{E:Nilpotent}
		  N:=\begin{pmatrix}
		  0&\cdots &0&1& 0&\cdots & 0 \\
		  & \ddots& & \ddots& \ddots & &\vdots \\
		  & &\ddots& &\ddots & \ddots&0\\
		  & &  & 0&\cdots&0 & 1\\
		  & &  & & 0&\cdots& 0\\
		  & &  & & &\ddots& \vdots\\
		  & &  & & & & 0\\
		  \end{pmatrix}\begin{aligned}
		  &\left.\begin{matrix}
		  \\
		  \\
		  \\
		  \\
		  \end{matrix} \right\} 
		  re-1\\ 
		  &\left.\begin{matrix}
		  \\
		  \\
		  \\
		  \end{matrix}\right\}
		  r
		  \end{aligned}
		\end{equation}
        and 
		\begin{equation}\label{E:Bmatrixtau}
		  B:=(-1)^{r-1}\begin{pmatrix}
		  0&\cdots&\cdots & \cdots &\cdots & \cdots & 0\\
		  \vdots& & & & & &\vdots\\
		  0& & & & & & 0\\
		  1& 0 & \cdots&\cdots&\cdots &\cdots &0\\
		  -\kappa_{r-1}&\kappa_r&\ddots &  & & &\vdots\\
		  \vdots& & \ddots& \ddots & & & \vdots\\
		  -\kappa_{1}&0&\cdots &\kappa_r&0&\cdots& 0
		  \end{pmatrix}\begin{aligned}
		  &\left.\begin{matrix}
		  \\
		  \\
		  \\
		  \end{matrix} \right\} 
		  re-1\\ 
		  &\left.\begin{matrix}
		  \\
		  \\
		  \\
		  \\
		  \end{matrix}\right\}
		  r\\
		  \end{aligned}.
		\end{equation}
  
        Let $\mathcal{M}_{E_e}=\Mat_{1\times (re+r-1)}(\mathbb{K}[\tau])$ be the $t$-motive associated to $E_e$ and $\{\widetilde{\mathbf{s}}_1, \dots, \widetilde{\mathbf{s}}_r\}$ be the standard $\mathbb{K}[\tau]$-basis. It is straightforward to check (see Appendix for more details) that for $e=0$,
        \begin{equation}\label{E:tbasistmotiveexteriorpower}
            \mathbf{m}_{E_0}:=\big((-1)^{r-1}\tau\widetilde{\mathbf{s}}_1,\widetilde{\mathbf{s}}_{r-1},\dots,\widetilde{\mathbf{s}}_{1}\big)^\tr\in\Mat_{r\times 1}(\mathcal{M}_{E_0})
        \end{equation}
        is the ordered $\mathbb{K}[t]$-basis of $\mathcal{M}_{E_0}$ so that $\tau\mathbf{m}_{E_0}=\mathrm{Cof}(\wPhi_E)\mathbf{m}_{E_0}$, where
        \begin{equation}\label{E:CofwPhiE}
            \mathrm{Cof}(\wPhi_E)=\frac{(-1)^{r-1}}{\kappa_r}\begin{pmatrix}
                \kappa_1 & t-\theta & 0 & \dots & 0\\
                \vdots & & \ddots & & \vdots\\
                \vdots & & & \ddots & \vdots\\
                \kappa_{r-1} & 0 & \dots & \dots & t-\theta\\
                \kappa_r & 0 & \dots & \dots & 0
            \end{pmatrix}\in\Mat_r(\KK[t])\cap\GL_r(\KK(t)),
        \end{equation}
        and for $e\geq 1$, \begin{equation}\label{E:tbasistmotiveexteriorpower_e_geq_1}
            \mathbf{m}_{E_e}:=\big(\widetilde{\mathbf{s}}_r,\dots,\widetilde{\mathbf{s}}_{1}\big)^\tr\in\Mat_{r\times 1}(\mathcal{M}_{E_e})
        \end{equation}
        is the ordered $\mathbb{K}[t]$-basis of $\mathcal{M}_{E_e}$ so that $\tau\mathbf{m}_{E_e}=(t-\theta)^e\mathrm{Cof}(\wPhi_E)\mathbf{m}_{E_e}$.  Hence, for $e\geq 0$, the pair $\big(\iota_{\mathbf{m}_{E_e}},(t-\theta)^e\mathrm{Cof}(\wPhi_E)\big)$ defines a $t$-frame for the $t$-motive $\mathcal{M}_{E_e}$.
    It is subtle to find an appropriate $t$-frame for the dual $t$-motive $\mathcal{N}_{E_e}$ of the $t$-module $E_e$, but it can be achieved by adopting the calculations from \cite[\S3]{GN24}. 
    In conclusion, there exists an ordered $\KK[t]$-basis $\mathbf{n}_{E_e}$ of $\mathcal{N}_{E_e}$ so that $\sigma\mathbf{n}_{E_e}=(t-\theta)^e\mathrm{Cof}(\Phi_E)\mathbf{n}_{E_e}$, where
    \begin{equation}\label{E:CofPhiE}
        \Cof(\Phi_E) = \frac{(-1)^{r-1}}{\kappa_r^{(-r)}}\begin{pmatrix}
            \kappa_1^{(-1)} & t-\theta & 0 & \dots & 0\\
            \vdots & & \ddots & & \vdots\\
            \vdots & & & \ddots & \vdots\\
            \kappa_{r-1}^{(-r+1)} & 0 & \dots & \dots & t-\theta\\
            \kappa_r^{(-r)} & 0 & \dots & \dots & 0
        \end{pmatrix}\in\Mat_r(\KK[t])\cap\GL_r(\KK(t)).
    \end{equation}
    In other words, the $t$-module $E_e$ has a $t$-frame $(\iota_{\mathbf{n}_{E_e}},(t-\theta)^e\mathrm{Cof}(\Phi_E))$ for its dual $t$-motive $\mathcal{N}_{E_e}$. To be self-contained, we put the details in Appendix \S A.1.
\end{example}

\subsection{Extensions of \texorpdfstring{$t$}{t}-modules and biderivations}
    In what follows, we recall essential results from \cite{PR03} and \cite{BrPa02} about extensions of $t$-modules and biderivations. For our purposes, we restrict ourselves to extensions of uniformizable abelian $t$-modules by the tensor powers of the Carlitz module. Let $G=(\mathbb{G}_a^{d},\varphi)$ be a $d$-dimensional $t$-module. For $n\geq 1$, recall the $n$-th tensor power of the Carlitz module $\mathbf{C}^{\otimes n}=(\mathbb{G}_a^n,[\cdot]_n)$ defined in Example~\ref{Ex:Carlitz_Tensor_Powers}. By an abuse of notation, we set $\mathbf{C}^{\otimes 0}:=(\mathbb{G}_a,[\cdot]_0)$ to be the trivial $t$-module with $[t]_0:=\theta$ and $\Exp_{[\cdot]_0}(z)=z$.

Let $n\geq 1$. A \emph{$(\varphi,[\cdot]_n)$-biderivation} $\delta$ is an $\mathbb{F}_q$-linear map
    \[
        \delta:\bA\to\Mat_{n\times d}(\KK[\tau]),
    \]
    such that for any $a,b\in\bA$
    \begin{equation}\label{E:bideraction}
        \delta(ab)=[a]_n\delta(b)+\delta(a)\varphi_b.
    \end{equation}
We denote by $\mathrm{Der}(\varphi,[\cdot]_n)$ the $\mathbb{F}_q$-vector space of all $(\varphi,[\cdot]_n)$-biderivations. Note that each $(\varphi,[\cdot]_n)$-biderivation $\delta$ is uniquely determined by $\delta(t)\in\Mat_{n\times d}(\mathbb{K}[\tau])$. Thus, we can naturally identify $\mathrm{Der}(\varphi,[\cdot]_n)$ with $\Mat_{n\times d}(\mathbb{K}[\tau])$ through the following isomorphism of $\mathbb{F}_q$-vector spaces:
    \begin{align*}
        \Mat_{n\times d}(\mathbb{K}[\tau])&\to\mathrm{Der}(\varphi,[\cdot]_n)\\
        B&\mapsto\delta_B
    \end{align*}
    where $\delta_B$ is the $(\varphi,[\cdot]_n)$-biderivation with $\delta_B(t):=B$. A $(\varphi,[\cdot]_n)$-biderivation $\delta$ is called \emph{inner} if there exists $U\in\Mat_{n\times d}(\mathbb{K}[\tau])$ so that
    \begin{equation}\label{E:biderinner}
        \delta(a)=\delta^{(U)}(a):=U\varphi_a-[a]_nU,~a\in\bA.
    \end{equation}
    We denote by $\mathrm{Der}_{\mathrm{in}}(\varphi,[\cdot]_n)\subset\mathrm{Der}(\varphi,[\cdot]_n)$ the subspace of inner $(\varphi,[\cdot]_n)$-biderivation. 
    For $\delta\in\mathrm{Der}(\varphi,[\cdot]_n)$, we set $G_\delta=(\mathbb{G}_a^{d+n},\phi)$ to be the $t$-module induced from $\delta$, that is, we have
    \begin{equation}\label{Eq:t-module_from_delta}
        \phi_t =\begin{pmatrix}
        \varphi_t & 0 \\
        \delta(t) &  [t]_n
        \end{pmatrix}\in\Mat_{d+n}(\KK[\tau]).
    \end{equation}
    Note that $G_\delta$ fits into the following short exact sequence of $t$-modules
    \[
        0\to \mathbf{C}^{\otimes n}\hookrightarrow G_\delta\twoheadrightarrow G\to 0.
    \]
    Let
    \begin{equation}\label{Eq:Der_0}
        \mathrm{Der}_\partial(\varphi,[\cdot]_n):=\{\delta\in\mathrm{Der}(\varphi,[\cdot]_n)\mid\delta(t)\in\Mat_{n\times d}(\mathbb{K}[\tau])\tau\}.
    \end{equation}
    Then, by \cite[Lem.~5.1]{PR03} and the uniqueness of the exponential function, for each $\delta\in\mathrm{Der}_\partial(\varphi,[\cdot]_n)$ we have
    \begin{equation}\label{Eq:Exp_of_Extension}
        \Exp_\phi\begin{pmatrix}
            \bz_1\\
            \bz_2
        \end{pmatrix}=\begin{pmatrix}
            \Exp_{\varphi}(\bz_1)\\
            \Exp_{[\cdot]_n}(\bz_2)+\mathrm{F}_{\delta}(\bz_1)
        \end{pmatrix},
    \end{equation}
    where $\mathrm{F}_\delta:\Mat_{d\times 1}(\mathbb{K})\to\Mat_{n\times 1}(\mathbb{K})$ 
    is the unique entire $\mathbb{F}_q$-linear analytic function so that
    \begin{subequations}\label{E:qpf}
   \begin{align}
      &  F_\delta(\bz)\equiv 0~(\mathrm{mod}~\deg q),\label{E:qpf1}\\
     &F_\delta(\rd\varphi_a \bz)=[a]_n(\mathrm{F}_\delta(\bz))+\delta(a)(\Exp_{\varphi}(\bz)) \text{ for any } a\in\bA. \label{E:qpf2}
   \end{align}
   \end{subequations}
    Note that $\mathrm{F}_\delta(\bz)$ can be realized as the \emph{quasi-periodic function} associated to the $(\varphi,[\cdot]_n)$-biderivation $\delta$.

    For the special case of $n=0$, a $(\varphi,[\cdot]_0)$-biderivation $\delta$ is an $\mathbb{F}_q$-linear map $\delta:\bA\to\Mat_{1\times d}(\KK[\tau])$ satisfying \eqref{E:bideraction} specialized to $n=0$. 
We can naturally identify $\mathrm{Der}(\varphi,[\cdot]_0)$ with $\Mat_{1\times d}(\mathbb{K}[\tau])$ through the $\mathbb{F}_q$-vector space isomorphism $\Mat_{1\times d}(\mathbb{K}[\tau])\to\mathrm{Der}(\varphi,[\cdot]_0)$ given by $B\mapsto\delta_B$, where $\delta_B$ is the $(\varphi,[\cdot]_0)$-biderivation with $\delta_B(t):=B$. Similarly,  a $(\varphi,[\cdot]_0)$-biderivation $\delta$ is inner if there exists $U\in\Mat_{1\times d}(\mathbb{K}[\tau])$ satisfying \eqref{E:biderinner} specialized to $n=0$. 
    For $\delta\in\mathrm{Der}(\varphi,[\cdot]_0)$, we also set $G_\delta=(\mathbb{G}_a^{d+1},\phi_0)$ to be the $t$-module induced from $\delta$, that is, $(\phi_0)_t$ is of the form \eqref{Eq:t-module_from_delta} specialized to $n=0$.
    Note that $\phi_0$ fits into the  short exact sequence of $t$-modules
    \[
        0\to \GG_a\hookrightarrow G_\delta\twoheadrightarrow G\to 0.
    \]
    Then, by the uniqueness of the exponential function, $\Exp_{\phi_0}\begin{psmallmatrix}\bz_1\\ \bz_2 \end{psmallmatrix}$ satisfies \eqref{Eq:Exp_of_Extension}
specialized to $n=0$, 
    where $\mathrm{F}_\delta:\Mat_{d\times 1}(\mathbb{K})\to \mathbb{K}$ 
    is the unique entire $\mathbb{F}_q$-linear analytic function also satisfying \eqref{E:qpf1} and \eqref{E:qpf2} specialized to $n=0$ (see \cite[\S3.2]{BrPa02}).
For $n=0$, we simply call $\mathrm{F}_\delta$ a quasi-periodic function of $G$ associated to $\delta$. Moreover, for $\bm{\omega} \in \Lambda_G$, the value $\mathrm{F}_{\delta}(\bm{\omega})$ is called the \emph{quasi-period} for $\bm{\omega}$ and more generally, for $\bm{y} \in \Mat_{d\times 1}(\KK)$, the value $F_{\delta}(\bm{y})$ is called the \emph{quasi-logarithm} for $\bm{y}$ associated to the $(\varphi,[\cdot]_0)$-biderivation $\delta$.
    
    Returning to $n\geq 1$, we define $\Ext^1(\varphi,[\cdot]_n)$ to be the group of $t$-module extensions of $G$ by $\mathbf{C}^{\otimes n}$ under the Baer sum up to Yoneda equivalence. Then, $\big(\delta\mapsto [G_\delta]\big)$ induces a surjective map from $\mathrm{Der}(\varphi,[\cdot]_n)$ into $\Ext^1(\varphi,[\cdot]_n)$. Furthermore, the following result gives a characterization of $\Ext^1(\varphi,[\cdot]_n)$.
    
    \begin{lemma}[{\cite[Lem.~2.1]{PR03}}]
        We have the following $\mathbb{F}_q[t]$-module isomorphism
        \begin{align*}
            \mathrm{Der}(\varphi,[\cdot]_n)/\mathrm{Der}_{\mathrm{in}}(\varphi,[\cdot]_n)&\overset{\sim}{\to}\Ext^1(\varphi,[\cdot]_n)\\
            \delta&\mapsto [G_\delta],
        \end{align*}
        where $G_\delta=(\mathbb{G}_a^{d+n},\phi)$ with $\phi_t$ as defined in \eqref{Eq:t-module_from_delta}.
    \end{lemma}

    The group of $t$-module extensions $\Ext^1(\varphi,[\cdot]_n)$ contains a special subgroup $\Ext^1_\partial(\varphi,[\cdot]_n)$ that collects all the extensions of $G$ by $\mathbf{C}^{\otimes n}$ whose induced extensions on the tangent space split. More precisely, we define
    \begin{equation}\label{Eq:Ext_0}
        \Ext^1_\partial(\varphi,[\cdot]_n):=\mathrm{Der}_\partial(\varphi,[\cdot]_n)/\big(\mathrm{Der}_\partial(\varphi,[\cdot]_n)\cap\mathrm{Der}_{\mathrm{in}}(\varphi,[\cdot]_n)\big).
    \end{equation}

\section{Duality relations among Anderson generating functions}\label{S:DUalityNAGF}
    In this section, we aim to demonstrate the main techniques used in the present article. 
    We compare two different constructions of solutions of the same Frobenius difference equation. Both constructions can be achieved by using \emph{Anderson generating functions} which was originally defined by Anderson \cite{And86}. To begin with, we follow \cite{NPapanikolas21} closely to recall the essential properties of Anderson generating functions.
\subsection{Frobenius difference equation from \texorpdfstring{$t$}{t}-motives}\label{SS:Frobsdiffs}
    Let $G=(\mathbb{G}_a^d,\varphi)$ be a $d$-dimensional uniformizable abelian $t$-module. For $\bm{y}\in\Mat_{d\times 1}(\mathbb{C}_\infty)$, we define the \emph{Anderson generating function of $\bm{y}$} associated to $\varphi$ as follows:
    \[
        \cG_{\bsy}(t;\varphi):=\sum\limits_{i=0}^\infty \Exp_{\varphi}\big((\rd\varphi_{t})^{-i-1}\bsy\big)t^i\in\Mat_{d\times 1}(\TT).
    \]
    By \cite[Lem.~4.2.2]{NPapanikolas21}, which generalizes the work of Pellarin \cite{Pellarin08} on Drinfeld modules, we have
    \begin{equation}\label{E:AGFRes}
      \cG_{\bsy}(t;\varphi)= \bigl( \rd \varphi_t - t\Id_d\bigr)^{-1} \bsy + \sum_{n=1}^{\infty} B_n \Bigl( \bigl( \rd \varphi_t - t\Id_d\bigr)^{-1} \Bigr)^{(n)} \bsy^{(n)}, 
    \end{equation}
    where for $n\geq 1$, each $B_n$ is as in \eqref{E:expgen}.

For a $(\varphi, [\cdot]_0)$-biderivation $\delta$ such that $\delta_t \in \Mat_{1\times d}(\KK[\tau])\tau$ and $\bsy \in \Mat_{d\times 1}(\KK)$, by \cite[4.3.5]{NPapanikolas21}, we have 
\begin{equation}\label{E:AGFandQLogs}
\langle \delta(t) \mid \cG_{\bsy}(t;\varphi)\rangle |_{t=\theta} = F_{\delta}(\bsy).
\end{equation}

    Assume that $\{\bm{\omega}_{1},\dots,\bm{\omega}_{r}\}\subset\Mat_{d\times 1}(\mathbb{C}_\infty)$ is an $\bA$-basis for the period lattice $\Lambda_G$. Consider the $t$-motive $\mathcal{M}_G=\Mat_{d\times 1}(\KK[\tau])$ associated to $G$ with a fixed $t$-frame $(\iota_\mathbf{m},\wPhi_G)$. If the $\KK[t]$-basis $\mathbf{m}=(\mathbf{m}_1,\dots,\mathbf{m}_r)^\tr\in\Mat_{r\times 1}(\mathcal{M}_G)$, then the matrix
    \begin{equation}\label{Eq:Rigid_Analytic_Trivialization_AGFs}
        \begin{split}
            \Upsilon_G:&=\left\langle\begin{pmatrix}
                \tau\mathbf{m}_1\\
                \vdots\\
                \tau\mathbf{m}_{r}
            \end{pmatrix}\mid \big(\cG_{\bm{\omega}_1}(t;\varphi),\dots,\cG_{\bm{\omega}_r}(t;\varphi)\big)\right\rangle\\
            &=\begin{pmatrix}
                \langle\tau\mathbf{m}_1\mid \cG_{\bm{\omega}_1}(t;\varphi)\rangle & \dots & \langle\tau\mathbf{m}_1\mid \cG_{\bm{\omega}_r}(t;\varphi)\rangle\\
                \vdots & &\vdots\\
                \langle\tau\mathbf{m}_{r}\mid \cG_{\bm{\omega}_1}(t;\varphi)\rangle & \dots & \langle\tau\mathbf{m}_{r}\mid \cG_{\bm{\omega}_r}(t;\varphi)\rangle
            \end{pmatrix}\in\GL_r(\TT)
        \end{split}
    \end{equation}
    satisfies the Frobenius difference equation $\Upsilon_G=\wPhi_G\Upsilon_G^{(-1)}$, where $\langle\tau\mathbf{m}_{i}\mid \cG_{\omega_j}(t;\varphi)\rangle$ is defined in \eqref{Eq:Inner_Product}. Hence, $(\iota_{\mathbf{m}},\wPhi_G,\Upsilon_G)$ is a rigid analytic trivialization of $\mathcal{M}_G$.

    Let $(\iota_\mathbf{n},\Phi_G)$ be a $t$-frame of the dual $t$-motive $\mathcal{N}_G=\Mat_{d\times 1}(\KK[\sigma])$ associated to $G$. It was shown by Hartl and Juschka \cite[Thm.~2.5.13]{HJ20} (see also \cite[Thm.~4.4.9]{NPapanikolas21}) that there exists $V_G\in\GL_{r}(\KK[t])$ so that
    \begin{equation}\label{Eq:The_Change_of_Basis_Matrix}
        V_G^{(-1)}\Phi_G=\wPhi_G^{\tr}V_G.
    \end{equation}
    Moreover, if we set $\Psi_G:=(\Upsilon_G^\tr V_G)^{-1}$,
    then 
    \[
        \Psi_G^{(-1)}=\Phi_G\Psi_G.
    \]
    Hence the tuple $(\iota_{\mathbf{n}},\Phi_G,\Psi_G)$ gives a rigid analytic trivialization of $\mathcal{N}_G$. The existence of $V_G$ will provide the bridge for us to switch from the $t$-motive side to the dual $t$-motive side.

    For $n\geq 1$, consider $\mathbf{C}^{\otimes n}=(\mathbb{G}_a^n,[\cdot]_n)$ and its $t$-motive $\mathcal{M}_{\mathbf{C}^{\otimes n}}=\Mat_{1\times n}(\KK[\tau])$ (resp. dual $t$-motive $\mathcal{N}_{\mathbf{C}^{\otimes n}}=\Mat_{1\times n}(\KK[\sigma])$) as well as a $t$-frame $(\iota_{\mathbf{m}_{\mathbf{C}^{\otimes n}}},\wPhi_{\mathbf{C}^{\otimes n}})$ (resp. $(\iota_{\mathbf{n}_{\mathbf{C}^{\otimes n}}},\Phi_{\mathbf{C}^{\otimes n}})$) defined in Example~\ref{Ex:Carlitz_Tensor_Powers}. Note that $\Lambda_{\mathbf{C}^{\otimes n}}$ is free of rank $1$ over $\bA$, and we can choose a generator $\bm{\gamma}_n\in\Mat_{n\times 1}(\mathbb{C}_\infty)$ with the last entry given by $\Tilde{\pi}^n$.  
    More precisely, we have \begin{equation}\label{Eq:Period_Carlitz_Tensor_Powers}
        \Lambda_{\mathbf{C}^{\otimes n}}=\bA\bm{\gamma}_n=\bA\begin{pmatrix}
            \star\\
            \vdots\\
            \star\\
            \Tilde{\pi}^n
        \end{pmatrix}\subset\Mat_{n\times 1}(\mathbb{C}_\infty).
    \end{equation}
    We set $\omega:=\cG_{\Tilde{\pi}}(t;\mathbf{C})\in\TT$ and $\Omega:=1/\omega^{(1)}\in\TT$. Then, we have $\Upsilon_{\mathbf{C}^{\otimes n}}=\omega^n$ and $\Psi_{\mathbf{C}^{\otimes n}}=\Omega^n$. In other words, $\big(\iota_{\mathbf{m}_{\mathbf{C}^{\otimes n}}},(t-\theta)^n,\omega^n\big)$ (resp. $\big(\iota_{\mathbf{n}_{\mathbf{C}^{\otimes n}}},(t-\theta)^n,\Omega^n\big)$) is a rigid analytic trivialization of $\mathcal{M}_{\mathbf{C}^{\otimes n}}$ (resp. $\mathcal{N}_{\mathbf{C}^{\otimes n}}$).

    For $\delta\in\mathrm{Der}(\varphi,[\cdot]_n)$, we set $G_\delta=(\mathbb{G}_a^{d+n},\phi)$ to be the $t$-module induced from $\delta$ in the sense that $\phi_t$ is defined as in \eqref{Eq:t-module_from_delta}. Let $\mathcal{M_\phi}=\Mat_{1\times(d+n)}(\KK[\tau])$ 
    be its $t$-motive with an evident choice of the ordered $\KK[t]$-basis for $\mathcal{M}_\phi$ given by
    \begin{equation}\label{Eq:Basis_of_M_phi}
        \mathbf{m}_\phi:=\bigg((\mathbf{m}_{1},0),\dots,(\mathbf{m}_{r},0),(0,\mathbf{m}_{\mathbf{C}^{\otimes n}})\bigg)^\tr\subset\Mat_{(d+n)\times 1}(\mathcal{M}_\phi).
    \end{equation}
    Note that $\tau\mathbf{m}_\phi=\wPhi_\phi\mathbf{m}_\phi$ for some $\wPhi_\phi$ of the form
    \[
        \widetilde{\Phi}_\phi:=\begin{pmatrix}
            \widetilde{\Phi}_G & \\
            \widetilde{\Phi}_\delta & (t-\theta)^n
        \end{pmatrix}\in\Mat_{r+1}(\mathbb{K}[t]),
    \]
    for some $\widetilde{\Phi}_\delta\in\Mat_{1\times r}(\mathbb{K}[t])$. This induces an identification
    \begin{equation}\label{Eq:Phi_delta}
        \begin{split}
            \mathrm{Der}(\varphi,[\cdot]_n)&\to\Mat_{1\times r}(\mathbb{K}[t])\\
            \delta&\mapsto\widetilde{\Phi}_\delta.
        \end{split}
    \end{equation}

    By \eqref{Eq:Exp_of_Extension}, we deduce that the period lattice $\Lambda_{G_\delta}$ can be chosen to be the form
    \begin{equation}\label{Eq:Lambda_delta}
        \Lambda_{G_\delta}=\bA\bm{\omega}_{\phi,1}+\cdots+\bA\bm{\omega}_{\phi,r+1}\subset\Mat_{(d+n)\times 1}(\mathbb{C}_\infty),
    \end{equation}
    where $\bm{\omega}_{\phi,j}=\big(\bm{\omega}_{j}^\tr,\bm{\lambda}_{\delta,j}^\tr\big)^\tr$ for some $\bm{\lambda}_{\delta,j}\in\Mat_{n\times 1}(\mathbb{C}_\infty)$ if $1\leq j\leq r$ and $\bm{\omega}_{\phi,r+1}=\big(0,\bm{\gamma}_{n}^\tr\big)^\tr$. By using \eqref{Eq:Exp_of_Extension} again, we deduce that for $1\leq j\leq r$, we have
    \begin{equation}\label{Eq:AGF_of_delta}
        \cG_{\bm{\omega}_{\phi,j}}(t;\phi)=\begin{pmatrix}
            \cG_{\bm{\omega}_{j}}(t;\varphi)\\
            \cF_{\bm{\omega}_{\phi,j}}(t;\phi)
        \end{pmatrix}\in \Mat_{(d+n)\times 1}(\TT)
    \end{equation}
    for some $\cF_{\bm{\omega}_{\phi,j}}(t;\phi)=\big(g_{j,1}(t;\phi),\dots,g_{j,n}(t;\phi)\big)^\tr\in\Mat_{n\times 1}(\TT)$ and
    \[
        \cG_{\bm{\omega}_{\phi,r+1}}(t;\phi)=\begin{pmatrix}
            0\\
            \cG_{\bm{\gamma_n}}(t;[\cdot]_n)
        \end{pmatrix}\in \Mat_{(d+n)\times 1}(\TT).
    \]
    By following the construction given in \eqref{Eq:Rigid_Analytic_Trivialization_AGFs}, we obtain that
    \begin{equation}\label{Eq:Rigid_Analytic_Trivialization_Ext}
        \Upsilon_{G_\delta}=\left\langle\tau\mathbf{m}_\phi
        \mid \big(\begin{pmatrix}
            \cG_{\bm{\omega}_{1}}(t;\varphi)\\
            \cF_{\bm{\omega}_{\phi,1}}(t;\phi)
        \end{pmatrix},\dots,\begin{pmatrix}
            \cG_{\bm{\omega}_{r}}(t;\varphi)\\
            \cF_{\bm{\omega}_{\phi,r}}(t;\phi)
        \end{pmatrix},\begin{pmatrix}
            0\\
            \cG_{\bm{\gamma_n}}(t;[\cdot]_n)
        \end{pmatrix}\big)\right\rangle=\begin{pmatrix}
        \Upsilon_{G} & \\
        \Upsilon_\delta & \omega^n
        \end{pmatrix},
    \end{equation}
    where
    \begin{equation}\label{Eq:Upsilon_Delta}
        \begin{split}
            \Upsilon_\delta:&=\bigg(\langle\tau\mathbf{m}_{\mathbf{C}^{\otimes n}}\mid \cF_{\bm{\omega}_{\phi,1}}(t;\phi)\rangle,\dots,\langle\tau\mathbf{m}_{\mathbf{C}^{\otimes n}}\mid \cF_{\bm{\omega}_{\phi,r}}(t;\phi)\rangle\bigg)\\
            &=\bigg(g_{1,1}(t;\phi)^{(1)},\dots,g_{r,1}(t;\phi)^{(1)}\bigg)\in\Mat_{1\times r}(\TT).
        \end{split}
    \end{equation}

    For the convenience of later use, for any $\eta\in\mathbb{K}^\times$, we set $\mathfrak{u}_\eta\in\mathbb{K}^\times$ to be a fixed choice of the $(q-1)$-th root of $\eta^{-q}$. Note that we must have
    \begin{equation}\label{Eq:Normalization_Constant}
        \mathfrak{u}_\eta^{(-1)}=\eta\mathfrak{u}_\eta.
    \end{equation}
    The starting point of the present paper is the following Proposition.
    
    \begin{proposition}\label{Prop:First_Difference_Equation}
        Let $B\in\GL_{r+1}(\TT)$ satisfying the Frobenius difference equation
        \[
            B^{(-1)}=\wPhi_\phi^{\ad} B.
        \]
        Assume that $\det(\wPhi_G)=\widetilde{c}(t-\theta)^d$ for some $\widetilde{c}\in\KK^\times$. Then, there exists $C\in\Mat_{r+1}(\bA)$ such that
        \[
            \mathfrak{u}_{\widetilde{c}}\Omega^d\Upsilon_\phi=BC.
        \]
        In particular, if we denote by $\mathfrak{b}$ the last row of $B$, then we have
        \[
            \mathfrak{u}_{\widetilde{c}}\big(\Omega^d\Upsilon_\delta,\Omega^{d-n}\big)=\mathfrak{b}C\in\Mat_{1\times (r+1)}(\TT).
        \]
    \end{proposition}
    
    \begin{proof}
        By \eqref{Eq:Normalization_Constant} and the construction of $\Upsilon_\phi$ given in \eqref{Eq:Rigid_Analytic_Trivialization_Ext}, one verifies directly that
        \begin{align*}
            (\mathfrak{u}_{\widetilde{c}}\Omega^d\Upsilon_\phi)^{(-1)}&=\widetilde{c}\mathfrak{u}_{\widetilde{c}}(t-\theta)^d\Omega^d\widetilde{\Phi}_\phi^{-1}\Upsilon_\phi\\
            &=\big(\widetilde{c}(t-\theta)^d\widetilde{\Phi}_\phi^{-1}\big)\big(\mathfrak{u}_{\widetilde{c}}\Omega^d\Upsilon_\phi\big)\\
            &=\widetilde{\Phi}_\phi^{\mathrm{ad}}(\mathfrak{u}_{\widetilde{c}}\Omega^d\Upsilon_\phi).
        \end{align*}
        Consider
        \[
            \Mat_{r+1}(\TT)^\sigma:=\{D\in\Mat_{r+1}(\TT)\mid D^{(-1)}=D\}=\Mat_{r+1}(\bA).
        \]
        It is straightforward to see that
        \[
            (B^{-1}\mathfrak{u}_{\widetilde{c}}\Omega^d\Upsilon_\phi)^{(-1)}=B^{-1}(\widetilde{\Phi}_\phi^{\mathrm{ad}})^{-1}\mathfrak{u}_{\widetilde{c}}\widetilde{\Phi}_\phi^{\mathrm{ad}}\Omega^d\Upsilon_\phi=B^{-1}\mathfrak{u}_{\widetilde{c}}\Omega^d\Upsilon_\phi,
        \]
        which implies that $B^{-1}\mathfrak{u}_{\widetilde{c}}\Omega^d\Upsilon_\phi\in\Mat_{r+1}(\TT)^\sigma=\Mat_{r+1}(\bA)$. In particular, there exists $C\in\Mat_{r+1}(\bA)$ such that
        \[
            \mathfrak{u}_{\widetilde{c}}\Omega^d\Upsilon_\phi=BC.
        \]
    \end{proof}

    By Proposition~\ref{Prop:First_Difference_Equation}, if we can construct $B$ satisfying $B^{(-1)}=\wPhi_\phi^{\ad} B$ without using $\Upsilon_\phi$, then it will produce non-trivial relations involving $\Upsilon_\delta$. This will be our major goal in the rest of this section.

\subsection{Frobenius difference equation from dual \texorpdfstring{$t$}{t}-motives}
    In this subsection, we investigate another way to construct the solution of the Frobenius difference equation
    \[
        \mathbf{X}^{(-1)}=\widetilde{\Phi}_\phi^{\mathrm{ad}}\mathbf{X}.
    \]
    To begin with, consider
    \[
        V_\phi:=\begin{pmatrix}
            V_G & \\
             & 1
        \end{pmatrix}\GL_{r+1}(\mathbb{K}[t]).
    \]
    We define
    \[
        \Phi_\phi:=\left(V_\phi^{-1}\right)^{(-1)}\widetilde{\Phi}_\phi^\tr V_\phi=\begin{pmatrix}
            \Phi_G & (V_G^{-1})^{(-1)}\widetilde{\Phi}_\delta^\tr\\
            0 & (t-\theta)^n
        \end{pmatrix}\in\Mat_{r+1}(\mathbb{K}[t])\cap\GL_{r+1}(\mathbb{K}(t)).
    \]
    By the definition of $\Phi_\phi$, we get
    \[
        \left(V_\phi\right)^{(-1)}\Phi_\phi=\widetilde{\Phi}_\phi^\tr V_\phi.
    \]
    By using \eqref{Eq:Dagger_Operator} and the fact that $\det (V_\phi)^{(-1)}\det(\Phi_\phi)=\det(\widetilde{\Phi}_\phi^\tr)\det (V_\phi)$, we obtain
    \[
        \Cof(V_\phi)^{(-1)}\mathrm{Cof}(\Phi_\phi)=\widetilde{\Phi}_\phi^{\mathrm{ad}}\Cof(V_\phi).
    \]
    The following lemma allows us to switch our study from the $t$-motive side to the dual $t$-motive side.
    \begin{lemma}\label{Lem:Reduction}
        Let $\mathbf{Y}=\Psi$ be a solution of the following Frobenius difference equation
        \[
            \mathbf{Y}^{(-1)}=\mathrm{Cof}(\Phi_\phi)\mathbf{Y}.
        \]
        Then, $\mathbf{X}=\Cof(V_\phi)\Psi$ is a solution of
        \[
            \mathbf{X}^{(-1)}=\widetilde{\Phi}_\phi^{\mathrm{ad}}\mathbf{X}.
        \]
        Consequently, there exists $C\in\Mat_{r+1}(\bA)$ such that
        \[
            \mathfrak{u}_{\widetilde{c}}\Omega^d\Upsilon_\phi=\Cof(V_\phi)\Psi C.
        \]
    \end{lemma}

    \begin{proof}
        Note that
        \[
            \left(\Cof(V_\phi)\Psi\right)^{(-1)}=\Cof(V_\phi)^{(-1)}\mathrm{Cof}(\Phi_\phi)\Psi=\widetilde{\Phi}_\phi^{\mathrm{ad}}\left(\Cof(V_\phi)\Psi\right).
        \]
        Then, the desired result follows from Proposition~\ref{Prop:First_Difference_Equation}.
    \end{proof}
    
    In what follows, we aim to construct a solution $\mathbf{Y}=\Psi$ to the difference equation
    \begin{equation}\label{Eq:Dual_Difference_Equation}
        \mathbf{Y}^{(-1)}=\mathrm{Cof}(\Phi_\phi)\mathbf{Y}.
    \end{equation}
    Assume that $\det(\Phi_G)=c(t-\theta)^d$ for some $c\in\KK^\times$. Since
    \[
        \det(\Phi_\phi)=\det(\Phi_G)(t-\theta)^n=c(t-\theta)^{d+n},
    \]
    we notice that
    \begin{equation}\label{Eq:Second_Difference_Equation}
        \mathrm{Cof}(\Phi_\phi)=\det(\Phi_\phi)\big(\Phi_\phi^{-1}\big)^{\tr}=\begin{pmatrix}
            (t-\theta)^{n}\mathrm{Cof}(\Phi_G) & \\
            \mathfrak{h} & c(t-\theta)^{d}
        \end{pmatrix}
    \end{equation}
    where $\mathfrak{h}=-\widetilde{\Phi}_\delta\big((V_G^{-1})^\tr\big)^{(-1)}\mathrm{Cof}(\Phi_E)\in\Mat_{1\times r}(\mathbb{K}[t])$. Then, to solve the difference equation \eqref{Eq:Dual_Difference_Equation}, we aim to find $\mathfrak{g}\in\Mat_{1\times r}(\TT)$ so that
    \begin{equation}\label{Eq:Second_Wired_Difference_Equation}
        \mathfrak{g}^{(-1)}(t-\theta)^{n}\mathrm{Cof}(\Phi_E)=c(t-\theta)^{d}\mathfrak{g}+\mathfrak{h}.
    \end{equation}
    Indeed, if we consider
    \[
        \Psi:=\begin{pmatrix}
            \Omega^{n}\mathrm{Cof}(\Psi_G) & \\
            \mathfrak{g}\Omega^{n}\mathrm{Cof}(\Psi_G) & \mathfrak{u}_{c}\Omega^{d}
        \end{pmatrix},
    \]
    where $\mathfrak{u}_{c}\in\mathbb{K}^\times$ is the fixed constant that satisfies \eqref{Eq:Normalization_Constant}.
    Then, it fits into the desired difference equation
    \[
        \Psi^{(-1)}=\mathrm{Cof}(\Phi_\phi)\Psi.
    \]

    The following lemma provides a way for solving \eqref{Eq:Second_Wired_Difference_Equation} in some special cases. 

    \begin{lemma}\label{Lem:Solution_Difference_Equation}
        Let $G=(\mathbb{G}_a^d,\varphi)$ be a $d$-dimensional uniformizable abelian $t$-module of rank $r$ with the associated dual $t$-motive $\mathcal{N}_E$. Let $(\iota_{\mathbf{n}_G},\Phi_G,\Psi_G)$ be a rigid analytic trivialization of $\mathcal{N}_G$. Assume that $\det(\Phi_G)=c(t-\theta)^d$ for some $c\in\KK^\times$. For any $n\geq d$, we fix $\mathbf{h}\in\Mat_{1\times r}(\mathbb{K}[t])$ and we suppose that there exists $\mathbf{g}\in\Mat_{1\times r}(\TT)$ so that
        \begin{equation}\label{Eq:Anderson_Exponential_Theorem}
            \mathbf{g}^{(-1)}(t-\theta)^{n-d}\mathrm{Cof}(\Phi_G)=\mathbf{g}+\mathbf{h}.
        \end{equation}
        Then, we have
        \[
            \begin{pmatrix}
                \Omega^{n}\mathrm{Cof}(\Psi_G) & \\
                \mathfrak{u}_c\mathbf{g}\Omega^{n}\mathrm{Cof}(\Psi_G) & \mathfrak{u}_c\Omega^{d}
            \end{pmatrix}^{(-1)}=\begin{pmatrix}
                (t-\theta)^{n}\mathrm{Cof}(\Phi_G) & \\
                c\mathfrak{u}_c(t-\theta)^{d}\mathbf{h} & c(t-\theta)^d
            \end{pmatrix}\begin{pmatrix}
                \Omega^{n}\mathrm{Cof}(\Psi_G) & \\
                \mathfrak{u}_c\mathbf{g}\Omega^{n}\mathrm{Cof}(\Psi_E) & \mathfrak{u}_c\Omega^{d}
            \end{pmatrix},
        \]
        where $\mathfrak{u}_{c}\in\mathbb{K}^\times$ is a fixed constant that satisfies \eqref{Eq:Normalization_Constant}.
    \end{lemma}

    \begin{proof}
        The desired result follows by checking the bottom row.
        \begin{align*}
            \big(\mathfrak{u}_c\mathbf{g}\Omega^{n}\mathrm{Cof}(\Psi_G),\mathfrak{u}_c\Omega^{d}\big)^{(-1)}&=\big(c\mathfrak{u}_c\mathbf{g}^{(-1)}(t-\theta)^{n}\Omega^{n}\mathrm{Cof}(\Phi_G)\mathrm{Cof}(\Psi_G),c\mathfrak{u}_c(t-\theta)^d\Omega^d\big)\\
            &=\big(c\mathfrak{u}_c(\mathbf{g}+\mathbf{h})(t-\theta)^{d}\Omega^{n}\mathrm{Cof}(\Psi_G),c\mathfrak{u}_c(t-\theta)^d\Omega^d\big)\\
            &=\big(c\mathfrak{u}_c(t-\theta)^{d}\mathbf{h},c(t-\theta)^d\big)\begin{pmatrix}
                \Omega^{n}\mathrm{Cof}(\Psi_G) & \\
                \mathfrak{u}_c\mathbf{g}\Omega^{n}\mathrm{Cof}(\Psi_G) & \mathfrak{u}_c\Omega^{d}
            \end{pmatrix}.
        \end{align*}
    \end{proof}

    One of the obstructions to applying Lemma~\ref{Lem:Solution_Difference_Equation} is the existence of $\mathbf{g}$ satisfying \eqref{Eq:Anderson_Exponential_Theorem}. In what follows, we propose a general construction for solving \eqref{Eq:Anderson_Exponential_Theorem}. We begin with introducing a variant of Furusho's $\wp$-function. The following lemma is essentially the same as \cite[Lem.~1.1.1]{Fur22} (cf. \cite[Lem.~3.2.1]{Che22}).

    \begin{lemma}
        Let
        \begin{align*}
            \wp_r:=\big(\sigma-1\big):\Mat_{1\times r}(\TT)&\to\Mat_{1\times r}(\TT)\\
            (Z_1,\dots,Z_r)&\mapsto(Z_1^{(-1)}-Z_1,\dots,Z_r^{(-1)}-Z_r).
        \end{align*}
        Then, $\wp_r$ is surjective with kernel $\Mat_{1\times r}(\mathbb{F}_q[t])$. Moreover, $\mathbf{f}$ and $\wp_r(\mathbf{f})$ have the same radius of convergence for any $\mathbf{f}\in\Mat_{1\times r}(\TT)$.
    \end{lemma}

    Since $\wp_r$ is surjective with kernel $\Mat_{1\times r}(\mathbb{F}_q[t])$, we may consider the following inverse image map
    \begin{align*}
        \mathscr{L}:\Mat_{1\times r}(\TT)&\to\Mat_{1\times r}(\TT)/\Mat_{1\times r}(\mathbb{F}_q[t])\\
        \mathbf{f}&\mapsto\wp_r^{-1}(\mathbf{f}).
    \end{align*}
    Note that whenever $\|\mathbf{f}\|<1$, 
    we must have
    \begin{equation}\label{Eq:Special_Case}
        \mathscr{L}(\mathbf{f})=\wp_r^{-1}(\mathbf{f})=\sum_{m\geq 1}\mathbf{f}^{(m)}+\Mat_{1\times r}(\mathbb{F}_q[t]).
    \end{equation}
    \begin{proposition}\label{Prop:Continuation}
        Let $G=(\mathbb{G}_a^d,\varphi)$ be a $d$-dimensional uniformizable abelian $t$-module of rank $r$ with the associated dual $t$-motive $\mathcal{N}_E$. Let $(\iota_{\mathbf{n}_G},\Phi_G,\Psi_G)$ be a rigid analytic trivialization of $\mathcal{N}_G$. Assume that $\det(\Phi_G)=c(t-\theta)^d$ for some $c\in\KK^\times$. For any $n\geq d$, we fix $\mathbf{h}\in\Mat_{1\times r}(\mathbb{K}[t])$. Then, for any $\mathbf{u}\in\mathscr{L}\big(\mathbf{h}\Omega^{n-d}\mathrm{Cof}(\Psi_G)\big)$, we have that $\mathbf{g}:=\mathbf{u}\Omega^{d-n}\mathrm{Cof}(\Psi_G)^{-1}$ satisfies the Frobenius difference equation \eqref{Eq:Anderson_Exponential_Theorem}. In particular, if $\|\mathbf{h}\Omega^{n-d}\mathrm{Cof}(\Psi_G)\|<1$, then we can simply choose
        \begin{align*}
            \mathbf{g}:&=\bigg(\sum_{m\geq 1}\big(\mathbf{h}\Omega^{n-d}\mathrm{Cof}(\Psi_G)\big)^{(m)}\bigg)\Omega^{d-n}\mathrm{Cof}(\Psi_G)^{-1}\\
            &=\sum_{m\geq 1}\frac{\mathbf{h}^{(m)}\big(\Phi_G^\tr\big)^{(m)}\cdots\big(\Phi_G^\tr\big)^{(1)}}{c^{q(q^m-1)/(q-1)}(t-\theta^{q^m})^{n}\cdots(t-\theta^q)^{n}}.
        \end{align*}
    \end{proposition}

    \begin{proof}
        Let $\mathbf{u}\in\mathscr{L}(\mathbf{h}\Omega^{n-d}\mathrm{Cof}(\Psi_G))$. Then, we must have
        \[
            \wp_r(\mathbf{u})=\mathbf{u}^{(-1)}-\mathbf{u}=\mathbf{h}\Omega^{n-d}\mathrm{Cof}(\Psi_G).
        \]
        In particular, we have $\mathbf{u}^{(-1)}=\mathbf{u}+\mathbf{h}\Omega^{n-d}\mathrm{Cof}(\Psi_G)$. It follows that
        \begin{align*}
            \big(\mathbf{u}\Omega^{d-n}\mathrm{Cof}(\Psi_G)^{-1}\big)^{(-1)}(t-\theta)^{n-d}\mathrm{Cof}(\Phi_G)&=\big(\mathbf{u}+\mathbf{h}\Omega^{n-d}\mathrm{Cof}(\Psi_G)\big)\Omega^{d-n}\mathrm{Cof}(\Psi_E)^{-1}\\
            &=\big(\mathbf{u}\Omega^{d-n}\mathrm{Cof}(\Psi_G)^{-1}\big)+\mathbf{h}.
        \end{align*}
        Therefore, $\mathbf{g}:=\mathbf{u}\Omega^{d-n}\mathrm{Cof}(\Psi_G)^{-1}$ provides a desired solution of the difference equation \eqref{Eq:Anderson_Exponential_Theorem}. The special case $\|\mathbf{h}\Omega^{n-d}\mathrm{Cof}(\Psi_G)\|<1$ follows immediately from \eqref{Eq:Special_Case}.
    \end{proof}

    Motivated by Proposition~\ref{Prop:Continuation}, we formulate a special class of extensions as follows.

    \begin{definition}\label{Def:Special_Extensions}
        Let $G=(\mathbb{G}_a^d,\varphi)$ be a $d$-dimensional uniformizable abelian $t$-module of rank $r$. Let  $n\geq d$. For $\delta \in \mathrm{Der}(\varphi, [\cdot]_n)$, recall $\wPhi_\delta \in \Mat_{1\times r}(\KK[t])$ from \eqref{Eq:Phi_delta}.  We define
        \[
            \mathrm{Der}_{\mathrm{sp}}(\varphi,[\cdot]_n):=\{\delta\in\mathrm{Der}(\varphi,[\cdot]_n)\mid\widetilde{\Phi}_\delta\in\Mat_{1\times r}(\mathbb{K}[t])\widetilde{\Phi}_G\}.
        \]
        We further set $\Ext^1_{\mathrm{sp}}(\varphi,[\cdot]_n):=\mathrm{Der}_{\mathrm{sp}}(\varphi,[\cdot]_n)/\big(\mathrm{Der}_{\mathrm{sp}}(\varphi,[\cdot]_n)\cap\mathrm{Der}_{\mathrm{in}}(\varphi,[\cdot]_n)\big)$ to be the $\mathbb{F}_q[t]$-submodule of $\Ext^1(\varphi,[\cdot]_n)$ generated by $\mathrm{Der}_{\mathrm{sp}}(\varphi,[\cdot]_n)$. We call the equivalence classes in $\Ext^1_{\mathrm{sp}}(\varphi,[\cdot]_n)$ special extensions.
    \end{definition}

    \begin{remark}\label{R:whatish}
        For $\delta\in\mathrm{Der}_{\mathrm{sp}}(\varphi,[\cdot]_n)$ and any $c\in\KK^\times$, set $\mathbf{h}:=-\mathfrak{u}_c^{-1}\widetilde{\Phi}_\delta\widetilde{\Phi}_G^{-1}(V_G^{-1})^\tr\in\Mat_{1\times r}(\mathbb{K}[t])$. Since $\widetilde{\Phi}_\delta\in\Mat_{1\times r}(\mathbb{K}[t])\widetilde{\Phi}_G$, we have that $\mathbf{h}$ has integral entries. This guarantees the assumption of Theorem~\ref{Thm:Identity_among_AGF}.
    \end{remark}
    
    Combining all the observations we have so far, we can produce identities involving Anderson generating functions for some special extensions.

    \begin{theorem}\label{Thm:Identity_among_AGF}
        Let $G=(\mathbb{G}_a^d,\varphi)$ be a $d$-dimensional uniformizable abelian $t$-module of rank $r$ with the associated dual $t$-motive $\mathcal{N}_E$. Let $(\iota_{\mathbf{n}_G},\Phi_G,\Psi_G)$ be a rigid analytic trivialization of $\mathcal{N}_G$. Assume that $\det(\Phi_G)=c(t-\theta)^d$ for some $c\in\KK^\times$. For any $n\geq d$, consider $\delta\in\mathrm{Der}_{\mathrm{sp}}(\varphi,[\cdot]_n)$ and its induced $t$-module $G_\delta=(\mathbb{G}_a^{d+n},\phi)$ defined in \eqref{Eq:t-module_from_delta}. Let $\mathbf{h}:=-\mathfrak{u}_c^{-1}\widetilde{\Phi}_\delta\widetilde{\Phi}_G^{-1}(V_G^{-1})^\tr\in\Mat_{1\times r}(\mathbb{K}[t])$. Then, for any $\mathbf{u}\in\mathscr{L}\big(\mathbf{h}\Omega^{n-d}\mathrm{Cof}(\Psi_G)\big)$, we have that $\mathbf{g}:=\mathbf{u}\Omega^{d-n}\mathrm{Cof}(\Psi_G)^{-1}$
        satisfies \eqref{Eq:Anderson_Exponential_Theorem}. Moreover, there exists
        \[
            C=\begin{pmatrix}
            \mathbb{I}_r & \\
            a_1,\dots,a_r & 1
        \end{pmatrix}\in\Mat_{r+1}(\mathbf{A})
        \]
        so that
        \[
            \big(\Omega^{n}g_{1,1}(t;\phi)^{(1)},\dots,\Omega^{n}g_{r,1}(t;\phi)^{(1)},1\big)=(\mathfrak{u}_c\mathbf{g}\Omega^{n}V_G^\tr\Upsilon_G,1)C,
        \]
        where $g_{1,1}(t;\phi),\dots,g_{r,1}(t;\phi)$ are entries of $\Upsilon_\delta$ defined in \eqref{Eq:Upsilon_Delta}. If we further have
        \[
            \|\mathbf{h}\Omega^{n-d}\mathrm{Cof}(\Psi_G)\|<1,
        \]
        then
        \[
            \big(\Omega^{n}g_{1,1}(t;\phi)^{(1)},\dots,\Omega^{n}g_{r,1}(t;\phi)^{(1)},1\big)=(\mathfrak{u}_c\bigg(\sum_{m\geq 1}\big(\mathbf{h}\Omega^{n-d}\mathrm{Cof}(\Psi_G)\big)^{(m)}\bigg),1)C.
        \]
    \end{theorem}

    \begin{proof}
        The existence of $\mathbf{g}$ is guaranteed by Proposition~\ref{Prop:Continuation}.
        By Lemma~\ref{Lem:Solution_Difference_Equation}, we have
        \[
            \begin{pmatrix}
                \Omega^{n}\mathrm{Cof}(\Psi_G) & \\
                \mathfrak{u}_c\mathbf{g}\Omega^{n}\mathrm{Cof}(\Psi_G) & \mathfrak{u}_c\Omega^{d}
            \end{pmatrix}^{(-1)}=\begin{pmatrix}
                (t-\theta)^{n}\mathrm{Cof}(\Phi_G) & \\
                c\mathfrak{u}_c(t-\theta)^{d}\mathbf{h} & c(t-\theta)^d
            \end{pmatrix}\begin{pmatrix}
                \Omega^{n}\mathrm{Cof}(\Psi_G) & \\
                \mathfrak{u}_c\mathbf{g}\Omega^{n}\mathrm{Cof}(\Psi_G) & \mathfrak{u}_c\Omega^{d}
            \end{pmatrix}.
        \]
        Note that
        \begin{align*}
            c\mathfrak{u}_c(t-\theta)^d\mathbf{h}&=c\mathfrak{u}_c(t-\theta)^d\left(-\mathfrak{u}_c^{-1}\widetilde{\Phi}_\delta\widetilde{\Phi}_G^{-1}(V_G^{-1})^\tr\right)\\
            &=-c(t-\theta)^d\widetilde{\Phi}_\delta\widetilde{\Phi}_G^{-1}(V_G^{-1})^\tr\\
            &=-c(t-\theta)^d\widetilde{\Phi}_\delta\left((V_G^{-1})^\tr\right)^{(-1)}(\Phi_G^{-1})^\tr\\
            &=-\widetilde{\Phi}_\delta\left((V_G^{-1})^\tr\right)^{(-1)}\mathrm{Cof}(\Phi_G),
        \end{align*}
        where the third equality comes from the relation $\left((V_G^{-1})^\tr\right)^{(-1)}(\Phi_G^{-1})^\tr=\widetilde{\Phi}_G^{-1}(V_G^{-1})^\tr$ and the last identity follows from the fact that $\mathrm{Cof}(\Phi_G)=\det(\Phi_G)(\Phi_G^{-1})^\tr=c(t-\theta)^d(\Phi_G^{-1})^\tr$. Hence $c\mathfrak{u}_c(t-\theta)^d\mathbf{h}$ coincides with $\mathfrak{h}$ defined in \eqref{Eq:Second_Difference_Equation}. In other words,
        \[
            \mathbf{Y}=\begin{pmatrix}
                \Omega^{n}\mathrm{Cof}(\Psi_G) & \\
                \mathfrak{u}_c\mathbf{g}\Omega^{n}\mathrm{Cof}(\Psi_G) & \mathfrak{u}_c\Omega^{d}
            \end{pmatrix}
        \]
        is a solution to the difference equation
        \[
            \mathbf{Y}^{(-1)}=\mathrm{Cof}(\Phi_\phi)\mathbf{Y}.
        \]
        Thus, by Proposition~\ref{Prop:First_Difference_Equation} and Lemma~\ref{Lem:Reduction}, there exists $C\in\Mat_{r+1}(\mathbf{A})$ so that
        \[
            \mathfrak{u}_{\widetilde{c}}\Omega^{d+n}\begin{pmatrix}
                \Upsilon_G & \\
                \Upsilon_\delta & \Omega^{-n}
            \end{pmatrix}=\begin{pmatrix}
                \mathrm{Cof}(V_G) & \\
                 & \det(V_G)
            \end{pmatrix}\begin{pmatrix}
                \Omega^{n}\mathrm{Cof}(\Psi_G) & \\
                \mathfrak{u}_c\mathbf{g}\Omega^{n}\mathrm{Cof}(\Psi_G) & \mathfrak{u}_c\Omega^{d}
            \end{pmatrix}C,
        \]
        where we recall that $\det(\Phi_G)=c(t-\theta)^d$ (resp. $\det(\widetilde{\Phi}_G)=\widetilde{c}(t-\theta)^d$) for some $c\in\KK^\times$ (resp. $\widetilde{c}\in\KK^\times$) and $\mathfrak{u}_c$ (resp. $\mathfrak{u}_{\widetilde{c}}$) satisfies \eqref{Eq:Normalization_Constant}.

        It is clear to see that $C\in\Mat_{r+1}(\mathbf{A})$ is of the form
        \[
            C=\begin{pmatrix}
                \mathbb{I}_r & \\
                a_1,\dots,a_r & 1
            \end{pmatrix}.
        \]
        Furthermore, by comparing the last rows on the both sides and eliminating the common factor $\Omega^d$, we deduce that
        \[
           \mathfrak{u}_{\widetilde{c}}\big(\Omega^{n}g_{1,1}(t;\phi)^{(1)},\dots,\Omega^{n}g_{r,1}(t;\phi)^{(1)},1\big)=\mathfrak{u}_c\det(V_G)\big(\mathbf{g}\Omega^{n-d}\mathrm{Cof}(\Psi_G),1\big)C.
        \]
        Note that the relation $\left(V_\phi\right)^{(-1)}\Phi_\phi=\widetilde{\Phi}_\phi^\tr V_\phi$ implies that $\det(V_\phi)^{(-1)}c=\widetilde{c}\det(V_\phi)$, and thus $\mathfrak{u}_{\widetilde{c}}=\det(V_E)\mathfrak{u}_c$. Then, the above equality becomes
        \[
            \big(\Omega^{n}g_{1,1}(t;\phi)^{(1)},\dots,\Omega^{n}g_{r,1}(t;\phi)^{(1)},1\big)=(\mathbf{g}\Omega^{n-d}\mathrm{Cof}(\Psi_G),1)C.
        \]
        Finally, since $\Psi_G^{-1}=\Upsilon_G^\tr V_G$ and $\det(\Psi_G)=\mathfrak{u}_c\Omega^d$, 
          it follows that $\mathrm{Cof}(\Psi_G)=\mathfrak{u}_c\Omega^dV_G^\tr\Upsilon_G$. If we apply this relation to the above equation, we derive that
        \[
            \big(\Omega^{n}g_{1,1}(t;\phi)^{(1)},\dots,\Omega^{n}g_{r,1}(t;\phi)^{(1)},1\big)=(\mathfrak{u}_c\mathbf{g}\Omega^{n}V_G^\tr\Upsilon_G,1)C,
        \]
        which gives the desired result. The situation of $\|\mathbf{h}\Omega^{n-d}\mathrm{Cof}(\Psi_G)\|<1$ follows immediately from the special case of Proposition~\ref{Prop:Continuation}. 
    \end{proof}

\subsection{Constructions from Anderson generating functions}
    To apply Theorem~\ref{Thm:Identity_among_AGF} to produce relations involving certain coordinates of the periods of extensions coming from $\delta\in\mathrm{Der}_{\mathrm{sp}}(\varphi,[\cdot]_n)$, we need to construct the solution $\mathbf{g}$ of the difference equation \eqref{Eq:Anderson_Exponential_Theorem} in a specific way. This can be achieved by using Anderson generating functions and related constructions developed in \cite{NPapanikolas21} for a special family of $t$-modules. To be more precise, we call a $t$-module $G=(\mathbb{G}_a^d,\varphi)$ \emph{almost strictly pure} if the top coefficient of $\varphi_{t^s}$ is invertible for some $s\geq 1$. 
    
    \begin{proposition}\label{Prop:Identity_among_AGF_2}
        Let $G=(\mathbb{G}_a^d,\varphi)$ be a $d$-dimensional uniformizable abelian $t$-module of rank $r$ defined over $\oK$ with associated dual $t$-motive $\mathcal{N}_G$. Let $(\iota_{\mathbf{n}_G},\Phi_G,\Psi_G)$ be a rigid analytic trivialization of $\mathcal{N}_G$. For any $n\geq d$, consider $\delta\in\mathrm{Der}_{\mathrm{sp}}(\varphi,[\cdot]_n)$ and its induced $t$-module $G_\delta=(\mathbb{G}_a^{d+n},\phi)$ as defined in \eqref{Eq:t-module_from_delta}. Let $\mathbf{h}:=-\mathfrak{u}_c^{-1}\widetilde{\Phi}_\delta\widetilde{\Phi}_G^{-1}(V_G^{-1})^\tr\in\Mat_{1\times r}(\mathbb{K}[t])$. Assume that $H=(\mathbb{G}_a^s,\varrho)$ is a uniformizable almost strictly pure $t$-module defined over $\oK$ together with a $t$-frame $(\iota_\mathbf{n},(t-\theta)^{n-d}\mathrm{Cof}(\Phi_G))$ for its associated dual $t$-motive $\mathcal{N}_H$. Then, there exists $\mathbf{g}\in\Mat_{1\times r}(\TT)$ satisfying \eqref{Eq:Anderson_Exponential_Theorem} and an extra property that
        \[
            \mathrm{Span}_{\overline{K}}\bigg(1,\mathrm{F}_{\bm{\epsilon}}(\bm{y}_{\mathbf{h}})\mid\bm{\epsilon}\in\mathrm{Der}_\partial(\varrho,[\cdot]_0),~\bm{\epsilon}(t)\in\Mat_{1\times s}(\overline{K}[\tau])\tau\bigg)=\mathrm{Span}_{\overline{K}}\bigg(\{1\}\cup\{\mathbf{g}|_{t=\theta}\}\bigg)
        \]
        for some $\bm{y}_{\mathbf{h}}\in(\Lie H)(\mathbb{C}_\infty)$ with $\Exp_{\varrho}(\bm{y}_{\mathbf{h}})\in H(\overline{K})$. Moreover, there exists $C\in\Mat_{r+1}(\mathbf{A})$ of the form
        \[
            C=\begin{pmatrix}
            \mathbb{I}_r & \\
            a_1,\dots,a_r & 1
        \end{pmatrix}
        \]
        so that
        \[
            \big(\Omega^{n}g_{1,1}(t;\phi)^{(1)},\dots,\Omega^{n}g_{r,1}(t;\phi)^{(1)},1\big)=(\mathfrak{u}_c\mathbf{g}\Omega^{n}V_G^\tr\Upsilon_G,1)C.
        \]
    \end{proposition}

    \begin{proof} 
        We set $\mathscr{M}=\Mat_{1\times r}(\oK[t])$ to be the $\oK[t,\sigma]$-module with $\sigma$-action on the standard $\oK[t]$-basis given by $(t-\theta)^{n-d}\mathrm{Cof}(\Phi_G)$. It is known due to Anderson (see \cite[Prop.~2.5.8]{HJ20} and \cite[Lem.~3.1.2]{NPapanikolas21}) that we have the $\mathbb{F}_q[t]$-module isomorphism
        \begin{align*}
\epsilon_1\circ\iota_{\mathbf{n}}:\mathscr{M}/(\sigma-1)\mathscr{M}&\to H(\oK)\\
            \mathbf{h}&\mapsto\epsilon_1\circ\iota_{\mathbf{n}}(\mathbf{h})=:\bm{\alpha}_\mathbf{h},
        \end{align*}
        where $\epsilon_1:\Mat_{1\times s}(\oK[\sigma])\to\Mat_{s\times 1}(\oK)$ is defined precisely by $\epsilon_1(\sum_{i=0}^\ell\bm{\alpha}_i\sigma^i):=\big(\sum_{i=0}^\ell\bm{\alpha}^{(i)}\big)^\tr$.
        Recall that $\mathbf{h}=-\mathfrak{u}_c^{-1}\widetilde{\Phi}_\delta\widetilde{\Phi}_G^{-1}(V_G^{-1})^\tr\in\Mat_{1\times r}(\mathbb{K}[t])$. Since $H$ is uniformizable, we can choose $\bm{y}_{\mathbf{h}}\in(\Lie H)(\mathbb{C}_\infty)$ so that $\Exp_{\varrho}(\bm{y}_{\mathbf{h}})=\bm{\alpha}_{\mathbf{h}}$.
        On the other hand, starting from the pair
        \[
            (\bm{\alpha}_{\mathbf{h}},\bm{y}_{\mathbf{h}})\in H(\oK)\times(\Lie H)(\mathbb{C}_\infty),
        \]
        \cite[Lemma~4.4.19]{NPapanikolas21} provides an explicit construction of $\bm{h}_{{\bm{\alpha}_{\mathbf{h}}}}$ and $\bm{g}_{\bm{y}_{\mathbf{h}}}$ so that
        \[
            \bm{g}_{\bm{y}_{\mathbf{h}}}^{(-1)}(t-\theta)^{n-d}\mathrm{Cof}(\Phi_G)=\bm{g}_{\bm{y}_{\mathbf{h}}}+\bm{h}_{\bm{\alpha}_{\mathbf{h}}}.
        \]
        Furthermore, by our assumption that $H$ is almost strictly pure, \cite[Proposition~4.5.22]{NPapanikolas21} shows that $\epsilon_1\circ\iota_{\mathbf{n}}(\bm{h}_{\bm{\alpha}_{\mathbf{h}}})=\bm{\alpha}_{\mathbf{h}}=\epsilon_1\circ\iota_{\mathbf{n}}(\mathbf{h})$. In other words, there exists $\mathbf{u}_\mathbf{h}\in\Mat_{1\times r}(\oK[t])$ so that
        \[
            \bm{h}_{\bm{\alpha}_{\mathbf{h}}}-\mathbf{h}=\mathbf{u}_\mathbf{h}^{(-1)}(t-\theta)^{n-d}\mathrm{Cof}(\Phi_G)-\mathbf{u}_\mathbf{h}
        \]
        If we define $\mathbf{g}:=\bm{g}_{\bm{y}_{\mathbf{h}}}-\mathbf{u}_{\mathbf{h}}$, then we can verify directly that
        \begin{align*}
            \mathbf{g}^{(-1)}(t-\theta)^{n-d}\mathrm{Cof}(\Phi_G)&=\bm{g}_{\bm{y}_{\mathbf{h}}}^{(-1)}(t-\theta)^{n-d}\mathrm{Cof}(\Phi_G)-\mathbf{u}_{\mathbf{h}}^{(-1)}(t-\theta)^{n-d}\mathrm{Cof}(\Phi_G)\\
            &=(\bm{g}_{\bm{y}_{\mathbf{h}}}+\bm{h}_{\bm{\alpha}_{\mathbf{h}}})-(\bm{h}_{\bm{\alpha}_{\mathbf{h}}}-\mathbf{h}+\mathbf{u}_{\mathbf{h}})\\
            &=\mathbf{g}+\mathbf{h}.
        \end{align*}
        Assume that $\det(\Phi_G)=c(t-\theta)^d$ for some $c\in\oK^\times$. By the same argument as in the proof of Theorem~\ref{Thm:Identity_among_AGF}, Proposition~\ref{Prop:First_Difference_Equation} and Lemma~\ref{Lem:Reduction} imply that there exists $C\in\Mat_{r+1}(\mathbf{A})$ of the form
        \[
            C=\begin{pmatrix}
            \mathbb{I}_r & \\
            a_1,\dots,a_r & 1
        \end{pmatrix}
        \]
        so that
        \[
            \big(\Omega^{n}g_{1,1}(t;\phi)^{(1)},\dots,\Omega^{n}g_{r,1}(t;\phi)^{(1)},1\big)=(\mathfrak{u}_c\mathbf{g}\Omega^{n}V_G^\tr\Upsilon_G,1)C,
        \]
        Finally, since $H$ is defined over $\overline{K}$, by \cite[Thm.~4.4.30]{NPapanikolas21} we derive that
        \[
            \mathrm{Span}_{\overline{K}}\bigg(1,\mathrm{F}_{\bm{\epsilon}}(\bm{y}_{\mathbf{h}})\mid\bm{\epsilon}\in\mathrm{Der}_\partial(\varrho,[\cdot]_0),~\bm{\epsilon}(t)\in\Mat_{1\times s}(\overline{K}[\tau])\tau\bigg)=\mathrm{Span}_{\overline{K}}\bigg(\{1\}\cup\{\bm{g}_{\bm{y}_{\mathbf{h}}}|_{t=\theta}\}\bigg).
        \]
        Thus, from the definition $\mathbf{g}=\bm{g}_{\bm{y}_{\mathbf{h}}}-\mathbf{u}_{\mathbf{h}}$ for some $\mathbf{u}_{\mathbf{h}}\in\Mat_{1\times r}(\overline{K}[t])$, we deduce the desired result
        \[
            \mathrm{Span}_{\overline{K}}\bigg(1,\mathrm{F}_{\bm{\epsilon}}(\bm{y}_{\mathbf{h}})\mid\bm{\epsilon}\in\mathrm{Der}_\partial(\varrho,[\cdot]_0),~\bm{\epsilon}(t)\in\Mat_{1\times s}(\overline{K}[\tau])\tau\bigg)=\mathrm{Span}_{\overline{K}}\bigg(\{1\}\cup\{\mathbf{g}|_{t=\theta}\}\bigg).
        \]
    \end{proof}

    \begin{remark}\label{Rem:Subtle_Points}
        Let $G=(\mathbb{G}_a^d,\varphi)$ be a $d$-dimensional uniformizable abelian $t$-module of rank $r$ defined over $\oK$ with a rigid analytic trivialization of its associated dual $t$-motive $\mathcal{N}_G$ given by $(\iota_{\mathbf{n}_G},\Phi_G,\Psi_G)$. If we consider the $\oK[t,\sigma]$-module $\mathcal{N}_{\mathbf{C}^{\otimes (n-d)}}\otimes\wedge^{r-1}\mathcal{N}_G:=\mathcal{N}_{\mathbf{C}^{\otimes (n-d)}}\otimes_{\oK[t]}\big(\mathcal{N}_G\wedge_{\oK[t]}\wedge\cdots\wedge_{\oK[t]}\mathcal{N}_G\big)$ with an evident choice of a $\oK[t]$-basis, then one can verify directly that $\mathcal{N}_{\mathbf{C}^{\otimes (n-d)}}\otimes\wedge^{r-1}\mathcal{N}_G$ is free of finite rank over $\oK[\sigma]$, and it defines a dual $t$-motive. Thus, we always have a $t$-module $H=(\mathbb{G}_a^s,\varrho)$ with a $t$-frame $(\iota_\mathbf{n},(t-\theta)^{n-d}\mathrm{Cof}(\Phi_G))$ for its dual $t$-motive $\mathcal{N}_H$. 
        
        It is not clear to the authors if there is an elegant way to determine the almost strictly pureness of  $H$ without calculating the top coefficient of $\varrho_{t^s}$. Moreover, if we replace the dual $t$-motive $\mathcal{N}_G$ by the $t$-motive $\mathcal{M}_G$ in the above construction, then we will see that the $\oK[t,\tau]$-module $\mathcal{M}_{\mathbf{C}^{\otimes (n-d)}}\otimes\wedge^{r-1}\mathcal{M}_G:=\mathcal{M}_{\mathbf{C}^{\otimes (n-d)}}\otimes\big(\mathcal{M}_G\wedge_{\oK[t]}\wedge\cdots\wedge_{\oK[t]}\mathcal{M}_G\big)$ defines a $t$-motive. We denote its corresponding $t$-module by $\mathbf{C}^{\otimes (n-d)}\otimes\wedge^{r-1}G$. At the writing of this paper, the authors do not know whether $\mathbf{C}^{\otimes (n-d)}\otimes\wedge^{r-1}G\cong H$ as a $t$-module except the case when $G$ is a Drinfeld module of rank $r\geq 2$ (see Example~\ref{Ex:Exterior_Powers} for more details). Note that this question is equivalent to asking if $\mathbf{C}^{\otimes (n-d)}\otimes\wedge^{r-1}G$ admits a $t$-frame $(\iota_\mathbf{n},(t-\theta)^{n-d}\mathrm{Cof}(\Phi_G))$ for its associated dual $t$-motive.
    \end{remark}

    As a consequence of Proposition~\ref{Prop:Identity_among_AGF_2}, we deduce the following theorem that allows us to generate the 
    bottom coordinate of the period of $G_\delta$ in terms of periods and quasi-periods of $G$ together with logarithm and quasi-logarithm of the $t$-module $H$ at a specific algebraic point.

    \begin{theorem}\label{Thm:Generates_Trackable_Coordinate}
        Let notation be the same as in Proposition~\ref{Prop:Identity_among_AGF_2}. Consider
        \[
            S:=\{1,\mathrm{F}_{\bm{\epsilon}}(\bm{y}_{\mathbf{h}})\mid\bm{\epsilon}\in\mathrm{Der}_\partial(\varrho,[\cdot]_0),~\bm{\epsilon}(t)\in \Mat_{1\times s}(\overline{K}[\tau])\tau\}
        \]
        and
        \[
            T:=\{\mathrm{F}_{\bm{\varpi}}(\bm{\omega})\mid\bm{\varpi}\in\mathrm{Der}_\partial(\varphi,[\cdot]_0),~\bm{\varpi}(t)\in \Mat_{1\times d}(\oK[\tau])\tau,~\bm{\omega}\in\Lambda_G\}
        \]
        Assume that $\delta\in\mathrm{Der}_{\mathrm{sp}}(\varphi,[\cdot]_n)\cap\mathrm{Der}_\partial(\varphi,[\cdot]_n)$ and the period lattice of $G_\delta$ is given by 
        \[
            \Lambda_{G_\delta}=\bA\begin{pmatrix}
                \bm{\omega}_1\\
                \bm{\lambda}_{\delta,1}
            \end{pmatrix}+\cdots+\bA\begin{pmatrix}
                \bm{\omega}_r\\
                \bm{\lambda}_{\delta,r}
            \end{pmatrix}+\bA\begin{pmatrix}
                0\\
                \bm{\gamma}_{n}
            \end{pmatrix}\subset\Mat_{(n+d)\times 1}(\mathbb{C}_\infty).
        \]
        For $1\leq j\leq r$, if we express $\bm{\lambda}_{\delta,j}=(\lambda_{j,1},\dots,\lambda_{j,n})^\tr$, then we have 
        \[
            \lambda_{j,n}\in\mathrm{Span}_{\oK}\bigg(\{\Tilde{\pi}^n\}\cup\{xy\mid x\in S,~y\in T\}\bigg).
        \]
    \end{theorem}

    \begin{proof}
        Assume that $\det(\Phi_G)=c(t-\theta)^d$ for some $c\in\oK^\times$. By Proposition~\ref{Prop:Identity_among_AGF_2}, there exists $C\in\Mat_{r+1}(\mathbf{A})$ of the form
        \[
            C=\begin{pmatrix}
            \mathbb{I}_r & \\
            a_1,\dots,a_r & 1
        \end{pmatrix}
        \]
        so that
        \begin{equation}\label{Eq:Dual_Relation}
            \big(\Omega^{n}g_{1,1}(t;\phi)^{(1)},\dots,\Omega^{n}g_{r,1}(t;\phi)^{(1)},1\big)=(\mathfrak{u}_c\mathbf{g}\Omega^{n}V_G^\tr\Upsilon_G,1)C,
        \end{equation}
        where $\mathbf{g}\in\Mat_{1\times r}(\TT)$ has the property that
        \[
            \mathrm{Span}_{\overline{K}}\bigg(\{1\}\cup\{\mathbf{g}|_{t=\theta}\}\bigg)=\mathrm{Span}_{\overline{K}}(S).
        \]

        Recall that for each $1\leq j\leq r$, we have introduced the Anderson generating function of $\bm{\omega}_{\phi,j}=(\bm{\omega}_j,\bm{\lambda}_{\delta,j})^\tr$ associated to $\phi$ in \eqref{Eq:AGF_of_delta} as follows
        \[
            \cG_{\bm{\omega}_{\phi,j}}(t;\phi):=\sum\limits_{m=0}^\infty \Exp_{\phi}\left(\rd\phi_t^{-m-1}\begin{pmatrix}
            \bm{\omega}_j \\ 
            \bm{\lambda}_{\delta,j}
            \end{pmatrix}\right)t^i=\begin{pmatrix}
                \cG_{\bm{\omega}_{j}}(t;\varphi)\\
                \cF_{\bm{\omega}_{\phi,j}}(t;\phi)
            \end{pmatrix}\in \Mat_{(d+n)\times 1}(\TT).
        \]
        Note that by \cite[Prop.~4.2.12~(c)]{NPapanikolas21} we have the relation
        \[
            \left\langle\phi_t\mid\begin{pmatrix}
                \cG_{\bm{\omega}_{j}}(t;\varphi)\\
                \cF_{\bm{\omega}_{\phi,j}}(t;\phi)
            \end{pmatrix}\right\rangle=t\begin{pmatrix}
                \cG_{\bm{\omega}_{j}}(t;\varphi)\\
                \cF_{\bm{\omega}_{\phi,j}}(t;\phi)
            \end{pmatrix}.
        \]
        If we express $\cG_{\bm{\omega}_{j}}(t;\varphi)=(f_{j,1}(t;\varphi),\dots,f_{j,d}(t;\varphi))^\tr$, $\cF_{\bm{\omega}_{\phi,j}}(t;\phi)=(g_{j,1}(t;\phi),\dots,g_{j,n}(t;\phi))^\tr$ and denote the bottom row of $\delta(t)$ by $(\delta_1,\dots,\delta_d)\in\Mat_{1\times d}(\oK[\tau])\tau$, then we have
        \begin{equation}\label{Eq:Relation_AGF}
            g_{j,1}(t;\phi)^{(1)}=(t-\theta)g_{j,n}(t;\phi)-\sum_{i=1}^d\langle\delta_i\mid f_{j,i}(t;\varphi)\rangle.
        \end{equation}
        By using \eqref{Eq:Relation_AGF}, we see that \eqref{Eq:Dual_Relation} becomes
        \begin{equation}\label{Eq:Periods_Relations_from_AGFs}
            \begin{split}
                \big((t-\theta)g_{1,n}(t;\phi),\dots,(t-\theta)g_{r,n}(t;\phi),\Omega^{-n}\big)&=\\
                (\mathfrak{u}_c\mathbf{g}V_G^\tr\Upsilon_G,\Omega^{-n})C+(&\sum_{i=1}^d\langle\delta_i\mid f_{1,i}(t;\varphi)\rangle,\dots,\sum_{i=1}^d\langle\delta_i\mid f_{r,i}(t;\varphi)\rangle,0).
            \end{split}
        \end{equation}
        On the one hand, it follows from \eqref{E:AGFRes} that 
        \begin{align*}
            \big((t-\theta)g_{1,n}(t;\phi),\dots,(t-\theta)g_{r,n}(t;\phi),\Omega^{-n}\big)\mid_{t=\theta}&=(\Res_{t=\theta}g_{1,n}(t;\phi),\dots,\Res_{t=\theta}g_{r,n}(t;\phi),\widetilde{\pi}^{n})\\
            &=(-\lambda_{1,n},\dots,-\lambda_{r,n},\widetilde{\pi}^{n}).
        \end{align*}
        Here we use the fact that $\delta\in\mathrm{Der}_\partial(\varphi,[\cdot]_n)$, and thus 
        \begin{equation}\label{Eq:Inverse_Matrix}
            \bigl( \rd \phi_t - t\Id_{d+n}\bigr)^{-1}=-\begin{pmatrix}
                (t\Id_d-\rd\varphi_t)^{-1} & & & &\\
                 & (t-\theta)^{-1} & (t-\theta)^{-2} & \cdots & (t-\theta)^{-n}\\
                 & & (t-\theta)^{-1} & \ddots & \vdots\\
                 & & & \ddots & (t-\theta)^{-2}\\
                 & & & & (t-\theta)^{-1}
            \end{pmatrix}
        \end{equation}
        implies that $\mathrm{ord}_{t=\theta}\big(g_{j,n}(t;\phi)\big)=-1$ for each $1\leq j\leq r$.
        On the other hand, by \cite[Prop.~4.3.12]{NPapanikolas21}, we have
        \[
            \mathrm{Span}_{\overline{K}}\big(\Upsilon_G|_{t=\theta}\big)=\mathrm{Span}_{\overline{K}}\big( T\big).
        \]
        Moreover, for $1\leq i\leq d$ and $1\leq j\leq r$, by \cite[Prop.~4.3.5(a)]{NPapanikolas21} we have
        \[
            \langle\delta_i\mid f_{j,i}(t;\varphi)\rangle|_{t=\theta}\in\mathrm{Span}_{\overline{K}}\big( T\big).
        \]
        The desired result now follows immediately.
    \end{proof}

    \begin{remark}
The last coordinate $\lambda_{j,n}$ of $\bm{\lambda}_{\delta, j}$ in the above theorem is called  the \emph{tractable} coordinate of $\bm{\lambda}_{\delta, j}$, i.e., it lies at the bottom coordinate of the last $n \times n$ Jordan block of $\rd\phi_t$ when $\rd\phi_t$ is in Jordan normal form. Note that $\rd [t]_n$ is already a Jordan block and $\delta\in\mathrm{Der}_\partial(\varphi,[\cdot]_n)$. Thus, if $\rd\varphi_t$ is in Jordan normal form, then $\rd\phi_t$ is in Jordan normal form with the last block being $n \times n$. 
\end{remark}

\section{Periods of special extensions of Drinfeld modules}\label{S:PeriodsofExts}
    In this section, we apply Theorem~\ref{Prop:Identity_among_AGF_2} to the case of $G=E=(\mathbb{G}_a,\rho)$ a Drinfeld module of rank $r\geq 2$. We will first prove that for any $n\geq 1$
    \[
        \mathrm{Der}_\partial(\rho,[\cdot]_{n})\subset\mathrm{Der}_{\mathrm{sp}}(\rho,[\cdot]_{n}),
    \]
    where
    \[
        \mathrm{Der}_\partial(\rho,[\cdot]_{n})=\{\delta\in\mathrm{Der}(\rho,[\cdot]_{n})\mid\delta(t)\in\Mat_{n\times 1}(\mathbb{K}[\tau])\tau\}
    \]
    has been defined in \eqref{Eq:Der_0}. Then, by using Theorem~\ref{Prop:Identity_among_AGF_2}, we will show that the  tractable coordinate of the period of the extension of $E$ by $\mathbf{C}^{\otimes n}$ arising from $\delta\in\mathrm{Der}_0(\rho,[\cdot]_{n})$ can be generated by periods and quasi-periods of $E$ together with logarithms and quasi-logarithms of $\mathbf{C}^{\otimes n-1}\otimes\wedge^{r-1}E$ at some explicitly constructed algebraic points. In particular, the case of $n=1$ generalizes Chang's formula \cite[Thm.~2.4.4]{Chang13} on the third kind period of rank $2$ Drinfeld modules to arbitrary rank.
    
\subsection{Special extensions of Drinfeld modules}\label{SS:PeriodsofExts1}
    Let $E=(\mathbb{G}_a,\rho)$ be a Drinfeld module of rank $r\geq 2$ with $\rho_t=\theta+\kappa_1\tau+\cdots+\kappa_r\tau^r\in\overline{K}[\tau]$. Let $(\iota_{\mathbf{m}_E},\wPhi_E)$ be the $t$-frame of the associated $t$-motive $\mathcal{M}_E$ as defined in Example~\ref{Ex:Drinfeld_Module}. For $n\geq 1$, we consider $\delta\in\mathrm{Der}(\rho,[\cdot]_{n})$ and denote by $E_\delta=(\mathbb{G}_a^{n+1},\phi)$ the $t$-module arising from $\delta$ given in \eqref{Eq:t-module_from_delta}. Recall that by the identification given in \eqref{Eq:Phi_delta}, we have an associated vector $\wPhi_\delta\in\Mat_{1\times r}(\overline{K}[t])$.

    \begin{proposition}\label{Lem:Special_Extensions_Drinfeld_Modules}
        If we assume that $\delta\in\mathrm{Der}_\partial(\rho,[\cdot]_n)$ and express by $\widetilde{\Phi}_\delta=(\nu_1,\dots,\nu_r)\in\Mat_{1\times r}(\overline{K}[t])$, then we have
        \[
            \nu_1\in(t-\theta)\oK[t].
        \]
        Consequently, we have $\widetilde{\Phi}_\delta\in\Mat_{1\times r}(\oK[t])\wPhi_E$ and hence
        \[
            \mathrm{Der}_\partial(\rho,[\cdot]_{n})\subset\mathrm{Der}_{\mathrm{sp}}(\rho,[\cdot]_{n}).
        \]
    \end{proposition}

    \begin{proof}
        Since $\nu_1\in \oK[t]$, suppose $\nu_1=\nu_{10}+\nu_{11}(t-\theta)+\cdots+\nu_{1s}(t-\theta)^s$ for some $\nu_{1i}\in\oK$ and $s\geq 0$. We aim to show that $\nu_{10}=0$. Let $\mathbf{e}_i$ be the standard $\mathbb{K}[\tau]$-basis of the associated $t$-motive of $E_\delta$ given by $\mathcal{M}_\phi=\Mat_{1\times(n+1)}(\oK[\tau])$. Consider $\delta\in\mathrm{Der}_\partial(\rho,[\cdot]_n)$ with $\delta(t)=(\delta_1,\dots,\delta_{n})\in \Mat_{n\times 1}(\oK[\tau])\tau$. By using $t\mathbf{e}_i=\mathbf{e}_i\phi_t$, we see that
        \begin{equation}\label{Eq:rho}
            (t-\theta)\mathbf{e}_1=\mathbf{e}_1(\rho_t-\theta)\in\tau\mathcal{M}_\phi.
        \end{equation} 
        Also, for $2\leq i\leq n$ we have $(t-\theta)\mathbf{e}_i=\delta_{i-1}\mathbf{e}_1+\mathbf{e}_{i+1}$ and $(t-\theta)\mathbf{e}_{n+1}=\delta_r\mathbf{e}_1+\tau\mathbf{e}_2$. In particular, we get
        \begin{equation}\label{Eq:Carlitz}
            (t-\theta)^n\mathbf{e}_2\equiv 0~(\mbox{mod}~\tau\mathcal{M}_\phi).
        \end{equation}
        Consider the ordered $\oK[t]$-basis for $\mathcal{M}_\phi$ defined in \eqref{Eq:Basis_of_M_phi}
        \begin{equation}\label{Eq:Basis_of_M_phi_Drinfeld}
            \mathbf{m}_\phi=\bigg((1,\mathbf{0}),\dots,(\tau^{r-1},\mathbf{0}),(0,\mathbf{m}_{\mathbf{C}^{\otimes n}})\bigg)^\tr\subset\Mat_{(n+1)\times 1}(\mathcal{M}_\phi).
        \end{equation}
        By comparing the last entry of the both sides of $\tau\mathbf{m}_\phi=\wPhi_\phi\mathbf{m}_\phi$, we have
        \[
            \tau\mathbf{e}_2=\nu_1\mathbf{e}_1+\cdots+\nu_r\tau^{r-1}\mathbf{e}_1+(t-\theta)^{n}\mathbf{e}_2.
        \]
        If we regard the above equation in $\cM_\phi/\tau \cM_\phi$, then we get
        \begin{align*}
            0&\equiv \nu_1\mathbf{e}_1+(t-\theta)^{n}\mathbf{e}_2~(\mbox{mod}~\tau\mathcal{M}_\phi)\\
            &\equiv\nu_{10}\mathbf{e}_1+\nu_{11}(t-\theta)\mathbf{e}_1+\cdots+\nu_{1s}(t-\theta)^s\mathbf{e}_1~(\mbox{mod}~\tau\mathcal{M}_\phi)\\
            &\equiv\nu_{10}\mathbf{e}_1~(\mbox{mod}~\tau\mathcal{M}_\phi).
        \end{align*}
        Here the second equality comes from expanding $\nu_1$ and \eqref{Eq:Carlitz} and the third equality follows from \eqref{Eq:rho}. Consequently, we conclude that $\nu_{10}\in\oK\cap\tau\oK[\tau]$ and the desired vanishing on $\nu_{10}$ follows immediately by the fact that the intersection $\oK\cap\tau\oK[\tau]$ only contains $0$.

        Finally, by \eqref{E:CofwPhiE}, we note that
        \[
            \wPhi_E^{-1} = \det(\wPhi_E)^{-1}\mathrm{Cof}(\wPhi_E) = \dfrac{1}{t-\theta}\begin{pmatrix}
            \kappa_1 & \kappa_2 & \cdots & \kappa_{r-1} & \kappa_r \\
            t-\theta & 0 & \cdots & \cdots & 0 \\
            0 & t-\theta & \ddots & \vdots & \vdots \\
            \vdots & \vdots & \cdots & 0 & \vdots \\
            0 & 0 & \cdots & t-\theta & 0
            \end{pmatrix}
        \]
        This implies that $\wPhi_\delta\wPhi_E^{-1}\in\Mat_{1\times r}(\oK[t])$ as $\nu_1\in(t-\theta)\oK[t]$, or equivalently
        \[
            \wPhi_\delta\in\Mat_{1\times r}(\oK[t])\wPhi_E.
        \]
        The desired inclusion $\mathrm{Der}_\partial(\rho,[\cdot]_{n})\subset\mathrm{Der}_{\mathrm{sp}}(\rho,[\cdot]_{n})$ follows immediately.
    \end{proof}

    \begin{example}\label{Ex:Reduced_Derivations}
        We know that every equivalence class of extensions in $\Ext_\partial^1(\rho,[\cdot]_1)$ admits a representative $\delta\in\mathrm{Der}_{\partial}(\rho,[\cdot]_1)$ so that $\delta(t) =\beta_1\tau+\cdots+\beta_{r-1}\tau^{r-1} \in \oK[\tau]\tau$ for some $\beta_1,\dots,\beta_{r-1}\in\oK$ \cite[Eq.~(10)]{PR03}. 
        
        Let $E_\delta=(\mathbb{G}_a^{n+1},\phi)$ be the $t$-module arising from $\delta$ given in \eqref{Eq:t-module_from_delta}. In this example, we aim to calculate $\wPhi_\delta\in\Mat_{1\times r}(\overline{K}[t])$ given in \eqref{Eq:Phi_delta} explicitly. Note that \eqref{Eq:Basis_of_M_phi_Drinfeld} in this case is given by $\mathbf{m}_\phi=\big((1,0),\dots,(\tau^{r-1},0),(0,1)\big)^\tr\in\Mat_{(n+1)\times 1}(\mathcal{M}_\phi)$. On the one hand, $\tau\cdot(0,1)=(0,\tau)$. On the other hand, we have
        \[
            t\cdot(0,1)=(0,1)\begin{pmatrix}
                \rho_t & \\
                \delta(t) & [t]_1
            \end{pmatrix}=(\beta_1\tau+\cdots+\beta_{r-1}\tau^{r-1},\theta+\tau).
        \]
        Thus, we obtain $\tau\cdot(0,1)=(0,\tau)=(-\beta_1)(\tau,0)+\cdots+(-\beta_{r-1})(\tau^{r-1},0)+(t-\theta)\cdot(0,1)$, which implies that $\wPhi_\delta=(0,-\beta_1,\dots,-\beta_{r-1})\in\Mat_{1\times r}(\oK)$,
        and hence
        \begin{equation}\label{Eq:wPhi_E_of_reduced}
            \wPhi_\delta\wPhi_E^{-1}=(-\beta_1,\dots,-\beta_{r-1},0)\in\Mat_{1\times r}(\oK).
        \end{equation}
        This matches with our result in Proposition~\ref{Lem:Special_Extensions_Drinfeld_Modules}, and will be used in our later calculations.
    \end{example}

    By using Proposition~\ref{Lem:Special_Extensions_Drinfeld_Modules}, we can make Theorem~\ref{Thm:Generates_Trackable_Coordinate} more concrete in the case of Drinfeld modules.

    \begin{theorem}\label{Thm:Trackable_Coordinate_Drinfeld_Modules}
        Let $E=(\mathbb{G}_a,\rho)$ be a Drinfeld module of rank $r\geq 2$ with $\rho_t=\theta+\kappa_1\tau+\cdots+\kappa_r\tau^r\in\overline{K}[\tau]$. For $n\geq 1$, we consider $\delta\in\mathrm{Der}_\partial(\rho,[\cdot]_{n})$ and denote by $E_\delta=(\mathbb{G}_a^{n+1},\phi)$ the $t$-module arising from $\delta$ given in \eqref{Eq:t-module_from_delta}.
        Assume that the period lattice of $\phi$ is given by
        \[
            \Lambda_\phi=\bA\begin{pmatrix}
                \omega_1\\
                \bm{\lambda}_1
            \end{pmatrix}+\cdots+\bA\begin{pmatrix}
                \omega_{r}\\
                \bm{\lambda}_{r}
            \end{pmatrix}+\bA\begin{pmatrix}
                0\\
                \bm{\gamma}_{n}
            \end{pmatrix}\subset\Mat_{(n+1)\times 1}(\mathbb{C}_\infty),
        \]
        where $\bm{\gamma}_{n}\in\Mat_{n\times 1}(\mathbb{C}_\infty)$ is the period vector of the $n$-th tensor power of the Carlitz module $\mathbf{C}^{\otimes n}$ defined in \eqref{Eq:Period_Carlitz_Tensor_Powers}.
        Then, for the $t$-module $E_{n-1}=(\mathbb{G}_a^{r(n-1)+r-1},\psi_{n-1})$ defined in Example~\ref{Ex:Exterior_Powers}, there exists $\bm{y}\in(\Lie E_{n-1})(\mathbb{C}_\infty)$ with the property $\Exp_{\psi_{n-1}}(\bm{y})\in E_{n-1}(\oK)$ so that if we express $\bm{\lambda}_j=(\lambda_{j,1},\dots,\lambda_{j,n})^\tr$, and define 
        \[
            S:=\{1,\mathrm{F}_{\bm{\epsilon}}(\bm{y})\mid\bm{\epsilon}\in\mathrm{Der}_\partial(\psi_{n-1},[\cdot]_0),~\bm{\epsilon}(t)\in\Mat_{1\times (rn-1)}(\overline{K}[\tau])\tau\}
        \]
        as well as
        \[
            T:=\{\mathrm{F}_{\bm{m}}(\bm{\omega})\mid\bm{m}\in\mathrm{Der}_\partial(\rho,[\cdot]_0),~\bm{m}(t)\in\oK[\tau]\tau,~\bm{\omega}\in\Lambda_E\},
        \]
        then $\lambda_{j,n}$ is the tractable coordinate of $\bm{\lambda}_j$, and we have
        \[
            \lambda_{j,n}\in\mathrm{Span}_{\oK}\bigg(\{\Tilde{\pi}^n\}\cup\{xy\mid x\in S,~y\in T\}\bigg).
        \]
    \end{theorem}

    \begin{proof}
        By Proposition~\ref{Lem:Special_Extensions_Drinfeld_Modules}, we have 
        \[
            \mathrm{Der}_\partial(\rho,[\cdot]_{n})\subset\mathrm{Der}_{\mathrm{sp}}(\rho,[\cdot]_{n}).
        \]
        Moreover, as we have mentioned in Example~\ref{Ex:Exterior_Powers} that $E_{n-1}=\mathbf{C}^{\otimes n-1}\otimes\wedge^{r-1}E$ admits a $t$-frame $(\iota_\mathbf{n},(t-\theta)^{n-1}\mathrm{Cof}(\Phi_E)$ for its dual $t$-motive. Finally, one checks directly that the top coefficient of $(\psi_{n-1})_{t^{rn-1}}$ is lower triangular invertible which implies that $E_{n-1}$ is almost strictly pure (see Appendix for examples and more details). The desired result follows from Theorem~\ref{Thm:Generates_Trackable_Coordinate} immediately.
    \end{proof}

\subsection{On periods of the third kind for Drinfeld modules}\label{SS:PeriodsofExts2}
    The main purpose of this subsection is to generalize the results of the third kind period for rank $2$ Drinfeld module \cite[Thm~2.4.4]{Chang13} to arbitrary rank $r\geq 2$. Recall that we have defined in Example~\ref{Ex:Exterior_Powers} that $E_0=\wedge^{r-1}E=(\mathbb{G}_a^{r-1},\psi_0)$. We begin with some preliminary calculations. The first lemma is adopted from \cite[Lem.~4.4.19]{NPapanikolas21} in our setting with more flexibility, but the proof is essentially the same.
  
    \begin{lemma}\label{Lem:V_exterior_powers}
        Let $\bsalpha\in E_0(\oK)$ and $\bsy\in\Lie E_0(\mathbb{C}_\infty)$ so that $\Exp_{\psi_0}(\bsy)=\bsalpha$. We write $\mathrm{Cof}(\wPhi_E)=\widetilde{U}_0+\widetilde{U}_1t$ with $\widetilde{U}_i\in\Mat_{r}(\overline{K})$.
        For any $W\in\GL_r(\overline{K}[t])$ with the property that $W^{(-1)}\mathrm{Cof}(\Phi_E)=\mathrm{Cof}(\wPhi_E)W$, if we set
        \[
            \bsg_{\bsy}:=-\langle\tau\mathbf{m}_{E_0}\mid\cG_{\bsy}(t;\psi_0)\rangle^\tr\cdot W
        \]
        and
        \[
            \bsh_{\bsalpha}:=\langle\widetilde{U}_1\mathbf{m}_{E_0}\mid\bsalpha\rangle^\tr \cdot W,
        \]
        then
        \[
            \bsg_{\bsy}^{(-1)}\mathrm{Cof}(\Phi_E)-\bsg_{\bsy}=\bsh_{\bsalpha}.
        \]
    \end{lemma}

    By adopting \cite[Eq.~(10)]{PR03} as we did in Example~\ref{Ex:Reduced_Derivations}, we consider $\delta\in\mathrm{Der}_{\partial}(\rho,[\cdot]_1)$ so that $\delta(t) =\beta_1\tau+\cdots+\beta_{r-1}\tau^{r-1} \in \oK[\tau]\tau$ for some $\beta_1,\dots,\beta_{r-1}\in\oK$. For our Drinfeld module $E$, we have $\widetilde{c}=\det(\wPhi_E)/(t-\theta)=(-1)^{r-1}/\kappa_r$ and $c=\det(\Phi_E)/(t-\theta)=(-1)^{r-1}/(\kappa_r^{(-r)})$. Then, the vector $\mathbf{h}$ in Proposition~\ref{Prop:Identity_among_AGF_2} and Theorem~\ref{Thm:Identity_among_AGF} is given by
    \begin{equation}\label{Eq:h_Drinfeld_Modules}
        \mathbf{h}=-\mathfrak{u}_c^{-1}\widetilde{\Phi}_\delta\widetilde{\Phi}_E^{-1}(V_E^{-1})^\tr=\big((-1)^{r-1}(\kappa_r)^{(1-r)}\big)^{-1/q-1}(\beta_1,\dots,\beta_{r-1},0)(V_E^{-1})^\tr.
    \end{equation}
    To proceed our calculations, we aim to find the algebraic point $(\epsilon_1\circ\iota_{\mathbf{n}_{E_0}})(\mathbf{h})\in E_0(\oK)$ used in the proof of Proposition~\ref{Prop:Identity_among_AGF_2}. We adopt the following approach to avoid the direct but heavy calculations of $(\epsilon_1\circ\iota_{\mathbf{n}_{E_0}})(\mathbf{h})$.

    \begin{lemma}\label{Lem:Special_Algebraic_Point}
        For $\delta\in\mathrm{Der}_{\partial}(\rho,[\cdot]_1)$ with $\delta(t) =\beta_1\tau+\cdots+\beta_{r-1}\tau^{r-1} \in \oK[\tau]\tau$, we set
        \[
            \bsalpha_\delta:=\bigg((-1)^{\frac{(r-1)(r-2)}{2}}\big(\prod_{j=1}^{r-2}\kappa_r^{(-j)}\big)\big(\sqrt[q-1]{(-1)^{r-1}\kappa_{r}^{(2-r)}}\big)\bigg)^{-1}(\beta_{r-1}, \beta_{r-2}, \dots, \beta_1)^\tr\in E_0(\oK),
        \]
        where the finite product $\prod_{j=1}^{r-2}\kappa_r^{(-j)}$ is defined to be $1$ if $r=2$.
        Then, the vector
        \[
            \bm{h}_{{\bm{\alpha}}_\delta}=\langle\widetilde{U}_1\mathbf{m}_{E_0}\mid\bsalpha_\delta\rangle^\tr\cdot\mathrm{Cof}(V_E)
        \]
        coincides with $\mathbf{h}$ given in \eqref{Eq:h_Drinfeld_Modules}, where we express $\mathrm{Cof}(\wPhi_E)=\widetilde{U}_0+\widetilde{U}_1 t$ with $\widetilde{U}_i\in\Mat_{r}(\overline{K})$.
        Consequently, we have
        \[
            (\epsilon_1\circ\iota_{\mathbf{n}_{E_0}})(\mathbf{h})=(\epsilon_1\circ\iota_{\mathbf{n}_{E_0}})(\bm{h}_{\bsalpha_\delta})=\bsalpha_\delta.
        \]
    \end{lemma}

    \begin{proof}
        Recall that $\mathbf{m}_{E_0}=\big((-1)^{r-1}\tau\widetilde{\mathbf{s}}_1,\widetilde{\mathbf{s}}_{r-1},\dots,\widetilde{\mathbf{s}}_{1}\big)^\tr$ has been defined in \eqref{E:tbasistmotiveexteriorpower}. If we set
        \begin{equation}\label{Eq:Constant_in_Formula}
            \mathfrak{c}_E:=\bigg((-1)^{\frac{(r-1)(r-2)}{2}}\big(\prod_{j=1}^{r-2}\kappa_r^{(-j)}\big)\big(\sqrt[q-1]{(-1)^{r-1}\kappa_{r}^{(2-r)}}\big)\bigg)^{-1}\in\oK^\times,
        \end{equation}
        then, by Lemma~\ref{Lem:V_exterior_powers} and the fact that
        \[
            \widetilde{U}_1=\frac{(-1)^{r-1}}{\kappa_r}\begin{pmatrix}
                0 & 1 & & \\
                \vdots & & \ddots & \\
                \vdots & & & 1\\
                0 & \dots & \dots & 0
            \end{pmatrix},
        \]
        we have
        \begin{equation}\label{Eq:Formula_of_h_delta}
            \bm{h}_{{\bm{\alpha}_\delta}}=\langle\widetilde{U}_1\mathbf{m}_{\wedge^{r-1}E}\mid\bsalpha\rangle^\tr V_{\wedge^{r-1}E}=\mathfrak{c}_E(-1)^{r-1}\kappa_r^{-1}\det(V_E)\big(\beta_1,\dots,\beta_{r-1},0\big)(V_E^{-1})^\tr.
        \end{equation}
        By comparing with $\mathbf{h}$ and $\bm{h}_{{\bm{\alpha}_\delta}}$, it is enough to show that
        \begin{equation}\label{Eq:Constant_Equation}
            \mathfrak{c}_E=\mathfrak{u}_c^{-1}(-1)^{r-1}\kappa_r\det(V_E)^{-1}
        \end{equation}
        It follows from \cite[Eq.~(4.6.8)]{NPapanikolas21} (see also \cite[\S3.4]{CP12}) that
        \[
            V_E=\begin{pmatrix}
                \kappa_1 & \kappa_2^{(-1)} & \cdots & \kappa_{r-1}^{(2-r)} & \kappa_r^{(1-r)}\\
                \kappa_2 & \kappa_3^{(-1)} & \cdots & \kappa_r^{(2-r)} & \\
                \vdots & \vdots & \iddots & & \\
                \kappa_{r-1} & \kappa_r^{(-1)} & & & \\
                \kappa_r & & & &
            \end{pmatrix}
        \]
        is an anti-diagonal matrix and its determinant is given by
        \[
            \det(V_E)=(-1)^{\frac{r(r-1)}{2}}\prod_{j=0}^{r-1}\kappa_r^{(-j)}.
        \]
    Now we can deduce the desired equality \eqref{Eq:Constant_Equation} by computing 
        \begin{align*}
            \mathfrak{u}_c^{-1}(-1)^{r-1}\kappa_r\det(V_E)^{-1}&=\bigg((-1)^{r-1}(\kappa_r)^{(1-r)}\bigg)^{-1/q-1}\bigg((-1)^{r-1}\kappa_r\bigg)\bigg((-1)^{\frac{r(r-1)}{2}}\prod_{j=0}^{r-1}\kappa_r^{(-j)}\bigg)^{-1}\\
            &=\bigg((-1)^{\frac{(r-1)(r-2)}{2}}\big(\prod_{j=1}^{r-2}\kappa_r^{(-j)}\big)\big(\sqrt[q-1]{(-1)^{r-1}\kappa_{r}^{(2-r)}}\big)\bigg)^{-1}\\
            &=\mathfrak{c}_E.
        \end{align*}
        Finally, the evaluation of $\epsilon_1\circ\iota_{\mathbf{n}_{E_0}}$ follows from the same proof of \cite[Prop.~4.5.22]{NPapanikolas21} and the almost strictly pureness of $E_0$.
    \end{proof}

    \begin{remark}
        In the case of rank $2$ Drinfeld modules, if we consider $\delta\in\mathrm{Der}_{\partial}(\rho,[\cdot]_1)$ with $\delta(t) =\beta_1\tau\in \oK[\tau]\tau$, then the constant $\mathfrak{c}_E$ defined in the proof of Lemma~\ref{Lem:Special_Algebraic_Point} is given by $\mathfrak{c}_E=1/\sqrt[q-1]{-\kappa_{2}}$. In particular, the special algebraic point $\bsalpha_\delta$ defined in Lemma~\ref{Lem:Special_Algebraic_Point} is given by
        \[
            \bsalpha_\delta=\frac{\beta_1}{\sqrt[q-1]{-\kappa_{2}}}.
        \]
        This matches with the algebraic point chosen in \cite[Thm.~2.4.4]{Chang13}, where the notation $\alpha/\sqrt[q-1]{-\Delta}$ is used in the statement there.
    \end{remark}

Now, we are able to present the main result of this subsection concerning the explicit formula of the periods of the third kind for Drinfeld modules in arbitrary rank $r\geq 2$. For the convenience of later use, for $0\leq i\leq r-1$ we denote by $\mathrm{F}_{\tau^i}(z)$ the quasi-periodic function associated to the $(\rho,[\cdot]_0)$-biderivation $(t\mapsto \tau^i)$.

\begin{theorem}\label{Thm:Third_Kind_Periods_Formula}
    Let $E=(\mathbb{G}_a,\rho)$ be a Drinfeld module of rank $r\geq 2$ with $\rho_t=\theta+\kappa_1\tau+\cdots+\kappa_r\tau^r\in\overline{K}[\tau]$. Consider $\delta\in\mathrm{Der}_\partial(\rho,[\cdot]_{1})$ so that $\delta(t) =\beta_1\tau+\cdots+\beta_{r-1}\tau^{r-1} \in \oK[\tau]\tau$ for some $\beta_1,\dots,\beta_{r-1}\in\oK$ and denote by $E_\delta=(\mathbb{G}_a^{2},\phi)$ the $t$-module arising from $\delta$ given in \eqref{Eq:t-module_from_delta}.
    Assume that the period lattice of $\phi$ is given by
    \[
        \Lambda_\phi=\bA\begin{pmatrix}
            \omega_1\\
            \lambda_1
        \end{pmatrix}+\cdots\bA\begin{pmatrix}
            \omega_{r}\\
            \lambda_{r}
        \end{pmatrix}+\bA\begin{pmatrix}
            0\\
            \Tilde{\pi}
        \end{pmatrix}\subset\Mat_{2\times 1}(\mathbb{C}_\infty).
    \]
    For the $t$-module $E_{0}=(\mathbb{G}_a^{r-1},\psi_{0})$ defined in Example~\ref{Ex:Exterior_Powers}, let $\bsalpha_{\delta}\in E_0(\oK)$ be the algebraic point constructed explicitly in Lemma~\ref{Lem:Special_Algebraic_Point} and $\bsy_\delta=(y_1,\dots,y_{r-1})^\tr\in\Lie E_0(\mathbb{C}_\infty)$ be chosen so that $\Exp_{\psi_0}(\bsy_\delta)=\bsalpha_\delta$.
    Then, there exists $a_j\in\bA$ so that
    \[
        \lambda_j=-\mathfrak{c}_E^{-1}\left[\sum_{\ell=2}^r y_{r-\ell+1} \mathrm{F}_{\tau^{\ell-1}}(\omega_j) + (-1)^{r-1}\mathrm{F}_{\bm{\epsilon}}(\bsy_\delta)\omega_j\right]+a_j(\theta)\widetilde{\pi},
    \]
    where $\mathfrak{c}_E$ is the constant defined in \eqref{Eq:Constant_in_Formula} and $\bm{\epsilon}\in\mathrm{Der}_\partial(\psi_0,[\cdot]_0)$ is uniquely determined by $\bm{\epsilon}(t)=\tau\widetilde{\mathbf{s}}_1\in\Mat_{1\times (r-1)}(\oK[\tau])\tau$.
\end{theorem}

\begin{proof}
    Let
    \[
        \bsh_{\bsalpha}:=\langle\widetilde{U}_1\mathbf{m}_{E_0}\mid\bsalpha_\delta\rangle^\tr \cdot \mathrm{Cof}(V_E)
    \] 
    and
    \begin{align*}
        \bsg_{\bsy_{\delta}}&:=-\langle\tau\mathbf{m}_{E_0}\mid\cG_{\bsy}(t;\psi_0)\rangle^\tr\mathrm{Cof}(V_E)\\
        &=-\bigg(
            (-1)^{r-1} \langle  \tau^2\widetilde{\mathbf{s}}_1\mid \cG_{\bsy}(t;\psi_0)\rangle,
            \langle \tau\widetilde{\mathbf{s}}_{r-1}\mid \cG_{\bsy}(t;\psi_0)\rangle,
            \dots,
            \langle\tau\widetilde{\mathbf{s}}_{1}\mid \cG_{\bsy}(t;\psi_0)\rangle
        \bigg)\mathrm{Cof}(V_E).
    \end{align*}
    Since $V_E^{(-1)}\Phi_E=\wPhi_E^{\tr}V_E$, it follows that $\mathrm{Cof}(V_E)^{(-1)}\mathrm{Cof}(\Phi_E)=\mathrm{Cof}(\wPhi_E)^{\tr}\mathrm{Cof}(V_E)$. Then, Lemma~\ref{Lem:V_exterior_powers} implies that
    \[
        \bsg_{\bsy_\delta}^{(-1)}\mathrm{Cof}(\Phi_E)-\bsg_{\bsy_\delta}=\bsh_{\bsalpha_\delta}.
    \]
    
    By Proposition~\ref{Prop:Identity_among_AGF_2} and \eqref{Eq:Relation_AGF}, there exists $C\in\Mat_{r+1}(\mathbf{A})$ of the form
    \[
        C=\begin{pmatrix}
        \mathbb{I}_r & \\
        a_1,\dots,a_r & 1
    \end{pmatrix}
    \] 
    so that we can deduce \eqref{Eq:Periods_Relations_from_AGFs} more explicitly in our setting as
    \begin{equation}\label{Eq:Periods_Relations_from_AGFs_Drinfeld}
        \begin{split}
            \big((t-\theta)g_{1,n}(t;\phi),\dots,(t-\theta)g_{r,n}(t;\phi),\Omega^{-n}\big)&=\\
            (\mathfrak{u}_c\bsg_{\bsy_\delta}V_E^\tr\Upsilon_E,\Omega^{-n})C+(&\sum_{i=1}^{r-1}\beta_i\cG_{\omega_1}(t;\rho)^{(i)},\dots,\sum_{i=1}^{r-1}\beta_i\cG_{\omega_r}(t;\rho)^{(i)},0).
        \end{split}
    \end{equation}
We observe that
\begin{equation}\label{E:calculations_g_ext_powers}
    \begin{split}
        \mathfrak{u}_c\bsg_{\bsy_\delta}V_E^\tr\Upsilon_E &=\mathfrak{u}_c (\bsg_{\bsy_\delta}^{(-1)}\Cof(\Phi_E)-\bsh_{\bsalpha_\delta})V_E^\tr\Upsilon_E\\
    \end{split}
\end{equation}
On the one hand, by \eqref{Eq:Formula_of_h_delta} we have
\begin{align*}
    \mathfrak{u}_c\bsh_{\bsalpha_\delta}V_E^\tr\Upsilon_E&=\mathfrak{u}_c\mathfrak{c}_E(-1)^{r-1}\kappa_r^{-1}\det(V_E)\big(\beta_1,\dots,\beta_{r-1},0\big)(V_E^{-1})^\tr V_E^\tr\Upsilon_E\\
    &=\big(\beta_1,\dots,\beta_{r-1},0\big)\Upsilon_E\\
    &=(\sum_{i=1}^{r-1}\beta_i\cG_{\omega_1}(t;\rho)^{(i)},\dots,\sum_{i=1}^{r-1}\beta_i\cG_{\omega_r}(t;\rho)^{(i)}).
\end{align*}
On the other hand, we have
\begin{align*}
    \mathfrak{u}_c\bsg_{\bsy_\delta}^{(-1)}\Cof(\Phi_E)V_E^\tr\Upsilon_E&=-\mathfrak{u}_c\langle\mathbf{m}_{E_0}\mid\cG_{\bsy_\delta}(t;\psi_0)\rangle^\tr\mathrm{Cof}(V_E)^{(-1)}\Cof(\Phi_E)V_E^\tr\Upsilon_E\\
    &=-\mathfrak{u}_c\langle\mathbf{m}_{E_0}\mid\cG_{\bsy_\delta}(t;\psi_0)\rangle^\tr\mathrm{Cof}(\wPhi_E)^{\tr}\mathrm{Cof}(V_E)V_E^\tr\Upsilon_E\\
    &=-\mathfrak{u}_c\det(V_E)\det(\wPhi_E)\langle\mathbf{m}_{E_0}\mid\cG_{\bsy_\delta}(t;\psi_0)\rangle^\tr\Upsilon_E^{(-1)}\\
    &=-\mathfrak{c}_E^{-1}(t-\theta)\langle\mathbf{m}_{E_0}\mid\cG_{\bsy_\delta}(t;\psi_0)\rangle^\tr\Upsilon_E^{(-1)}
\end{align*}
where the second equality follows from $\mathrm{Cof}(V_E)^{(-1)}\mathrm{Cof}(\Phi_E)=\mathrm{Cof}(\wPhi_E)^{\tr}\mathrm{Cof}(V_E)$, the third equality comes from $\Upsilon_E = \wPhi_E \Upsilon_E ^{(-1)}$, and the last equality uses $\mathfrak{c}_E$ defined in \eqref{Eq:Constant_in_Formula}. It follows from \eqref{E:calculations_g_ext_powers} that \eqref{Eq:Periods_Relations_from_AGFs_Drinfeld} becomes
\begin{equation}\label{Eq:Periods_Relations_from_AGFs_Drinfeld_2}
    \big((t-\theta)g_{1,n}(t;\phi),\dots,(t-\theta)g_{r,n}(t;\phi),\Omega^{-n}\big)=
    (-\mathfrak{c}_E^{-1}(t-\theta)\langle\mathbf{m}_{E_0}\mid\cG_{\bsy_\delta}(t;\psi_0)\rangle^\tr\Upsilon_E^{(-1)},\Omega^{-n})C.
\end{equation}
For $1\leq j\leq r$, the $j$-th entry of $-\mathfrak{c}_E^{-1}(t-\theta)\langle\mathbf{m}_{E_0}\mid\cG_{\bsy_\delta}(t;\psi_0)\rangle^\tr\Upsilon_E^{(-1)}$ is given by
\begin{align*}
    \begin{split}
        -\mathfrak{c}_E^{-1}\bigg[\sum_{\ell=2}^r\big((t-\theta)&\twistop{\widetilde{\mathbf{s}}_{r-\ell+1}}{\cG_{\bsy}(t;\psi_0)}\big)\cG_{\omega_j}(t;\rho)^{(\ell-1)}\\
        &+(-1)^{r-1}\twistop{\tau\widetilde{\mathbf{s}}_1}{\cG_{\bsy_\delta}(t;\psi_0)}\big((t-\theta)\cG_{\omega_j}(t;\rho)\big)\bigg].
    \end{split}
\end{align*}
By \eqref{E:AGFRes} and $\big(\rd\psi_0-t\Id_{r-1}\big)^{-1}=-(t-\theta)^{-1}\Id_{r-1}$, we derive that
\[
    \big((t-\theta)\twistop{\widetilde{\mathbf{s}}_{r-\ell+1}}{\cG_{\bsy_\delta}(t;\psi_0}\big)\mid_{t=\theta}=\Res_{t=\theta}\twistop{\widetilde{\mathbf{s}}_{r-\ell+1}}{\cG_{\bsy_\delta}(t;\psi_0}=-y_{r-\ell+1}.
\]
Finally, by specializing $t=\theta$ on both sides of \eqref{Eq:Periods_Relations_from_AGFs_Drinfeld_2}, we derive the desired equality
\begin{align*}
    \lambda_j=-\mathfrak{c}_E^{-1}\left[\sum_{\ell=2}^r y_{r-\ell+1} \mathrm{F}_{\tau^{\ell-1}}(\omega_j) + (-1)^{r-1}\mathrm{F}_{\bm{\epsilon}}(\bsy_\delta)\omega_j\right]+a_j(\theta)\widetilde{\pi}.
\end{align*}
\end{proof}     

As a consequence of the above theorem, we deduce a generalization of Chang \cite[Thm.~3.3.2]{Chang13} concerning the algebraic independence for periods of the first, the second, and the third kinds for $E$. We set
\[
    \mathbf{P}_E:=\oK\bigg(\mathrm{F}_{\tau^i}(w)\mid 0\leq i\leq r-1,~w\in\Lambda_E\bigg).
\]
Note that $\mathbf{P}_E$ is the field generated by all periods and quasi-periods of $E$.

\begin{theorem}\label{Thm:Algebraic_Independence}
    Let $E=(\mathbb{G}_a,\rho)$ be a Drinfeld module of rank $r\geq 2$ defined over $\overline{K}$ with the period lattice $\Lambda_E=\bA w_1+\cdots+\bA w_r$ and the endomorphism ring $\End(E)$ of rank $\bm{s}\geq 1$ over $\bA$. Consider $\delta_1,\dots,\delta_m\in\mathrm{Der}_\partial(\rho,[\cdot]_1)$ with $\delta_i(t)\in\oK[\tau]\tau$ and $\deg_\tau\delta_i(t)\leq r-1$ for each $1\leq i\leq m$. Let $G_{\phi_i}=(\mathbb{G}_a^2,\phi_j)$ be the $t$-module arising from $\delta_i$ as defined in \eqref{Eq:t-module_from_delta}. Assume that the period lattice of $G_{\phi_i}$ is given by
    \[
        \Lambda_{\phi_i}=\bA\begin{pmatrix}
            \omega_1\\
            \lambda_{11}^{[i]}
        \end{pmatrix}+\cdots+\bA\begin{pmatrix}
            \omega_{r}\\
            \lambda_{r1}^{[i]}
        \end{pmatrix}+\bA\begin{pmatrix}
            0\\
            \Tilde{\pi}
        \end{pmatrix}\subset\Mat_{2\times 1}(\mathbb{C}_\infty)
    \]
    Let $\bm{y}_{\delta_i}=(y_1^{[i]},\dots,y_{r-1}^{[i]})^\tr\in(\Lie E_0)(\mathbb{C}_\infty)$ be chosen as in Theorem~\ref{Thm:Third_Kind_Periods_Formula}. If 
    \[
        \rank_{\End E_0}\mathrm{Span}_{\End E_0}\bigg(\bsalpha_{\delta_1},\dots,\bsalpha_{\delta_m}\bigg)=m,
    \]
    then we have
    \[
        \trdeg_{\oK}\overline{K}\bigg(w_j,\mathrm{F}_{\tau^{\ell-1}}(\omega_j),\lambda_{j1}^{[i]}\mid 2\leq\ell\leq r,~1\leq j\leq r,~1\leq i\leq m\bigg)=\frac{r^2}{\bm{s}}+rm.
    \]
\end{theorem}

\begin{proof}
    By \cite[Thm.~1.2.2]{CP12}, it is known that that
    \[
        \trdeg_{\oK}\mathbf{P}_E=\frac{r^2}{\bm{s}}.
    \]
    Thus, our task is to show that the set of $rm$ elements
    \[
        \mathcal{S}:=\{\lambda_{j1}^{[i]}\mid 1\leq j\leq r,~1\leq i\leq m\}
    \]
    is algebraically independent over $\mathbf{P}_E$. By Theorem~\ref{Thm:Third_Kind_Periods_Formula}, for each $1\leq i\leq m$ we notice that
    \begin{equation}
        \begin{pmatrix}
            \lambda_{11}^{[i]}-a_{1}^{[i]}\Tilde{\pi}\\
            \vdots\\
            \lambda_{r1}^{[i]}-a_{r}^{[i]}\Tilde{\pi}
        \end{pmatrix}=-\mathfrak{c}_E^{-1}\begin{pmatrix}
            w_1 & \mathrm{F}_{\tau}(w_1) & \cdots & \mathrm{F}_{\tau^{r-1}}(w_1)\\
            \vdots & & & \vdots\\
            w_r & \mathrm{F}_{\tau}(w_r) & \cdots & \mathrm{F}_{\tau^{r-1}}(w_r)
        \end{pmatrix}\begin{pmatrix}
            (-1)^{r-1}\widetilde{\mathbf{F}}_{\bm{\epsilon}}(\bsy_{\delta_j})\\
            y_{r-1}^{[i]}\\
            \vdots\\
            y_{1}^{[i]}
        \end{pmatrix},
    \end{equation}
    for some $a_{j}^{[i]}\in A$ and $\bm{\epsilon}\in\mathrm{Der}_\partial(\psi_0,[\cdot]_0)$ is uniquely determined by $\bm{\epsilon}(t)=\tau\widetilde{\mathbf{s}}_1\in\Mat_{1\times (r-1)}(\oK[\tau])\tau$. We denote by
    \[
        \mathcal{F}_E:=\begin{pmatrix}
            w_1 & \mathrm{F}_{\tau}(w_1) & \cdots & \mathrm{F}_{\tau^{r-1}}(w_1)\\
            \vdots & & & \vdots\\
            w_r & \mathrm{F}_{\tau}(w_r) & \cdots & \mathrm{F}_{\tau^{r-1}}(w_r)
        \end{pmatrix}\in\Mat_{r}(\mathbf{P}_E).
    \]
    By the Legendre's relation for Drinfeld modules (see \cite{Gek89} for the analytic approach or \cite{Gos94} for Anderson's motivic method), we have $\det(\mathcal{F}_E)/\Tilde{\pi}\in\oK^\times$. It follows that $\Tilde{\pi}\in\mathbf{P}_E$ and $\mathcal{F}_E\in\GL_r(\mathbf{P}_E)$. In particular, if we consider the set of $rm$ elements
    \[
        \mathcal{T}:=\{\widetilde{\mathbf{F}}_{\bm{\epsilon}}(\bsy_{\delta_i}),y_{j}^{[i]}\mid 1\leq j\leq r-1,~1\leq i\leq m\},
    \]
    then the following two fields coincide, namely,
    \[
        \mathbf{P}_E(\mathcal{S})=\mathbf{P}_E(\mathcal{T}).
    \]
    Therefore, we have $\trdeg_{\mathbf{P}_E}\mathbf{P}_E(\mathcal{S})=\trdeg_{\mathbf{P}_E}\mathbf{P}_E(\mathcal{T})$ and our task is reduced to show that the set $\mathcal{T}$ of $rm$ elements is algebraically independent over $\mathbf{P}_E$. Since
    \[
        \rank_{\End{E_0}}\mathrm{Span}_{\End{E_0}}\big(\bsalpha_{\delta_1},\dots,\bsalpha_{\delta_m}\big)=m,
    \]
    if we denote by $\mathcal{K}_0:=\End E_0\otimes_{\bA}\mathbf{K}$ and $\{\bm{\nu}_1,\dots,\bm{\nu}_{r/\bm{s}}\}$ the maximal $\mathcal{K}_0$-linearly independent set in the period lattice of $E_0$, then the set $\{\bm{\nu}_1,\dots,\bm{\nu}_{r/\bm{s}},\bsy_{\delta_1},\dots,\bsy_{\delta_m}\}$ is a $\mathcal{K}_0$-linearly independent set. Indeed, by applying $\Exp_{\psi_0}(\cdot)$, any non-trivial $\mathcal{K}_0$-linear relation can be lifted to a non-trivial $\End(E_0)$-linear relation among $\bsalpha_{\delta_1},\dots,\bsalpha_{\delta_m}$. Now the desired algebraic independence follows immediately by \cite[Thm.~1.1]{GN24}.
\end{proof}

\begin{remark}
We mention that Pham, in their master's thesis \cite{Pha21}, studied transcendence of quasi-periods of $E_0$ in the case of $r=2$. 
During that time, Chang posed the question of finding explicit formulas for the third kind periods of $E_0$ for arbitrary $r \geq 2$. This paper started as a result of Chang's question. 
\end{remark}

\begin{appendices}
\section{On the \texorpdfstring{$t$}{t}-module structure of \texorpdfstring{$\mathbf{C}^{\otimes e}\otimes\wedge^{r-1}E$}{CeE}}
    In this appendix, we provide some detailed calculations about the properties of the $t$-module $E_{e}=\mathbf{C}^{\otimes {e}}\otimes\wedge^{r-1}E=(\mathbb{G}_a^{re+r-1},\psi_{e})$ defined in Example~\ref{Ex:Exterior_Powers} for $E=(\GG_a, \rho)$ a Drinfeld module of rank $r\geq 2$ given in Example~\ref{Ex:Drinfeld_Module} and $e\geq 0$.
\subsection{\texorpdfstring{$t$}{t}-motives, dual \texorpdfstring{$t$}{t}-motives, and their \texorpdfstring{$t$}{t}-frames}
    In what follows, we follow closely the ideas and approaches used in \cite[\S3]{GN24} to obtain the essential properties for the $t$-motives and dual $t$-motives of $E_{e}=(\mathbb{G}_a^{re+r-1},\psi_{e})$.
    Let $\mathcal{M}_{E_e}=\Mat_{1\times re+r-1}(\mathbb{K}[\tau])$ be its associated $t$-motive and $\widetilde{\mathbf{s}}_i$ be the standard $\mathbb{K}[\tau]$-basis. 
    
    First let $e=0$. We aim to find an appropriate $\mathbb{K}[t]$-basis for $\mathcal{M}_{E_0}$ so that the $\tau$-action on the $\mathbb{K}[t]$-basis in question can be represented precisely by $\mathrm{Cof}(\wPhi_E)$. To begin with, let $\Mat_{1\times r}(\mathbb{K}[t])$ be the free $\mathbb{K}[t]$-module with $\mathbf{e}_j$ the standard $\mathbb{K}[t]$-basis. We endow $\Mat_{1\times r}(\mathbb{K}[t])$ with a $\mathbb{K}[\tau]$-module structure given by
    \begin{equation}
        \tau\begin{pmatrix}
                \mathbf{e}_1\\
                \vdots\\
                \mathbf{e}_r
        \end{pmatrix}=\mathrm{Cof}(\wPhi_E)\begin{pmatrix}
                \mathbf{e}_1\\
                \vdots\\
                \mathbf{e}_r
        \end{pmatrix}=\frac{(-1)^{r-1}}{\kappa_r}\begin{pmatrix}
            \kappa_1 & t-\theta & 0 & \dots & 0\\
            \vdots & & \ddots & & \vdots\\
            \vdots & & & \ddots & \vdots\\
            \kappa_{r-1} & 0 & \dots & \dots & t-\theta\\
            \kappa_r & 0 & \dots & \dots & 0
        \end{pmatrix}\begin{pmatrix}
                \mathbf{e}_1\\
                \vdots\\
                \mathbf{e}_r
        \end{pmatrix}.
    \end{equation}
    Using the relation $\mathbf{e}_1=(-1)^{r-1}\tau\mathbf{e}_r$, we have
    \[
        t\mathbf{e}_2=\theta\mathbf{e}_2+\bigg(\kappa_r\tau^2+(-1)^r\kappa_1\tau\bigg)\mathbf{e}_r.
    \]
    Moreover, for $3\leq j\leq r$, we have
    \[
        t\mathbf{e}_j=\theta\mathbf{e}_j+(-1)^{r-1}\kappa_r\tau\mathbf{e}_{j-1}+(-1)^r\kappa_{j-1}\tau\mathbf{e}_r.
    \]
    Then, one may verify that $\{\mathbf{e}_2,\dots,\mathbf{e}_r\}$ forms a $\mathbb{K}[\tau]$-basis for $\Mat_{1\times r}(\mathbb{K}[t])$. Now we consider the $\mathbb{K}[\tau]$-linear map
    \begin{align*}
        \iota:\Mat_{1\times r}(\mathbb{K}[t])&\to\mathcal{M}_{\wedge^{r-1}E}\\
        \mathbf{e}_j&\mapsto\widetilde{\mathbf{s}}_{r+1-j}.
    \end{align*}
    By comparing the rank over $\mathbb{K}[\tau]$, it is an isomorphism of $\mathbb{K}[\tau]$-modules. We claim that $\iota$ is even a $\mathbb{K}[t]$-linear map. Indeed, for $j=2$ we have
    \begin{align*}
        \iota(t\mathbf{e}_2)&=\iota\bigg(\theta\mathbf{e}_2+\big(\kappa_r\tau^2+(-1)^r\kappa_1\tau\big)\mathbf{e}_r\bigg)\\
        &=\theta\widetilde{\mathbf{s}}_{r-1}+\big(\kappa_r\tau^2+(-1)^r\kappa_1\tau\big)\widetilde{\mathbf{s}}_{1}\\
        &=\widetilde{\mathbf{s}}_{r-1}(\psi_0)_t\\
        &=t\cdot\iota(\mathbf{e}_2).
    \end{align*}
    Moreover, for $3\leq j\leq r$, we have
    \begin{align*}
        \iota(t\mathbf{e}_j)&=\iota\bigg(\theta\mathbf{e}_j+(-1)^{r-1}\kappa_r\tau\mathbf{e}_{j-1}+(-1)^r\kappa_{j-1}\tau\mathbf{e}_r\bigg)\\
        &=\theta\widetilde{\mathbf{s}}_{r+1-j}+(-1)^{r-1}\kappa_r\tau\widetilde{\mathbf{s}}_{r+2-j}+(-1)^r\kappa_{j-1}\tau\widetilde{\mathbf{s}}_{1}\\
        &=\widetilde{\mathbf{s}}_{r+1-j}(\psi_0)_t\\
        &=t\cdot\iota(\mathbf{e}_j).
    \end{align*}
    Consequently, the pair $\big(\iota,\mathrm{Cof}(\wPhi_E)\big)$ defines a $t$-frame for the $t$-motive $\mathcal{M}_{E_0}$.

Next, let $e\geq 1$. Let $\Mat_{1\times r}(\mathbb{K}[t])$ be the free $\mathbb{K}[t]$-module with $\mathbf{e}_j$ the standard $\mathbb{K}[t]$-basis. We endow $\Mat_{1\times r}(\mathbb{K}[t])$ with a $\mathbb{K}[\tau]$-module structure given by
    \begin{align*}
        \tau\begin{pmatrix}
                \mathbf{e}_1\\
                \vdots\\
                \mathbf{e}_r
        \end{pmatrix}&=\mathrm{Cof}(\wPhi_E)(t-\theta)^e\begin{pmatrix}
                \mathbf{e}_1\\
                \vdots\\
                \mathbf{e}_r
        \end{pmatrix}\\
        &=\frac{(-1)^{r-1}}{\kappa_r}\begin{pmatrix}
            \kappa_1(t-\theta)^e & (t-\theta)^{e+1} & 0 & \dots & 0\\
            \vdots & & \ddots & & \vdots\\
            \vdots & & & \ddots & \vdots\\
            \kappa_{r-1}(t-\theta)^e & 0 & \dots & \dots & (t-\theta)^{e+1}\\
            \kappa_r (t-\theta)^e& 0 & \dots & \dots & 0
        \end{pmatrix}\begin{pmatrix}
                \mathbf{e}_1\\
                \vdots\\
                \mathbf{e}_r
        \end{pmatrix}.
    \end{align*}

    By a similar process to $e=0$ case, we can define a $\mathbb{K}[\tau]$-linear map
    \begin{align*}
        \iota:\Mat_{1\times r}(\mathbb{K}[t])&\to\mathcal{M}_{E_e}\\
        \mathbf{e}_j&\mapsto\widetilde{\mathbf{s}}_{r+1-j}.
    \end{align*}
    such that the pair $\big(\iota,\mathrm{Cof}(\wPhi_E)(t-\theta)^e\big)$ defines a $t$-frame for the $t$-motive $\mathcal{M}_{E_e}$ and  the $r$-tuple
    \[
        (\iota(\mathbf{e}_1),\iota(\mathbf{e}_2)\dots,\iota(\mathbf{e}_r)\}=\{\widetilde{\mathbf{s}}_r,\widetilde{\mathbf{s}}_{r-1},\dots,\widetilde{\mathbf{s}}_{1})
    \]
    gives the desired ordered $\mathbb{K}[t]$-basis of $\mathcal{M}_{E_e}=\Mat_{1\times (re+r-1)}(\mathbb{K}[\tau])$.

\begin{lemma}[{cf.~\cite[\S3]{GN24}}]\label{Lem:Exterior_Powers_basis_Dual_t_motive}
    Let $E=(\mathbb{G}_a,\rho)$ be the Drinfeld module given in Example~\ref{Ex:Drinfeld_Module} and for $e\geq 0$, we consider the $t$-module $E_e=\mathbf{C}^{\otimes e}\otimes\wedge^{r-1}E=(\mathbb{G}_a^{re+r-1},\psi_e)$ defined in Example~\ref{Ex:Exterior_Powers}.
    Let $\mathcal{N}_{E_e}=\Mat_{1\times (re+r-1)}(\mathbb{K}[\tau])$ be the dual $t$-motive associated to $E_e$. Then, there exists an explicitly constructed $\KK[t]$-basis $\mathbf{n}_{E_e}$ of $\mathcal{N}_{E_e}$ so that $\sigma\mathbf{n}_{E_e}=\mathrm{Cof}(\Phi_E)(t-\theta)^e\mathbf{n}_{E_e}$. In other words, the $t$-module $E_e$ has a $t$-frame $(\iota_{E_e},\mathrm{Cof}(\Phi_E)(t-\theta)^e)$ for its dual $t$-motive $\mathcal{N}_{E_e}$.
\end{lemma}

\begin{proof}
    Recall that for $e\geq 0$, the dual $t$-motive associated to $E_e$ is defined by setting $\mathcal{N}_{E_e}=\Mat_{1\times re+r-1}(\mathbb{K}[\sigma])$ such that its $\KK[t]$-action is uniquely determined by
    \[
        t\cdot n:=n(\psi_e)_t^*,~~n\in \mathcal{N}_{E_e}.
    \]   
    Our task is to seek an ordered $\mathbb{K}[t]$-basis $\mathcal{N}_{E_e}$ so that its $\sigma$-action can be represented precisely by $\mathrm{Cof}(\Phi_E)(t-\theta)^e$.  The following calculations are adapted from \cite[\S3]{GN24}. For $1\leq i \leq re+r-1$, let $\mathbf{s}_i$ be the standard $\mathbb{K}[\sigma]$-basis of $\cN_{E_e}$. Note that $\mathbf{s}_{E_0}:=(\sigma\mathbf{s}_{r-1}, \mathbf{s}_1, \dots, \mathbf{s}_{r-1})^\tr \in \Mat_{r\times 1}(\cN_{E_0})$ forms a $\KK[t]$-basis for $\cN_{E_0}$, while for $e\geq 1$, $\mathbf{s}_{E_e}:=(\mathbf{s}_{re}, \mathbf{s}_{re+1}, \dots, \mathbf{s}_{re+r-1})^\tr \in \Mat_{r\times 1}(\cN_{E_e})$ forms a $\KK[t]$-basis for $\cN_{E_e}$. 

    Recall from \eqref{Eq:Normalization_Constant} that for any $\eta\in\mathbb{K}^\times$, we set $\mathfrak{u}_\eta\in\mathbb{K}^\times$ to be a fixed choice of the $(q-1)$-st root of $\eta^{-q}$ so that we have
   \[
        \mathfrak{u}_\eta^{(-1)}=\eta\mathfrak{u}_\eta.
   \]

  For our Drinfeld module $E$, we set $\widetilde{c}= (-1)^{r-1}\kappa_r^{-1}$
   and $c= (-1)^{r-1}(\kappa_r^{(-r)})^{-1}$.  Let $\mathrm{A}_0 \in \Mat_r(\KK)$ be the diagonal matrix with $(\mathfrak{u}_{\widetilde{c}}^{-1})^{(-2)}$ in the first diagonal entry and $(\mathfrak{u}_{\widetilde{c}}^{-1})^{(-1)}$ in all other diagonal entries, while for $e\geq 1$, let $A_e = (\mathfrak{u}_{\widetilde{c}}^{-1})^{(-1)}\Id_r$. 
 Also, let
    \begin{equation}\label{E:AnotherPhie}
        \Pi:=  \begin{pmatrix}
            \kappa_1^{(-1)}&\kappa_r^{(-1)}(t-\theta)&\kappa_{r-1}^{(-1)}(t-\theta)&\dots& \dots & \kappa_2^{(-1)}(t-\theta)\\
            0 &\dots&\dots &\dots& 0 &(t-\theta)\\
            \vdots& &&\reflectbox{$\ddots$}&(t-\theta)&0\\
            \vdots&&0&\reflectbox{$\ddots$}&&\\
            0&0&(t-\theta)&&&\\
            1&0&\dots &\dots&\dots&0
        \end{pmatrix}.
    \end{equation}
    Then, for $e\geq 0$, if we set $\Phi_{E_e}:=(\mathrm{A}_e^{-1})^{(-1)}\Pi(t-\theta)^e\mathrm{A}_e$, we obtain $\sigma \mathbf{s}_{E_e} = \Phi_{E_e}\mathbf{s}_{E_e}$. 
    Consider
    \begin{equation}\label{E:Changeofbasise}
        \chi = \begin{pmatrix}
        1& 0&\dots &\dots & \dots & 0\\
        & \kappa_r^{(-1)}&\kappa_{r-1}^{(-1)} & \dots&\dots & \kappa_2^{(-1)}\\
        &&\ddots&\ddots&&\vdots\\
        &&&\ddots&\ddots&\vdots\\
        &  &&&\kappa_r^{(-r+2)}&\kappa_{r-1}^{(-r+2)}\\
        &  & &&&\kappa_r^{(-r+1)}
        \end{pmatrix}\in \GL_r(\KK). 
    \end{equation}
Since  $(\mathfrak{u}_{c}\chi)^{(-1)}\Pi (\mathfrak{u}_{c}\chi)^{-1} = \Cof(\Phi_E)$, 
we see that 
    \[
        (\mathfrak{u}_{c}\chi)^{(-1)}\Pi(t-\theta)^e (\mathfrak{u}_{c}\chi)^{-1} = \Cof(\Phi_E)(t-\theta)^e.
    \]   
Thus, if we instead pick the ordered $\KK[t]$-basis of  $\cN_{E_0}$ given by 
\[
  \mathbf{n}_{E_0}   :=   \begin{pmatrix}
1& 0&\dots &\dots & \dots & 0\\
& \kappa_r^{(-1)}&\kappa_{r-1}^{(-1)} & \dots&\dots & \kappa_2^{(-1)}\\
&&\ddots&\ddots&&\vdots\\
&&&\ddots&\ddots&\vdots\\
&  &&&\kappa_r^{(-r+2)}&\kappa_{r-1}^{(-r+2)}\\
&  & &&&\kappa_r^{(-r+1)}
\end{pmatrix}\begin{pmatrix}
   \mathfrak{u}_{c}(\mathfrak{u}_{\widetilde{c}}^{-1})^{(-2)}\sigma\mathbf{s}_{r-1}\\
    \mathfrak{u}_{c}(\mathfrak{u}_{\widetilde{c}}^{-1})^{(-1)}\mathbf{s}_{1}\\
    \vdots\\
     \vdots\\
      \mathfrak{u}_{c}(\mathfrak{u}_{\widetilde{c}}^{-1})^{(-1)}\mathbf{s}_{r-2}\\
    \mathfrak{u}_{c}(\mathfrak{u}_{\widetilde{c}}^{-1})^{(-1)}\mathbf{s}_{r-1}
\end{pmatrix} \in \Mat_{r\times 1}(\cN_{E_0}),
        \]
        while for $e\geq 1$,  we instead pick the ordered $\KK[t]$-basis of  $\cN_{E_e}$ given by
         \[
        \mathbf{n}_{E_e} := \mathfrak{u}_{c}\chi \mathrm{A}_e \, \mathbf{s}_{E_e}\in \Mat_{r\times 1}(\cN_{E_e}), 
    \]
    then we have $\sigma \mathbf{n}_{E_0} = \Cof(\Phi_E) \mathbf{n}_{E_0}$ and for $e\geq 1$, $\sigma \mathbf{n}_{E_e} = \Cof(\Phi_E)(t-\theta)^e \mathbf{n}_{E_e}$.
    In other words, for $e\geq 0$, the $t$-module $E_e$ has a $t$-frame $(\iota_{\mathbf{n}_{E_e}},\mathrm{Cof}(\Phi_E)(t-\theta)^e)$ for its dual $t$-motive $\mathcal{N}_{E_e}$.
\end{proof}

\subsection{Almost strictly pureness}
 Let $E =(\GG_a, \rho)$ be a Drinfeld module of rank $r\geq 2$ defined over $\overline{K}$.   For any $e\geq 0$, consider the $t$-module $E_{e}=\mathbf{C}^{\otimes {e}}\otimes\wedge^{r-1}E=(\mathbb{G}_a^{re+r-1},\psi_e)$ defined in Example~\ref{Ex:Exterior_Powers}. One finds by direct computation that $(\psi_0)_{t^{r-1}}$ is of the form
  \[
(\psi_0)_{t^{r-1}} = C_0 + C_1\tau + \dots + C_r\tau^r,
  \]
  where $C_r \tau^r$ is a lower triangular matrix such that for $1\leq i \leq r-1$, the $i$-th diagonal entry is $\kappa_r\tau \cdots\kappa_r\tau\cdot\underbrace{\kappa_r\tau^{2}}\cdot\kappa_r\tau \cdots\kappa_r\tau$ with $\kappa_r\tau^{2}$ in the $(r-i)$-th term of the product. Thus, $C_r$ is invertible.

  For $e\geq 1$, a direct computation also shows that 
  $(\psi_e)_{t^{re+r-1}}$ is of the form
  \[
(\psi_e)_{t^{re+r-1}} = D_0 + D_1\tau + \dots + D_r\tau^r,
  \]
  where $D_r \tau^r$ is a lower triangular matrix such that for $1\leq j \leq r$ and $0\leq u \leq e-1$, the $(ru+j)$-th diagonal entry is $\kappa_r\tau \cdots \kappa_r\tau\cdot \underbrace{\tau}\cdot\kappa_r\tau \cdots \kappa_r\tau$ with $\tau$ in the $(r-j+1)$-th term of the product. Moreover, for $1\leq i \leq r-1$, the $(re+i)$-th diagonal entry is $\kappa_r\tau \cdots \kappa_r\tau\cdot \underbrace{\tau}\cdot\kappa_r\tau \cdots \kappa_r\tau$ with $\tau$ in the $(r-i+1)$-th term of the product. Thus, $D_r$ is invertible. 
  
\begin{example}
\begin{enumerate}
 \item[(i)] {\bf The case $r=2, e=1$}:
Let $E=(\GG_a, \rho)$ be  a Drinfeld module of rank $2$. Consider the $t$-module $E_1=\mathbf{C}\otimes \wedge^{1}E = \mathbf{C}\otimes E = (\GG_a^3, \psi_1)$. Then, by direct computation, we see that 
\begin{multline*}
    (\psi_1)_{t^3} = \begin{pmatrix}
        \theta^3  & 0 & 3\theta^2 \\
        & \theta^3 & 0\\
        && \theta^3
    \end{pmatrix} \\+ \begin{pmatrix}
        \kappa_1\theta^{(1)}+2\theta\kappa_1 & -\kappa_1\theta^{(1)}-2\theta\kappa_2 & \kappa_1\\
        -(\theta^{(1)})^2-\theta\theta^{(1)}-\theta^2 & 0 & -2\theta^{(1)}-\theta\\
        \kappa_1(\theta^{(1)})^2  +  \theta\kappa_1\theta^{(1)}+\theta^2\kappa_1 &-\kappa_2(\theta^{(1)})^2-\theta\kappa_2\theta^{(1)}-\theta^2\kappa_2& 2\kappa_1\theta^{(1)}+\theta\kappa_1
    \end{pmatrix}\tau \\
   + \begin{pmatrix}
        \kappa_2 &&\\
       -\kappa_1^{(1)} &   \kappa_2^{(1)}&\\
        \kappa_2\theta^{(2)}+\kappa_2\theta^{(1)}+\theta\kappa_2+\kappa_1\kappa_1^{(1)} &-\kappa_1\kappa_2^{(1)}&\kappa_2
    \end{pmatrix}\tau^2
\end{multline*}
    \item[(ii)] {\bf The case $r=3, e=0$}:
Let $E=(\GG_a, \rho)$ be  a Drinfeld module of rank $3$. Consider the $t$-module $E_0=\wedge^{2}E =(\GG_a^2, \psi_0)$ as in Example~\ref{Ex:Exterior_Powers}. Then, by direct computation, we see that 
\begin{multline*}
    (\psi_0)_{t^2} = \theta^2 \Id_2 + \begin{pmatrix}
    -\theta\kappa_2-\kappa^2\theta^{(1)} & \theta\kappa_3+\kappa_3\theta^{(1)}\\
    -\kappa_1\theta^{(1)}-\theta\kappa_1 & 0
\end{pmatrix}\tau + \begin{pmatrix}
    \kappa_2\kappa_2^{(1)}-\kappa_3\kappa_1^{(1)} & -\kappa_2\kappa_3^{(1)} \\
    \kappa_2\kappa_2^{(1)} + \kappa_3\theta^{(2)} +\theta \kappa_3 & -\kappa_1\kappa_3^{(1)}
\end{pmatrix}\tau^2\\  + \begin{pmatrix}
    \kappa_3\kappa_3^{(1)} & 0\\
    -\kappa_3\kappa_2^{(2)} & \kappa_3\kappa_3^{(2)}
\end{pmatrix}\tau^3.
\end{multline*}
 \item[(iii)] {\bf The case $r=3, e=1$}:
Let $E=(\GG_a, \rho)$ be  a Drinfeld module of rank $3$. Consider the $t$-module $E_1=\mathbf{C}\otimes \wedge^{2}E =(\GG_a^5, \psi_1)$ as in Example~\ref{Ex:Exterior_Powers}. Then, using
\[
(\psi_1)_t = \begin{pmatrix}
    \theta &0&0 &1&0\\
   0 &\theta &0&0&1\\
    \tau &0&\theta&0&0\\
    -\kappa_2\tau & \kappa_3\tau &0&\theta&0\\
    -\kappa_1\tau &0& \kappa_3\tau &0&\theta
\end{pmatrix},
\]
by direct computation, we see that 
\begin{multline*}
    (\psi_1)_{t^2} = \begin{pmatrix}
    \theta^2-\kappa_2\tau  &\kappa_3\tau&0 &2\theta&0\\
   -\kappa_1\tau &\theta^2 &\kappa_3\tau&0&2\theta\\
    \theta^{(1)}\tau+\theta\tau & 0 & \theta^2 & \tau & 0\\
    -\kappa_2\theta^{(1)}\tau - \theta\kappa_2\tau &    \kappa_3\theta^{(1)}\tau + \theta\kappa_3\tau & 0 & \theta^2-\kappa_2\tau & \kappa_3\tau \\
       -\kappa_1\theta^{(1)}\tau - \theta\kappa_1\tau+\kappa_3\tau^2 & 0 &    \kappa_3\theta^{(1)}\tau + \theta\kappa_3\tau & -\kappa_1\tau & \theta^2
\end{pmatrix}.
\end{multline*}
Then, using 
\[
(\psi_1)_{t^5} = (\psi_1)_{t^2} (\psi_1)_{t^2} (\psi_1)_{t}, 
\]
a straightforward calculation shows that $(\psi_1)_{t^5} = D_0+D_1\tau +D_2\tau^2 + D_3\tau^3$  such that
\[
D_3 = \begin{pmatrix}
    \kappa_3\kappa_3^{(1)} &&&&\\
    -\kappa_3\kappa_2^{(2)} & \kappa_3\kappa_3^{(2)} &&&\\
    \kappa_2^{(1)}\kappa_2^{(2)}- \kappa_3^{(1)}\kappa_1^{(2)} & -\kappa_2^{(1)}\kappa_3^{(2)} & \kappa_3^{(1)} \kappa_3^{(2)} &&\\
   (D_3)_{4,1} &(D_3)_{4,2}&-\kappa_2\kappa_3^{(1)}\kappa_3^{(2)}&\kappa_3\kappa_3^{(1)}&\\
    (D_3)_{5,1} &(D_3)_{5,2}&-\kappa_1\kappa_3^{(1)}\kappa_3^{(2)}&-\kappa_3\kappa_2^{(2)}& \kappa_3\kappa_3^{(2)}
    \end{pmatrix}
\]
where
\[
(D_3)_{4,1} = -\kappa_2\kappa_2^{(1)}\kappa_2^{(2)} + \kappa_3\kappa_1^{(1)}\kappa_2^{(2)}+ \kappa_2\kappa_3^{(1)}\kappa_1^{(2)}+ \kappa_3\kappa_3^{(1)}\theta^{(3)} + \kappa_3\kappa_3^{(1)}\theta^{(2)} + 2\kappa_3\theta^{(1)}\kappa_3^{(1)}+\theta\kappa_3\kappa_3^{(1)}, 
\]
\[
(D_3)_{4,2}=-\kappa_3\kappa_1^{(1)}\kappa_3^{(2)}+\kappa_2\kappa_2^{(1)}\kappa_3^{(2)},
\]
\[
 (D_3)_{5,1}= -\kappa_3\theta^{(2)}\kappa_2^{(2)} - \kappa_3 \theta^{(1)}\kappa_2^{(2)}-\theta \kappa_3\kappa_2^{(2)} -\kappa_1\kappa_2^{(1)}\kappa_2^{(2)} + \kappa_1\kappa_3^{(1)}\kappa_1^{(2)} - \kappa_3\kappa_2^{(2)}\theta^{(3)}- \kappa_3\theta^{(2)}\kappa_2^{(2)},
\]
and 
\[
(D_3)_{5,2}= 2\kappa_3\theta^{(2)}\kappa_3^{(2)} +\kappa_3\kappa_2^{(2)}\theta^{(3)}+ \kappa_3 \theta^{(1)}\kappa_3^{(2)}+\theta \kappa_3\kappa_3^{(2)} + \kappa_1\kappa_2^{(1)}\kappa_3^{(2)}.
\]
\end{enumerate}
\end{example}

 \end{appendices}

\end{document}